\documentclass[12pt]{amsart}%
\usepackage{amsfonts}
\usepackage{epsfig}
\usepackage{graphicx}
\usepackage{amsmath}
\usepackage{amsfonts}
\usepackage{amssymb}
\usepackage{latexsym}
\usepackage{rotating}
\usepackage{pstricks, pst-node, pst-text, pst-3d,pst-eps}
\usepackage[arrow, matrix, curve, xdvi, dvips]{xy} 
\usepackage[text={5.8in,8.3in}]{geometry}
\usepackage{stmaryrd}
\usepackage{amsbsy}
\usepackage{bm}
\usepackage{dsfont}

\input{amssym.def}
\newtheorem{theorem}{Theorem}[section]
\newtheorem{proposition}[theorem]{Proposition}
\newtheorem{corollary}[theorem]{Corollary}
\newtheorem{definition}[theorem]{Definition}
\newtheorem{lemma}[theorem]{Lemma}
\newtheorem{remark}{Remark}[section]
\newtheorem{example}{Example}[section]
\newtheorem{notation}{Notation}[section]
\numberwithin{equation}{section}

\newcommand{\Id}{\text{\rm Id}}

\newcommand{\spn}{\text{\rm span}\,}

\newcommand{\GL}{\text{\rm GL}}
\newcommand{\SL}{\text{\rm SL}}

\newcommand{\Orth}{\text{\rm O}}

\newcommand{\End}{\text{\rm End}}

\newcommand{\Cl}{\text{Cl}}
\newcommand{\U}{\text{\rm U}}

\newcommand{\Sp}{\text{\rm Sp}}
\newcommand{\Spin}{\text{\rm Spin}}
\newcommand{\Pin}{\text{\rm Pin}}

\newcommand{\Aut}{\text{\rm Aut}}

\newcommand{\fg}{\mathfrak{g}}

\newcommand{\fn}{\mathfrak{n}}

\newcommand{\la}{\langle}
\newcommand{\ra}{\rangle}

\newcommand{\Q}{\mathbb{Q}}

\usepackage{color}

\usepackage{hyperref}

\begin{document}
\title[Integral structures]{Construction and classification of integral structures on pseudo $H$-type Lie algebras}
\author[K.~Furutani, I.~Markina]{Kenro Furutani and Irina Markina}

\thanks{The work of authors was partially supported by the project Pure Mathematics in Norway, funded by Trond Mohn Foundation and {Troms\o}  Research Foundation,  
by JSPS fund {No. \hspace{-0.15cm}20K03662}, and 
Osaka Metropolitan University, Central Advanced Mathematical Institute
(MEXT Joint Usage/Research Center on Mathematics and Theoretical Physics (JPMXP0619217849)).}
 

\address{K.~Furutani. Osaka Central Advanced Mathematical Institute, Osaka Metropolitan University,
Sugimoto, Sumiyoshi-ku, Osaka 558-8585, Japan}
\email{kf46089@gmail.com}

\address{I.~Markina. Department of Mathematics, University of Bergen, P.O.~Box 7803,
Bergen N-5020, Norway}
\email{irina.markina@uib.no}

\begin{abstract}
Pseudo $H$-type Lie algebras are a special class of 2-step nilpotent metric Lie algebras, intimately related to Clifford algebras $\Cl_{r,s}$. In this work, we propose the classification method for integral orthonormal structures of pseudo $H$-type Lie algebras for a full range of parameters $(r,s)\in\mathbb Z^2_+$, $\mathbb Z_+=\{0,1,2,\ldots\}$. The existence of integral orthonormal structures gives rise to the integral discrete uniform subgroups or lattices of the pseudo $H$-type Lie groups. We apply the developed method for fully classifying the integral orthonormal structures for $0<r+s\leq 16$, and minimal admissible Clifford modules. The cases $0<r+s\leq 16$ form a core for further extensions by using the Atiyah-Bott periodicity and the reducibility of admissible Clifford modules. 
\end{abstract}
\maketitle

\tableofcontents


\section{Introduction}


Two-step nilpotent Lie algebras attracted the attention of G.~M\'etivier~\cite{MR563376} in an attempt to describe hypoelliptic operators in a non-Euclidean setting. The condition of hypo-ellipticity required the adjoint map on the Lie algebra with the value on the centre to be surjective. This type of Lie algebras was studied under different names and for different purposes, for instance, in~\cite{MR1250818,MR1683874,MR2063040,MR2943073,MR3790508}. A.~Kaplan~\cite{MR554324} showed that if the adjoint map is an isometry, then the sub-Laplacian on two-step nilpotent Lie groups, admits a fundamental solution, reminiscent of that in Euclidean space. His result extended a theorem obtained by G.~Folland on the Heisenberg group~\cite{MR315267}. Therefore, the class of these Lie algebras received the name $H$(eisenberg)-type Lie algebras. The $H$-type Lie algebras are in a bijective relation to Clifford algebras $\Cl_{r,0}$, generated by the Euclidean space $\mathbb R^r$~\cite{MR2075359}. The definition of $H$-type Lie algebras related to Clifford algebras $\Cl_{r,s}$, $s>0$, generated by pseudo Euclidean spaces $\mathbb R^{r,s}$ was extended by P.~Ciatti~\cite{Ciatti00} and received the name pseudo $H$-type Lie algebras, see also~\cite{GKM13}. The pseudo $H$-type Lie algebras, which will be denoted by $\mathfrak n_{r,s}$ is a fruitful source for studies of Damek-Ricci spaces~\cite{MR1340192}, Iwasawa decomposition of symmetric spaces~\cite{MR1705176}, Riemannian nilmanifolds~\cite{MR621376}, rigidity problems~\cite{MR1854087}, properties of PDE on Lie groups~\cite{MR3021808,MR1176678,MR4093612} and many other topics in geometry, analysis, and geometric measure theory. The classification of the pseudo $H$-type Lie algebras was completed in~\cite{ FurMar17, FurMar19}.

Our work is motivated by the study of uniform discrete subgroups on nilpotent Lie groups, which are crucial for the study of homogeneous spaces, compact nilmanifolds, and spectral problems. The existence of a uniform subgroup is guaranteed by a presence of a rational structure on the associated Lie algebra by a seminal work of A.~I.~Mal\v{c}ev~\cite{MR0028842}. The existence of rational structures on pseudo $H$-type Lie algebras was proved in~\cite{MR1885037, MR2015754, FurMar14}. A complete classification of rational structures in the class of pseudo $H$-type Lie algebras exists only on the Heisenberg algebra (related to the Clifford algebra $\Cl_{1,0}$)~\cite{MR837583,MR487050}. Some progress in the study of lattices can be found in~\cite{MR2396683}. 

In the present work, we describe the set of invariant orthonormal integral structures that lie at the core of the rational structures of Lie algebras. An invariant integral structure is a $\mathbb Z$‑span of an orthonormal basis constructed via the action of a subgroup $G(\mathfrak B_{r,s})$ of the invertible elements $\Pin(r,s)$ in the Clifford algebra $\Cl_{r,s}$ on a suitably chosen normal vector $v \in V$ in the Clifford module $V$; see Section~\ref{sec:invariant basis}. As a result, the basis of the Clifford module $V$ is invariant under the action of $G(\mathfrak B_{r,s})$, and the non-vanishing structure constants of the pseudo $H$-type Lie algebra are equal to $\pm 1$.
We emphasize that invariant orthonormal integral structures are particular cases of integral structures (having structure constants ${\color{violet}0}, \pm 1$) that are included in a general class of rational structures on a Lie algebra (having rational structure constants). Two invariant integral structures are isomorphic, if and only if the isotropy subgroups $\mathcal S^{(1)}_v\subset \Cl_{r,s}$ and 
$\mathcal S^{(2)}_v\subset \Cl_{r,s}$ of $v\in V$ belongs to the same equivalence class, see Definition~\ref{def:group equiv} and Section~\ref{sec:uniform-subgroups}.
The isomorphism of invariant integral structures of the Lie algebras leads to the
isomorphism of uniform discrete subgroups on the corresponding Lie groups, which is always extended to an automorphism of ambient pseudo $H$-type Lie groups, see~\cite{Raghunathan72}. 

We apply the classification algorithm to the isotropy groups $\mathcal S_v$ for the parameters $0<r+s\leq 16$ in Section~\ref{sec:nonisomorphic subgroups}. We note that the range $0<r+s\leq 16$ corresponds to the first and second periods in $r$ of pseudo $H$-type Lie groups originating from the Atiyah–Bott periodicity of Clifford algebras. The reader may notice that the second period $r\in{9,\ldots,16}$ contains more non-equivalent subgroups, with phenomena such as disconnectedness, which cannot appear in the first period $r\in{3,\ldots,8}$ due to the insufficient dimension of the center of the Lie algebra. The subsequent periods for $r\geq 17$ may contain new phenomena, which are not known to us at the present moment. Therefore, in the present work, we restrict ourselves to the first two periods. We emphasize that our algorithm can be inductively extended to any values of the parameters $(r,s)$, but it may not produce the full range of possible non-equivalent sets of involutions. As the dimension $r+s$ increases, the invariants introduced here, such as the $T1$- and $T2$-types or connectedness, may no longer be sufficient to describe the complete set of invariant integral structures.

A forthcoming paper will be dedicated to studying the new features that arise for increasing parameters $r,s$ and to further analysis of the periodicity.

Despite this, the theorems and the characterizations proved in Sections~\ref{sec:invariant basis} and~\ref{sec:uniform-subgroups} have general character and
are valid for arbitrary parameters $(r,s)$. 



\section{Preliminaries}


In this section we remind some classical objects and introduce the main ones of our interest.


\subsection{Clifford algebras}


We denote by $\mathbb{R}^{r,s}$ the pseudo Euclidean space, that is the vector space $\mathbb{R}^{r+s}$ endowed with the non-degenerate symmetric bilinear form 
$$
\la x,y\ra_{r,s}=\sum_{k=1}^{r}x_ky_k-\sum_{k=r+1}^{r+s}x_ky_k,\quad
x=(x_1,\ldots,x_{r+s}),\ y=(y_1,\ldots,y_{r+s}).
$$ 
Let $\Cl_{r,s}$ be a Clifford algebra over $\mathbb{R}$ generated by 
$\mathbb{R}^{r,s}$. Remind that $\Cl_{r,s}$ 
is a quotient of the tensor algebra 
\[\mathcal{T}(U):
=
\mathbb{R}\oplus \mathbb{R}^{r,s}\oplus\Big(\stackrel{2}{\otimes}\mathbb{R}^{r,s}\Big) \oplus\Big(\stackrel{3}{\otimes}\mathbb{R}^{r,s}\Big) \oplus 
\Big(\stackrel{4}{\otimes}\mathbb{R}^{r,s}\Big)
\oplus\ldots
\] 
by a two sided ideal $I_{r,s}$ generated by elements of the form
\[
 x\otimes x+\la x,x\ra_{r,s}{\mathbf 1},\quad x\in \mathbb{R}^{r+s},
\]
and ${\mathbf 1}$ is the identity element of the Clifford algebra $\Cl_{r,s}$.
Consider a representation of $\Cl_{r,s}$ on a real vector space $V$
\begin{equation*}
J\colon \Cl_{r,s}\to \End(V).
\end{equation*}
We call $V$ the $\Cl_{r,s}$-module, or simply module if we do not want to specify the signature $(r,s)$, and will denote by $J_zv$ the action of $z\in \mathbb R^{r,s}$ on $v\in V$.
Assume also that the module $V$ is equipped with a non-degenerate symmetric
bilinear form $\la.\,,.\ra_{V}$ satisfying the condition 
\begin{equation}\label{eq:admissible module}
\la J_{z}u,v\ra_{V}+\la u,J_{z}v\ra_{V}=0\quad \text{for any}\quad z\in \mathbb{R}^{r,s}\quad \text{and} \quad u,v\in V.
\end{equation}
We call such a module $V=(V,\la.\,,.\ra_{V})$ an {\it admissible module} of the Clifford algebra~$\Cl_{r,s}$.
We write $V_{min}=(V_{min},\la.\,,.\ra_{V})$ or simply $V_{min}$ for 
an admissible $\Cl_{r,s}$-module of the minimal
dimension and call it a {\it minimal admissible module}. The reader can find more about analogous constructions of 2 step nilpotent Lie algebras, not related to representation of Clifford algebras in~\cite{MR2090766}.

We emphasize the difference between an irreducible Clifford module and a minimal admissible module. Not all  irreducible modules can be equipped with a non-degenerate bilinear symmetric form, satisfying~\eqref{eq:admissible module}. For instance, lack of dimension of an irreducible module can make any bilinear symmetric form degenerate. An admissible module $V$ of $\Cl_{r,s}$ has an even dimension $\dim(V)=2n=N$. It is isometric to $\mathbb R^{n,n}$ if $s>0$ and it is isometric to $\mathbb R^{\pm N,0}$ if $s=0$, see~\cite[Theorem 3.1]{Ciatti00} and~\cite[Proposition 1]{FurMar17}. Any admissible $\Cl_{r,s}$-module can be decomposed into an orthogonal direct sum of minimal admissible modules~\cite[Proposition 2.3 (2)]{FurMar19}.


\subsection{Pseudo $H$-type Lie algebras and Lie groups}


\begin{definition}\label{pseudo H type Lie algebras and groups}
Let $(V,\la.\,,.\ra_{V})$ be an admissible module of a Clifford algebra $\Cl_{r,s}$ with the representation map $J$. Define the Lie bracket on $V\times\mathbb R^{r,s}$ by
\begin{equation}\label{eq:Lie bracket}
\la J_{z}u,v\ra_{V}= \la z,[u,\,v]\ra_{r,s},\quad z\in \mathbb{R}^{r,s},\quad u,v\in V. 
\end{equation}
The pseudo $H$-type Lie algebra $\mathfrak{n}_{r,s}(V)=(V\oplus\mathbb{R}^{r,s},[.\,,.])$
is a Lie algebra whose non-vanishing Lie bracket is defined  in~\eqref{eq:Lie bracket}.
\end{definition} 
Note that the Lie algebra $\mathfrak{n}_{r,s}(V)$ is 2-step nilpotent where $\mathbb R^{r,s}$ is the centre. Property~\eqref{eq:admissible module} and the representation property $J_{z}^2v=-\la z,z\ra_{r,s}v$ for $v\in V$ imply
\begin{equation}\label{eq:isometry}
\la J_zu,J_zv\ra_{r,s}=\la z,z\ra_{r,s}\la u,v\ra_{V},\quad \la J_zu,J_wu\ra_{r,s}=\la z,w\ra_{r,s}\la u,u\ra_{V}.
\end{equation}

The connected simply connected Lie group $\mathbb{N}_{r,s}(V)$ of the Lie algebra $\mathfrak{n}_{r,s}(V)$ is called the pseudo $H$-type Lie group. The exponential map $\exp\colon \mathfrak{n}_{r,s}(V)\to \mathbb{N}_{r,s}(V)$ is a global analytic diffeomorphism~\cite[Theorem 1.2.1]{corwin1990representations}. It allows to induce the coordinates on the Lie group from the Lie algebra by means of Baker-Campbell-Hausdorff formula. Thus for $g\in \mathbb{N}_{r,s}(V)=V\times \mathbb{R}^{r,s}$ we set $g= u\oplus z\in V\oplus\mathbb{R}^{r,s}$. The group product $\ast$ on $\mathbb{N}_{r,s}(V)$ is given by
\begin{align*}
\ast\colon \mathbb{N}_{r,s}(V)&\times \mathbb{N}_{r,s}(V) \to \mathbb{N}_{r,s}(V),
\\
(u_1,z_1)&\ast (u_2,z_2)=\Big(u_1+u_2, \ z_1+z_2+\frac{1}{2}[u_1,u_2]\Big).
 \end{align*}


\subsection{Automorphisms of pseudo $H$-type Lie algebras}


Since automorphisms of a Lie algebra define the automorphisms of 
its connected simply connected Lie group, we consider only the automorphisms
of Lie algebras. The complete description of the group of automorphisms 
of pseudo $H$-type Lie algebras can be found in~\cite{Riehm82,MR1490017,FurMar21}, see also~\cite{MR3218266}.

The automorphisms of pseudo $H$-type Lie algebras are generated by the following ones:

\quad [1] 
The dilations $\delta_{\lambda}(u,z)=(\lambda u, \lambda^2 z)$.

\quad [2] Let $A\colon V\to V$ be a nonsingular linear map and $C\in \Orth(r,s)$ an orthogonal transformation of $\mathbb R^{r,s}$.
Then the map $A\oplus C$ is a pseudo $H$-type Lie algebra automorphism, if and only if
\begin{equation}\label{automorphism condition:eq}
A^{\tau}\circ J_{z}\circ A=J_{C^{\tau}(z)},\quad z\in\mathbb{R}^{r,s},
\end{equation}
where $A^{\tau}$, $C^{\tau}$ are transpose maps defined as
\[
\la A^{\tau}u,v\ra_{V}=\la u,Av\ra_{V},\quad \la C^{\tau}z,w\ra_{r,s}=\la z,Cw\ra_{r,s}.
\]

\quad [3] Let $B\colon V\to \mathbb{R}^{r,s}$ be a linear map. Then
$(v,z)\mapsto \big(v,z+Bv\big)$ is an automorphism.


\subsection{Rational structures, uniform discrete subgroups, lattices}


We refer to works~\cite{Raghunathan72,corwin1990representations} 
for the details discussed in this section.

\begin{definition}\label{def:lattice} A Lie algebra $\mathfrak g_{\mathbb Q}$ over rational numbers $\mathbb Q$ is called the rational structure of a real Lie algebra $\mathfrak g$ if 
$\mathfrak g$ is isomorphic to $\mathfrak g_{\mathbb Q}\otimes\mathbb R$.
\end{definition}

A real Lie algebra $\mathfrak g$ has a rational structure if and only if there is a basis for $\mathfrak g$ such that the structure constants of the Lie algebra are rational numbers.

\begin{definition}
Let $G$ be a Lie group. A subgroup $\Gamma$ is called uniform subgroup if $\Gamma$ is discrete and $G/\Gamma$ is a compact space.
\end{definition}

\begin{definition}
Let $G$ be a Lie group with the Haar measure $\mu$. A subgroup $\Lambda$ is called lattice if $\mu(G/\Lambda)<\infty$.
\end{definition}

Let $G$ be a nilpotent Lie group and $\mu$ the Haar measure on it. Then a discrete subgroup $\Gamma$ is lattice if and only if it is a uniform subgroup, i.e $\mu(G/\Gamma)<\infty$ implies that $G/\Gamma$ is compact. From now on we will not distinguish the lattices and uniform subgroups. A result from~\cite{MR0028842} can be formulated as follows.
\begin{itemize} 
\item If $\Gamma$ is a uniform subgroup of $G$, then $\fg$ has a rational structure $\fg_{\Q}$ such that $\fg_{\Q}=\spn_{\Q} \{\log$\,$(\Gamma)\}$.
\item
If $\fg$ has a rational structure $\fg_{\Q}$, then $G$ has a uniform subgroup $\Gamma$ such that $\log(\Gamma)\subseteq\fg_{\Q}$.
\end{itemize}

\begin{theorem}\cite{Raghunathan72}
Let $\Gamma_{i}\subset G_{i}$, $i=1,2$,
be uniform subgroups of simply connected nilpotent Lie groups $G_{i}$.
An isomorphism $\varphi\colon \Gamma_{1}\to \Gamma_{2}$ of discrete subgroups,
can be extended to a smooth isomorphism  
$\tilde{\varphi}\colon G_{1}\to G_{2}$ of the Lie groups.
\end{theorem}


\section{Invariant bases and groups of positive involutions}\label{sec:invariant basis}



\subsection{Definition of an invariant integral structure and uniform subgroups}\label{sec:invariant basis-definition}
%

From now on we will consider only {\bf minimal admissible modules} of Clifford algebras $\Cl_{r,s}$, denoting them either by $V^{r,s}$ or simply by $V$. Let $\mathfrak n_{r,s}(V)=(V\oplus\mathbb{R}^{r,s},[.\,,.])$ be a pseudo $H$-type Lie algebra with $\mathfrak B_{r,s}$ a basis for $\mathbb{R}^{r,s}$ and $\mathfrak B(V)$ a basis for $V$. We write the structure constants $c_{ij}^l$ for $\mathfrak n_{r,s}(V)$ with respect to bases $\mathfrak B(V)$ and $\mathfrak B_{r,s}$ by 
\begin{equation}\label{eq:structure constants}
[v_i,v_j]=\sum_{l=1}^{r+s}c_{ij}^lz_l.
\end{equation}
\begin{definition}\label{def:standard basis}
A basis $\{\mathfrak B(V),\mathfrak B_{r,s}\}$ for $\mathfrak n_{r,s}(V)$ is called integral if the structure constants $c_{ij}^l$ in~\eqref{eq:structure constants} take the values in $\{-1,0,1\}$. 
\end{definition}

We want to study a special class of integral bases of $\mathfrak n_{r,s}(V)$. To describe it, 
we fix an orthonormal basis $\mathfrak B_{r,s}=\{z_1,\ldots, z_r,z_{r+1},\ldots,z_{r+s}\}$ of $\mathbb R^{r,s}$, where
\begin{equation}\label{eq:brs}
\begin{cases}
&z_1,\ldots,z_r\quad \text{are positive, i.e.,}~\la z_{i},\,z_{i}\ra_{r,s}=1,\ \ i=1,\ldots,r,
\\	
&z_{r+1},\ldots,z_{r+s}~\text{are negative, i.e.,}~\la z_{i},\,z_{i}\ra_{r,s}=-1,\ \ j=r+1,\ldots,r+s.
\end{cases}
\end{equation}
The group $\Pin(r,s)$ consists of elements of the Clifford algebra $\Cl_{r,s}$ of the form 
\begin{equation}\label{eqPin}
\sigma=x_{i_{1}}\cdots x_{i_{k}}, \quad \la x_{i_{j}},x_{i_{j}}\ra_{r,s}=\pm 1,
\end{equation}
where $x_{i_{k}}\neq x_{i_{l}}$ for $i_{k}\neq i_l$.
The subgroup $\Spin(r,s)\subset\Pin(r,s)$ is generated by the even number of elements in~\eqref{eqPin}. 
Consider a finite subgroup $G(\mathfrak B_{r,s})$ of the Pin group $\Pin(r,s)$ defined by
\begin{equation*}\label{group generated by basis}
\begin{array}{lll}
G(\mathfrak B_{r,s})=\big\{&\pm {\mathbf 1},\ \pm z_1,\ \ldots,\ \pm z_{r+s},\  \ldots,\ \pm z_{i_1}\cdots z_{i_k}\mid\ 
\\
&1\leq i_1<\cdots<i_k\leq r+s, \quad k=2,\ldots, r+s\big\}.
\end{array}
 \end{equation*}
Thus the generators of the group $G(\mathfrak B_{r,s})$ are $\{-{\mathbf 1},\mathfrak B_{r,s}\}$.
Elements $\sigma \in G(\mathfrak B_{r,s})$ satisfy the properties: either $\sigma^2={\mathbf 1}$ or $\sigma^2=-{\mathbf 1}$.   

We proceed to the construction of bases $\mathfrak B(V^{r,s})$ for the minimal admissible module $V^{r,s}$. In Table~\ref{t:dim} the reader finds dimensions of $V^{r,s}$ for $0\leq r,s\leq 8$, which are extended by periodicity. Table~\ref{t:dim} summarises results from~\cite[Theorem 3.1]{Ciatti00} and~\cite[Proposition 1]{FurMar17}. We marked by red color the Clifford algebras, where the minimal admissible modules differ from the irreducible modules. With the subscript \hspace{-0.08cm}$\ _{\times 2}$ we indicated the presence of two non-equivalent minimal admissible modules of the same dimension. 
Here the equivalence we understand as equivalence of representations of the Clifford algebra on the corresponding modules.

\begin{table}[h]
\center\caption{Dimensions of minimal admissible modules}
\scalebox{0.7}[0.7]{
\begin{tabular}{|c||c|c|c|c|c|c|c|c|c|}
\hline
${\text{\small 8}} $&$ {\text{\small{16}}}$&${\text{\small 32}}$&${\text{\small{64}}}
$&${\text{\small{64}$_{\times 2}$}}$&${\text{\small{128}}}$&${\text{\small{128}}}$&${\text{{\small{128}}}}$&$
{\text{{\small{128}$_{\times 2}$}}} $&${\text{\small{256}}}$
\\
\hline
${\text{\small 7}}$ &$ {\text{\small{16}}}$&${\text{\small{32}}}$&$
{\text{\small{\color{red}{64}}}} $&${\text{\small{64}}}
$&${\text{{\color{red}\small{128}}}}$&${\text{\small{{\color{red}{128}}}}}  $&${\text{\color{red}{\small{128}}}}$&$ {\text{\small{128}}} $&${\text{\small{256}}}$
\\
\hline
${\text{\small 6}}$ &${\text{\small{16}}}$&${\text{\small{16}$_{\times 2}$}}$&${\text{\small{32}}}$&${\text{\small{32}}}$&${\text{\small{\color{red}{64}}}}
$&${\text{\color{red}{\small{64}$_{\times 2}$}}} $&${\text{\color{red}\small{128}}} $&${\text{\small{128}}}$&$ {\text{\small{256}}} $
\\
\hline
${\text{\small 5}} $&${\color{red}\text{\small 16}}$&${\text{\small 16}}$&${\text{\small 16}}$&${\text{\small 16}}$&${\color{red}{\text{\small 32}}}$&${\color{red}{\text{\small 64}}} $&${\text{\small{\color{red}128}}}$&${\text{\small{128}}} $&$\text{\small{\color{red}256}}$
\\
\hline
${\text{\small 4}} $&$  {\text{\small 8}}$&$ {\text{\small 8}}$&$
{\text{\small 8}}$&$ 8_{\times 2}$&$16$&${\text{\small 32}}$&${\text{\small 64}}
$&${\text{\small 64}_{\times 2}} $&${\text{\small{128}}}$
\\
\hline
${\text{\small 3}}$&${\color{red}{\text{\small 8}}}$&${\color{red}{\text{\small 8}}}$&${\text{\small\color{red}8}}$&$8$&$16$&$32$
&${\text{\small\color{red}64}}$&$64$&${\color{red}{\text{\small 128}}}$
\\
\hline
${\text{\small 2}}$&${\color{red}{\text{\small 4}}}$&$
{\color{red}4_{\times 2}}$&${\color{red}8}$&$ 8$&$16$&$16_{\times 2}$&$32$&$32 $&${\color{red}{\text{\small 64}}}$
\\
\hline
${\text{\small 1}}$ &${\color{red}2}$&${\color{red}4}$&${\color{red}8}$& $8$&${\color{red}16}$&$16$&$16$&$16$&${\color{red}{\text{\small 32}}}$
\\
\hline
${\text{\small 0}} $&$  1$&$ 2$&$ 4$&$ 4_{\times 2}$&$ 8$&$ 8$&$ 8$&$ 8_{\times 2}$&$16$
\\
\hline\hline
{s/r}&  {\text{\small 0}}& {\text{\small 1}}& 
{\text{\small 2}}&{\text{\small 3}} & {\text{\small 4}}& {\text{\small 5}}& {\text{\small 6}}& {\text{\small 7}}& {\text{\small 8}}
\\
\hline
\end{tabular}\label{t:dim}
}
\end{table}

\begin{definition}\label{invariant basis}
Fix an orthonormal basis $\mathfrak B_{r,s}$ of $\mathbb R^{r,s}$. An orthonormal basis $\mathfrak{B}(V^{r,s})$ of a minimal admissible module 
$V^{r,s}$ is called invariant basis if it is invariant under the action of 
$G(\mathfrak B_{r,s})$; that is for any $v_i\in\mathfrak{B}(V^{r,s})$ and $z_j\in \mathfrak B_{r,s}$, there exists $v_k\in \mathfrak{B}(V^{r,s})$
such that
$J_{z_j}v_i=v_{k}$ or $J_{z_j}v_i=-v_k$.
\end{definition}
According to Definition~\ref{invariant basis} the maps $J_{z_j}$, $z_j\in \mathfrak B_{r,s}$ act on an invariant basis $\mathfrak{B}(V^{r,s})$ by permutations up to the sign $\pm$.

\begin{remark}
We emphasize that Definition~\ref{invariant basis} requires bases $\mathfrak{B}(V^{r,s})$ to be both orthonormal and invariant. 

{\sc Example A.} Consider the Heisenberg Lie algebra $\fn_{1,0}(V)$ with the normalized basis $\mathfrak B_{1,0}=\{z\}$ for the center and $V^{1,0}$ isometric to $\mathbb R^{2,0}$. Set $v_1\in \mathbb R^{2,0}$, $v_1\neq 0$, $v_2=J_{z}v_1$, and  
$$
u_1=Av_1,\quad u_2=Av_2,
$$
where $A$ is an orthogonal transformation of $V^{1,0}$. Then the basis $\mathfrak B(V^{1,0})=\{v_1,v_2\}$ is orthonormal and invariant, meanwhile the basis $\mathfrak B(V^{1,0})=\{u_1,u_2\}$ is orthonormal, but not necessary invariant. It is invariant under the action of $G(\mathfrak B_{1,0})$ if and only if $J_z$ commutes with $A$. 

{\sc Example B.} Consider $\fn_{0,3}(V)$ with an orthonormal basis $\mathfrak B_{0,3}=\{z_1,z_2,z_3\}$ for the center and a ${\rm Cl}_{\,0,3}$-minimal admissible module $V^{0,3}$ isometric to $\mathbb{R}^{4,4}$. We take $v\in \mathbb{R}^{4,4}$, such that $\la v,v\ra_{\mathbb{R}^{4,4}}=1$. 
The eight vectors
\begin{equation}\label{basis of C03}
v,~J_{z_1}v,~J_{z_2}v,~J_{z_3}v,~J_{z_1}J_{z_2}v,
~J_{z_1}J_{z_3}v,~J_{z_2}J_{z_3}v,~J_{z_1}J_{z_2}J_{z_3}v
\end{equation}
are linearly independent, have square of the norm equal to $\pm 1$, and invariant under the action of $G(\mathfrak B_{0,3})$. Nevertheless, the basis~\eqref{basis of C03} is not necessary orthogonal, since the value 
$\la v,J_{z_1} J_{z_2} J_{z_3}v\ra_{\mathbb{R}^{4,4}}$ depends on the choice of $v\in \mathbb{R}^{4,4}$, see~\cite[Lemmas 2.8, 2.9]{FurMar14}. 
\end{remark}

\begin{proposition}\label{integrality}
Let $\mathfrak{B}(V^{r,s})$ be an invariant basis. Then it is an integral basis.
\end{proposition}
\begin{proof} 
Assume that $\mathfrak{B}(V^{r,s})$ is an invariant basis for $V^{r,s}$ and that $J_{z_{\ell}}v_i=\pm v_{k}$ for $v_i,v_k\in \mathfrak{B}(V^{r,s})$. Then, by definition of the Lie bracket~\eqref{eq:Lie bracket}, we obtain
\[ 
\langle z_{\ell}, [v_{i},v_{j}]\rangle_{r,s}=\langle J_{z_{\ell}}v_{i}, v_{j}\rangle_{V^{r,s}}
=\langle \pm v_{k}, v_{j}\rangle_{V^{r,s}}=\pm \delta_{kj}.
\]
If $k=j$, then the orthonormality of $\mathfrak B_{r,s}$ and $\langle z_{\ell}, [v_{i},v_{j}]\rangle_{r,s}=\pm 1$ imply that $[v_{i},v_{j}]=\pm z_{\ell}$, and the structure constants in~\eqref{eq:structure constants} are such that $c_{ij}^{\ell}=\pm 1$. If $k\neq j$ then $c_{ij}^{\ell}=0$.
\end{proof}

The definition of an invariant basis leads to the definition of an invariant integral structure on pseudo $H$-type Lie algebras and (invariant) integral uniform subgroup on the respective pseudo $H$-type Lie groups.

\begin{definition}\label{def: integral lattice}
Let $\mathfrak B_{r,s}=\{z_k\}_{k=1}^{r+s}$ be an orthonormal basis for $\mathbb R^{r,s}$ and $\mathfrak B(V^{r,s})=\{v_i\}_{i=1}^{N}$ an invariant basis for a minimal admissible module $V^{r,s}$. An invariant integral structure on the pseudo $H$-type Lie algebra $\mathfrak{n}_{r,s}(V)$ is the vector space over $\mathbb{Z}$ given by
\[
\spn_{\mathbb{Z}}\{\mathfrak B(V^{r,s})\}\oplus \spn_{\mathbb{Z}}\{\mathfrak B_{r,s}\}
=\Big\{\sum_{i=1}^{N}n_i v_{i}\oplus \sum_{k=1}^{r+s}m_kz_{k}
~\Big| ~n_i,m_k\in\mathbb{Z}\Big\}.
\]
An (invariant) integral uniform subgroup on the pseudo $H$-type Lie group $\mathbb{N}_{r,s}(V)=\{(v,z)\mid v\in V^{r,s}, z\in\mathbb{R}^{r,s}\}$ is given by the coordinates
\[
\left(\Big(\sum_{i=1}^Nn_iv_{i}\mid n_i\in\mathbb{Z}\Big),\ \  
\Big(\frac{1}{2}\sum_{k=1}^{r+s}m_kz_{k}\mid m_k\in\mathbb{Z}\Big)\right).
\]
\end{definition}

The main goal of the present work is the classification of invariant integral structures on pseudo $H$-type Lie algebras, which leads to a classification of integral uniform subgroups of the corresponding pseudo $H$-type Lie groups. Note that invariant integral structures form a subclass of integral structures (not necessarily invariant and/or orthonormal) on pseudo $H$-type Lie algebras. In the present work, we take a first step and classify only the invariant orthonormal integral structures. The classification of general integral structures and, more broadly, rational structures is postponed to future work. We illustrate our approach for the first two periods in the Atiyah–Bott periodicity, corresponding to the range of dimensions $r+s\leq 16$.
In the articles~\cite{MR837583, MR487050}, the authors classified the rational uniform subgroups of the Heisenberg groups, where the starting point was a unique invariant integral basis of the Heisenberg algebra. Thus, in essence, we take a first step toward a full classification of rational structures on two-step nilpotent Lie algebras related to Clifford algebras.


\subsection{Subgroups $\mathcal S\subset G(\mathfrak B_{r,s})$ of positive involutions}\label{sec:construction-InvarBeses}


In the present section we study subgroups $\mathcal S$ of $G(\mathfrak B_{r,s})\subset \Pin(r,s)$ which will be a core for the construction of invariant bases $\mathfrak B(V^{r,s})$. Recall that $G(\mathfrak B_{r,s})$ is generated by $\{-\mathbf 1,\mathfrak B_{r,s}\}$. Some of the properties of $\mathcal S$ can be learned from the definition of the subgroups $\mathcal S$, but some of them became clear by  considering their action on minimal admissible modules $V^{r,s}$. The representation map $J\colon \mathcal S\to \End(V^{r,s})$ is not injective for $r-s=3\mod 4$. 
We concentrate on the module, where
the map $J\colon \mathcal S\to \End(V^{r,s})$ is injective and the proofs for another non-equivalent module is modified accordingly.

\begin{definition}\label{Def:S}
We denote by $\mathcal{S}$ a subgroup of $G(\mathfrak B_{r,s})$ satisfying the conditions
\begin{itemize}
\item[$(S1)$] {$ -{\mathbf 1} \notin\mathcal S$;}
\item[$(S2)$] {if $p\in\mathcal S$, then $p\in \Pin(r,0)\times\Spin(0,s)$;}
\item[$(S3)$] {if $p\in\mathcal S$, then $p^2={\mathbf 1}$.}
\end{itemize}
An elements $p\in \mathcal S$ is called a positive involution.
\end{definition}

The name {\it positive involution} refers to the action of $p\in\mathcal S$ on $V^{r,s}$ as follows 
$$
\text{if}\ \la v,v\ra_{V^{r,s}}>0\ \text{ then }\ \la J_pv,J_pv\ra_{V^{r,s}}>0,
$$
and 
$$
\text{if}\ \la v,v\ra_{V^{r,s}}<0,\ \text{ then }\ \la J_pv,J_pv\ra_{V^{r,s}}<0.
$$ 
We denote by $\mathbb{S}_{r,s}$ 
(or just $\mathbb{S}$), the set of all subgroups of $G(\mathfrak B_{r,s})$ satisfying Definition~\ref{Def:S}.
This set is a partially  ordered set with respect to the inclusion relation
among subsets. The following property is obvious.

\begin{lemma}
The groups $\mathcal{S}\in\mathbb{S}_{r,s}$ are  necessarily commutative.
\end{lemma}

\begin{example}\label{example 1}
Consider $G(\mathfrak B_{4,0})$. 
Then some examples of subgroups $\mathcal S\in\mathbb{S}_{r,s}$ are
$$
\mathcal S_1=\{ {\mathbf 1}, z_1z_2z_3\},\  \mathcal S_2=\{ {\mathbf 1}, z_1z_2z_4\},\ \mathcal S_3=\{ {\mathbf 1}, z_1z_3z_4\},\ \mathcal S_4=\{{\mathbf 1}, -z_1z_2z_4\}
$$
and
$$
\mathcal S_5=\{{\mathbf 1}, z_1z_2z_3z_4\}.
$$
The first four groups are isomorphic under the action of the orthogonal group $\Orth(4)$ and they are not isomorphic to the last one. The action of $A\in \Orth(4)$ is defined by $A(z_iz_jz_k)=A(z_i)A(z_j)A(z_k)$. This isomorphism is crucial in the use of subgroups $\mathcal S\in\mathbb{S}_{r,s}$ in the classification of invariant bases $\mathfrak B(V^{4,0})$. 
\end{example}

The groups $\mathcal S_2$ and $\mathcal S_4$ in the above example differs only by sign. It leads to the following definition.

\begin{definition}\label{def:hat-S}
Let $\mathcal S$ be a group satisfying Definition~\ref{Def:S}. We denote by $\widehat{\mathcal S}\subset G(\mathfrak B_{r,s})$ the extended group 
$$
\widehat{\mathcal S}=\mathcal S\cup\{-\sigma:\ \sigma\in\mathcal S\}.
$$
\end{definition}

In Example~\ref{example 1} $\mathcal S_2,\mathcal S_4$ are subgroups of $G(B_{4,0})$, where we fix the basis $\{z_1,z_2,z_3,z_4\}$. The subgroups  $\mathcal S_2,\mathcal S_4$ are different, nevertheless
$$
\widehat{\mathcal S_4}=\widehat{\mathcal S_2}=\{\pm {\bf 1},\pm z_1z_2z_4\}.
$$


\subsection{Generators for $\mathcal S$}\label{sec:generators}


In this section, we describe the groups $\mathcal S\in\mathbb{S}$ by using their generating sets.

\begin{definition}\label{def:PI}
We denote by $PI=\{\,p_{i}\,\}_{i=1}^{\ell}$, where $\ell=\#[PI]$ is the cardinality of the set $PI$, a subset in $G(\mathfrak B_{r,s})$ satisfying the conditions:
\begin{itemize}
\item [$(PI1)$] \quad ${\mathbf 1}\notin PI$, $p_{i}p_{j}=p_{j}p_{i}$, and $p_{i}\in PI$ satisfy $(S2)-(S3)$ in Definition~\ref{Def:S} for all $i=1,\ldots,\ell$.
\item [$(PI2)$] \quad The vectors 
\begin{equation}\label{all products of p} 
\{\, {\mathbf 1}, p_{1},\ \ldots,\ p_{\ell},\ \ 
p_{i_{1}}\cdots p_{i_{k}}\mid \ 1\leq i_{1}<\ldots<i_{k}\leq \ell,\ k=2,\ldots,\ell\}
\end{equation}
\quad are linearly independent in the vector space {\em $\Cl_{\,r,s}$}.
\end{itemize}
\end{definition}
  
\begin{proposition}
Condition $(PI2)$ is equivalent to the condition
\begin{itemize}
\item [$(PI2)'$] 
\quad none of the products $p_{i_1}\cdots p_{i_k}$, 
$1\leq i_1<\cdots< i_k \leq \ell$, $k=1,\ldots,\ell$, is equal to $\pm\mathbf 1$.
\end{itemize}
\end{proposition}

\begin{proof}
The elements
\begin{equation}\label{eq:Clifford basis}
\{\epsilon_0\mathbf 1,\  \epsilon_{i_1,\ldots,i_k}
z_{i_{1}}\cdots z_{i_{k}}\}\ \subset \Cl_{r,s},
\end{equation}
form a basis for $\Cl_{r,s}$. Here $1\leq i_1<\cdots< i_k \leq r+s$, $k=1,\ldots,r+s$, and $\epsilon_0$, $\epsilon_{i_1,\ldots,i_k}$ are equal to $1$ or $-1$. 

It is obvious that $(PI2)$ implies $(PI2)'$. Assume that condition 
$(PI2)'$ is fulfilled for a collection~\eqref{all products of p}. Collection~\eqref{all products of p} is a subfamily of the basis~\eqref{eq:Clifford basis} for the vector space $\Cl_{\,r,s}$, and therefore elements in~\eqref{all products of p} are linearly independent.
\end{proof}

As an example of a set $PI$ we present the minimal length positive involutions, which can be  classified in the following types and which will be generalized later in Definition~\eqref{def:TypesT1T2} to longer ones: 
\begin{equation}\label{eqT1T2}
\begin{array}{lll}
&T_1\,\,
\begin{cases}
p=z_{i_{1}}z_{i_{2}}z_{i_{3}}z_{i_{4}}, ~\text{where all $z_{i_{k}}$ are positive basis vectors;}\\
p=z_{i_{1}}z_{i_{2}}z_{i_{3}}z_{i_{4}}, ~\text{where all $z_{i_{k}}$ are negative basis vectors;}\\
p=z_{i_{1}}z_{i_{2}}z_{i_{3}}z_{i_{4}},
~\text{where two $z_{i_{k}}$ are positive and two $z_{i_{l}}$}\\
\text{\qquad\qquad  \qquad\qquad\qquad\qquad\qquad\qquad    
are negative basis vectors;}
\end{cases}\\
&T_2\,\,
\begin{cases}
q=z_{i_{1}}z_{i_{2}}z_{i_{3}},~\text{where all $z_{i_{k}}$ are positive basis vectors;}\\
q=z_{i_{1}}z_{i_{2}}z_{i_{3}},~\text{where one $z_{i_{k}}$ is positive and two $z_{i_{l}}$}\\
\text{\qquad\qquad \qquad\qquad\qquad \qquad\qquad\qquad  
are negative basis vectors.}
\end{cases}
\end{array}
\end{equation}

An easy combinatorial computation shows that generally positive involutions can contain either $3 \mod 4$ or $4\mod 4$ basis vectors. This observation inspires us to make a more general definition.

\begin{definition}\label{def:TypesT1T2}
A positive involution containing\ \ $4\mod 4$ basis vectors is called a type $T_1$ involution. A positive involution containing\ \  $3\mod 4$ basis vectors is called a type $T_2$ involution.
\end{definition}

\begin{notation}\label{set of components}
For an element $\sigma=\pm \,z_{i_1}\cdots z_{i_k}\in G(\mathfrak B_{r,s})$,
we denote by $\mathfrak{b}(\sigma)=\{z_{i_1}, \ldots,z_{i_{k}}\}$ the set of the vectors in the product $\sigma$, and by $|\mathfrak{b}(\sigma)|$ the number of the vectors in $\mathfrak{b}(\sigma)$. 
Analogously, $\mathfrak{b^+}(\sigma)$ \big($\mathfrak{b^-}(\sigma)$\big) is the set of positive $($negative$)$ vectors in $\sigma$ and $|\mathfrak{b^+}(\sigma)|$ $(|\mathfrak{b^-}(\sigma)|)$ is the cardinality of the respective sets.
\end{notation}


\begin{proposition}\label{prop:product}
The following properties can be easily verified.
\begin{itemize} 
\item[(A)] Two type $T_1$ involutions
$p_1$
and 
$p_2$ commute if the number
$|\mathfrak b(p_1)\cap \mathfrak b(p_2)|$ is even. The product $p_1p_2$ is an involution of type $T_1$.
\item[(B)]  A type $T_1$ involution  $p$
and a type $T_2$ involution $q$
commute if 
the number
$|\mathfrak b(p)\cap \mathfrak b(q)|$ is even. The product $pq$ is an involution of type $T_2$.
\item[(C)] Two type $T_2$ involutions  $q_1$
and $q_2$ commute if the number
$|\mathfrak b(q_1)\cap \mathfrak b(q_2)|$ is odd. The product $q_1q_2$ is an involution of type $T_1$.
\end{itemize}
\end{proposition}
\begin{proof}
The proof is based on the Clifford algebra property 
$$
z_1z_2+z_2z_1=-2\la z_1,z_2\ra_{r,s}{\mathbf 1},\quad z_1,z_2\in \mathbb{R}^{r,s},
$$
which for orthogonal vectors $z_1$ and $z_2$ leads to 
$
z_1z_2=-z_2z_1$.
\end{proof}

\begin{notation}\label{number of involutions in PI}
We denote by $\mathbb{PI}_{r,s}$
the collection of sets $PI$ satisfying Definition~\ref{def:PI}. 
The set $\mathbb{PI}_{r,s}$ is partially ordered by the inclusion relation similar to $\mathbb{S}_{r,s}$.
If $PI\in \mathbb{PI}_{r,s}$, then we denote by $\mathcal S(PI)$ a group generated by the set $PI$.
\end{notation}

\begin{proposition}\label{S to PI to S}
\begin{enumerate}
\item Let $PI\in \mathbb{PI}$. Then 
\begin{equation}\label{eq:subgroup}
\mathcal{S}(PI)=\{ {\mathbf 1}, \ p_{1},\ldots,p_{\ell}, \ \ldots,\  p_{i_{1}}\cdots p_{i_{k}}\vert\ 1\leq i_{1}<\cdots<i_{k}\leq \ell=\#[PI]\}
\end{equation}
is a group of order $\#[\,\mathcal{S}(PI)\,] = 2^{\ell}$ in $G(\mathfrak B_{r,s})$ and $\mathcal{S}(PI)\in \mathbb{S}$.
\item Conversely, let $\mathcal{S}\in\mathbb{S}$. Then there is a (non unique) set $PI\in \mathbb{PI}$ 
such that $\mathcal{S}(PI)=\mathcal{S}$.
\item Let $\varepsilon=(\varepsilon_1,\ldots, \varepsilon_{\ell})$ be a tuple consisting of $\pm 1$, and $PI=\{p_{i}\}_{i=1}^{\ell} \in\mathbb{PI}_{r,s}$. Then 
$
\varepsilon\cdot PI=\{\varepsilon_1p_1,\ldots,\varepsilon_{\ell}p_{\ell}\}\in \mathbb{PI}_{r,s}$ and 
$
\widehat{\mathcal{S}(PI)}=\widehat{\mathcal{S}(\varepsilon\cdot PI)}$. 
\end{enumerate}
\end{proposition}
\begin{proof}
Set in~\eqref{all products of p} is linearly independent and coincides with $\mathcal{S}(PI)$ in~\eqref{eq:subgroup}, therefore $\#[\mathcal{S}(PI)]=2^{\ell}$. If $p$ is in the set~\eqref{all products of p},
then $-p$ is not in the set~\eqref{all products of p}, which implies that 
$-{\mathbf 1}\notin\mathcal{S}(PI)$. Any $p\in \mathcal{S}(PI)$ is a positive involution by definition of the set $PI$. We showed (1). 

The second property will be proved by induction arguments 
with respect to the order of the group $\mathcal{S}$.
Let $\mathcal{S}\in \mathbb{S}_{r,s}$ be given. 
Assume $p_{1}\in\mathcal{S}$ and if there are no elements in $\mathcal{S}$ other than
${\mathbf 1}, p_{1}$, then we can put $PI=\{p_{1}\}$ and $\mathcal{S}(PI)=\mathcal{S}$.  

Assume now that there is a generating set $PI'=\{p_1, \ldots,p_{\ell}\}$, with $\ell\geq 2$ and with $p_j\in\mathcal S$, satisfying Definition~\ref{def:PI}. 
If 
$$
\mathcal S(PI')=\{{\mathbf 1},\  p_{1},\ldots,p_{\ell},\ \ldots,\ p_{i_{1}}\cdots p_{i_{k}}\mid\ 1\leq i_{1}<\cdots<i_{k}\leq \ell,\ k=1,\ldots,\ell\},
$$
is a proper subset of $\mathcal{S}$, 
then there is a positive involution $q\in \mathcal{S}$ such that $q\not\in \mathcal{S}(PI')$,
and $q\not= \pm \,{\mathbf 1}$.
Consider the set of commuting involutions
$$
\mathcal S(PI')\cdot q=\{q, \ p_{1}q,\ldots,p_{\ell}q,\ \ldots,\ p_{i_{1}}\cdots p_{i_{k}}q\mid\ 1\leq i_{1}<\cdots<i_{k}\leq \ell,\ k=1,\ldots,\ell\}.
$$
If $p_{i_{1}}\cdots p_{i_{m}}=p_{j_{1}}\cdots p_{j_{m'}}q$, then
$q\in \mathcal{S}(PI')$, as a product of involutions $p_{j_{1}}\cdots p_{j_{m'}}$ and $p_{i_{1}}\cdots p_{i_{m}}$ from $\mathcal{S}(PI')$. Thus none of the elements in $\mathcal S(PI')$ can be written in the form $p_{j_{1}}\cdots p_{j_{m'}}q$ for $p_{j_{1}}\cdots p_{j_{m'}}\in \mathcal S(PI')$. If
\[
p_{i_{1}}\cdots p_{i_{k}}\neq p_{j_{1}}\cdots p_{j_{k'}}\quad\text{for}\quad p_{i_{1}}\cdots p_{i_{k}}, ~p_{j_{1}}\cdots p_{j_{k'}}\in \mathcal{S}(PI'),
\]
then 
$p_{i_{1}}\cdots p_{i_{k}}q\neq p_{j_{1}}\cdots p_{j_{k'}}q$. So the set
$PI''=PI'\cup \{q\}$ satisfies Definition~\ref{def:PI}. 

Continuing the procedure, we find in finitely many steps a set $PI$ satisfying Definition~\ref{def:PI} such that $\mathcal{S}(PI)=\mathcal{S}$.

The proof of the last assertion easily follows from Definition~\ref{def:PI}.
\end{proof}


\subsection{Relation of $\mathcal S$ and an isotropy subgroup $\mathcal S_v$}


The groups $\mathcal S\in \mathbb S_{r,s}$ act on the Clifford module $V^{r,s}$ by the representation of the Clifford algebra. Namely if $p=z_{i_i}\cdots z_{i_k}\in \mathcal S$, then $J_pv=J_{z_{i_1}}\ldots J_{z_{i_k}}v$.

\begin{definition}
Let $v\in V^{r,s}$ be a non-null vector $\langle v,v\rangle_{V^{r,s}}\neq 0$, The isotropy subgroup $\mathcal S_{v}\subset G(\mathfrak B_{r,s})$ of the vector $v$ is
\[
\mathcal S_{v}=\{\sigma\in G(\mathfrak B_{r,s})\mid J_{\sigma}v=v\}.
\]
\end{definition}

Now we relate groups $\mathcal S\in \mathbb S_{r,s}$ of positive involutions with the isotropy subgroups $\mathcal S_v$ for some non-null vectors $v\in V^{r,s}$ and show that they are in a close relation.

\begin{proposition}
Let $v\in V^{r,s}$ be a non-null vector and let $\mathcal S_{v}$ the isotropy
subgroup of $v$. Then $\mathcal S_v\in \mathbb S_{r,s}$. 
\end{proposition}

\begin{proof}
Let us assume that $V^{r,s}$ is the Clifford module, where the volume form $\omega_{r,s}=\prod_{k=1}^{r+s}z_k$ acts as identity; that is $J_{\omega_{r,s}}u=J_{\mathbf 1}u=u$ for any $u\in V^{r,s}$. Then it is clear that $-{\mathbf 1}\notin \mathcal S_{v}$. To check the second property we take 
$\sigma\in \mathcal S_v\subset G(\mathfrak B_{r,s})$ and assume by contrary that $\sigma$ is a product containing an odd number of negative basis vectors from $\mathfrak B_{r,s}$. Then for $v\in V^{r,s}$ with $\la v,v\ra_{V^{r,s}}>0$ we obtain
$$
0<\la v,v\ra_{V^{r,s}}=\la J_{\sigma}v,J_{\sigma}v\ra_{V^{r,s}}<0
$$ 
by~\eqref{eq:isometry}, which is a contradiction. Similar arguments apply for a vector $v\in V^{r,s}$ with $\la v,v\ra_{V^{r,s}}<0$. Hence $\sigma\in \Pin(r,0)\times \Spin(0,s)$.  

The square of every element in $G(\mathfrak B_{r,s})$ equals either ${\mathbf 1}$ or $-{\mathbf 1}$.
If $\sigma\in \mathcal S_v$, then $J_{\sigma}^{2}=J_{\sigma^2}=\Id_{V^{r,s}}$. Hence ${\sigma}^{2}={\mathbf 1}$.

If $r-s=3\mod 4$ and $S_v$ includes the volume form $\omega{r,s}$ which acts as minus identity on $V^{r,s}$, then we change $\omega{r,s}$ to $-\omega{r,s}$ and the proof will follow. 
\end{proof}


\begin{notation}\label{def:ell(r,s)-0}
	Let $\mathbb{S}^{M}_{r,s}$ and $ \mathbb{PI}^{M}_{r,s}$
	be subsets of $\mathbb{S}_{r,s}$, respectively $\mathbb{PI}_{r,s}$,
	consisting of a maximal number of positive involutions. Then
	$PI\in \mathbb{PI}^{M}_{r,s}$ if and only if $\mathcal{S}(PI) \in\mathbb{S}^{M}_{r,s}$,
	although the correspondence $PI\mapsto \mathcal{S}(PI)$ is not injective.
\end{notation}

\begin{proposition}\label{mathcal S and  Sv}
Let $\mathcal{S}\in\mathbb{S}_{r,s}^M$ and $PI\in \mathbb{PI}_{r,s}^M$ 
be such that $\mathcal{S}(PI)=\mathcal{S}$. Then there is a non-null vector $v\in V^{r,s}$ such that $\mathcal S=\mathcal S_v$. In particular, $\#[\mathcal{S}]=\#[\mathcal S_{v}]=2^{\#[PI]}$.
\end{proposition}

\begin{proof}
Let $PI=\{p_{1},\ldots,p_{\ell}\}\in \mathbb{PI}_{r,s}^M$, $\mathcal{S}=\mathcal{S}(PI)$ and let 
\begin{equation}\label{eq:pm1}
E^{+1}(p_k)=\{u\in V^{r,s}\mid J_{p_k}u=u\},\quad
E^{-1}(p_k)=\{u\in V^{r,s}\mid J_{p_k}u=-u\}.
\end{equation}
We show that the intersection $E=\bigcap_{k=1}^{\ell}E^{+1}(p_k)$
contains a non-null vector $v\in V^{r,s}$ such that $\mathcal{S}(PI)=\mathcal S_{v}$. 

Let $r-s\neq 3\mod 4$ and let $E^{\pm1}(p_k)$, $E^{-1}(p_k)$ be as in~\eqref{eq:pm1}. If one of the spaces $E^{\pm 1}(p_k)$ is trivial, i.e. contains only $\{0\}$, then the symmetric
bi-linear form $\la.\,,.\ra_{V^{r,s}}$ on the non-trivial subspace is non-degenerate.
If both of $E^{\pm 1}(p_k)$ are non-trivial spaces, then  
they are orthogonal with respect to $\la.\,,.\ra_{V^{r,s}}$ and the restriction
of $\la.\,,.\ra_{V^{r,s}}$ onto $E^{\pm 1}(p_k)$ is non-degenerate too.

Assume $E^{+1}(p_1)\not= \{0\}$. Then the space $E^{+1}(p_1)$ is invariant under the action of the  involution $J_{p_2}$. Therefore,
$E^{+1}(p_1)\bigcap E^{+1}(p_2)\neq \{0\}$. By repeating the procedures we get that 
$E=\bigcap_{k=1}^{\ell}\,E^{+1}(p_k) \not=\{0\}$ and the restriction
of $\la.\,,.\ra_{V^{r,s}}$ onto $E$ is non-degenerate. Thus there is a
non-null vector $v\in E$ such that 
$J_{p_{k}}v=v$ for all $k=1,\ldots,\ell$. Hence $\mathcal{S}(PI)\subset\mathcal S_{v}$. Since the group $\mathcal{S}=\mathcal{S}(PI)$ was chosen maximal, we obtain $\mathcal{S}(PI)=\mathcal S_{v}$.

If $r-s=3\mod 4$, then without loss of generality we can assume that $J_{p_1}$ acts as $-\Id$. We change $p_1$ to $-p_1$ to get $E^{+1}(p_1)=\{u\in V^{r,s}\mid J_{p_1}u=u\}$ and continue the proof as above. 
\end{proof}

\begin{corollary}\label{orbit of non-null vector}
Let $\mathcal S\in\mathbb{S}_{r,s}$, and let $\mathcal S_v=\mathcal S$ be an isotropy subgroup of $v\in V^{r,s}$. The orbit
\begin{equation}\label{eq:orbit}
O_v=G(\mathfrak B_{r,s}).v:=\{J_{\sigma}v\mid\  \sigma\in G(\mathfrak B_{r,s})\}
\end{equation} 
contains an invariant basis $\mathfrak B(V^{r,s})$ of the minimal admissible module $V^{r,s}$ counted with signs $\pm$.
Hence $G(\mathfrak B_{r,s})/S_v\cong {G(\mathfrak B_{r,s}).v}$ and
$\dim(V^{r,s})=\frac{1}{2}\#[G(\mathfrak B_{r,s}).v]$.
\end{corollary}

\begin{proof}
If the group $\mathcal S_v$ is an isotropy subgroup of an invariant basis, then
\begin{equation}\label{max condition}
\#[\mathcal S_v]\cdot \#[G(\mathfrak B_{r,s}).v]=2^{r+s+1}=\#[G(\mathfrak B_{r,s})].
\end{equation}
Since the module is minimal admissible and the basis vectors are counted twice (with $\pm$ signs), we conclude $\#[G(\mathfrak B_{r,s}).v]=2\dim(V^{r,s})$. Therefore $\dim(V^{r,s})=2^{r+s-\ell}$, where $\ell=\#[PI]$ for $\mathcal S=\mathcal S(PI)$.
\end{proof}

\begin{notation}\label{def:ell(r,s)}
We denote by $\ell(r,s)$ the maximal number of involutions in $\mathbb{PI}_{r,s}^M$.
The value $\ell(r,s)$ depends only on the signature $(r,s)$ and it satisfies $2^{\ell(r,s)}=\frac{2^{r+s}}{\dim(V^{r,s})}$ by Corollary~\ref{orbit of non-null vector}. 
\end{notation}

\begin{proposition}\label{periodocity of ell(r,s)}
The number $\ell(r,s)$ possesses the following periodicity properties:
\begin{eqnarray*}
\ell(r+8,s)&=&\ell(r,s+8)=\ell(r+4,s+4)=\ell(r,s)+4\\
&=&
\ell(r,s)+\ell(8,0)=\ell(r,s)+\ell(0,8)=\ell(r,s)+\ell(4,4).
\end{eqnarray*}
\end{proposition}
\begin{proof}
The number $\ell(r,s)$ is determined by  
$2^{\ell(r,s)}\cdot \dim(V^{r,s})=2^{r+s}$.
Hence,
\[
2^{\ell(r+8,s)}\cdot \dim(V^{r+8,s})=2^{r+8+s}=2^{r+s}2^{8}
=2^{\ell(r,s)}\cdot \dim(V^{r,s}) \cdot 2^{8}.
\]
We know that $\dim(V^{r+8,s})=2^{4} \dim(V^{r,s})$, see~\cite[Section 4.1]{FurMar17}.
Hence it holds $\ell(r+8,s)=\ell(r,s)+4$.

Other equalities can be shown by the same arguments.
\end{proof}

\begin{proposition}
\begin{align}\label{inequality of ell(r,s)}
&\ell(r,s)\leq \ell(r+s,0),\quad\ell(r,s)\leq \ell(r+1,s),\quad\ell(r,s)\leq \ell(r,s+1).
\end{align}
\end{proposition}
\begin{proof}
The inequalities follow from the structure of involutions in
\eqref{eqT1T2}.
 \end{proof}

The orbit $O_v={G(\mathfrak B_{r,s}). v}$ gives the invariant basis for $V^{r,s}$ up to a sign. Since the elements in $G(\mathfrak B_{r,s})$ either commute or anti-commute with elements in $\mathcal S_v$, we can describe the construction of an invariant basis for a minimal admissible module $V^{r,s}$.

\begin{theorem}\label{basis expression}
Let $v\in V^{r,s}$ be a unit vector from Proposition~\ref{mathcal S and  Sv}. There is a set 
$\Sigma\subset G(\mathfrak B_{r,s})$ such that the family 
$\{J_{\sigma}v\}_{\sigma\in \Sigma}$ is an invariant basis of $V^{r,s}$.
\end{theorem}

\begin{proof}
Let $\mathcal S_v\in\mathbb{S}_{r,s}^M$. We fix a maximal set $PI_{r,s}=\{p_i\}_{i=1}^{\ell(r,s)}$ such that $\mathcal S(PI_{r,s})=\mathcal S_v$ and write 
$E^{\varepsilon_i}(p_i)=\{v\in V^{r,s}\mid J_{p_i}v= \varepsilon_i v\}$, where $\varepsilon_i$ is either $+1$ or $-1$. We denote $\varepsilon=(\varepsilon_1,\ldots, \varepsilon_{\ell(r,s)})$ and define
\begin{equation}\label{eq:epsilon}
E=\bigcap_{i=1}^{\ell(r,s)}\, E^{+1}(p_i),\quad E^{\varepsilon_1,\ldots, \varepsilon_{\ell(r,s)}}=\bigcap_{i=1}^{\ell(r,s)}\, E^{\varepsilon_i}(p_i).
\end{equation}
Before we continue the proof we note that $\dim(E)\in\{1,2,4,8\}$, and either $\dim(V^{r,s})=\dim(E)\times 2^{\ell(r,s)}$ or
$\dim(V^{r,s})=\dim(E)\times 2^{\ell(r,s)-1}$. In the latter case, one involution $J_{p_i}$ acts as $\Id$ or $-\Id$ on $V^{r,s}$, which happens if $r-s=3\mod 4$, see details in~\cite{FurMar21}. Thus 
$$
\dim(E)=2^{r+s-2\ell(r,s)}\quad\text{or}\quad \dim(E)=2^{r+s-2\ell(r,s)+1}.
$$

Let ${\mathbf C}_{G(\mathfrak B_{r,s})}(\mathcal{S}(PI_{r,s}))$ be the centralizer
of $\mathcal{S}(PI_{r,s})$ in $G(\mathfrak B_{r,s})$ and   
$v\in E$ a unit vector. Then 
we can find representatives 
$\{\sigma_{i}\}_{i=1}^{\dim(E)-1}\in {\mathbf C}_{G(\mathfrak B_{r,s})}\big(\mathcal{S}(PI_{r,s})\big)\big/ \widehat{\mathcal{S}(PI_{r,s})}$, 
and 
$\{\tau_{j}\}_{j=1}^{2^{\ell(r,s)}}\in G(\mathfrak B_{r,s})\big/{\mathbf C}_{G(\mathfrak B_{r,s})}\big(\mathcal{S}(PI_{r,s})\big)$ such that
\begin{equation*}\label{eq:Sigma}
\begin{split}
&\text{the vectors}
\quad \{J_{\sigma_i}v\}_{i=1}^{\dim(E)-1}\quad \text{form an orthonormal basis for $E$},
\\
&\text{the vectors}\quad \{J_{\tau_j}J_{\sigma_i}v\}_{i=1}^{\dim(E)-1}~_{j=1}^{2^{\ell(r,s)}}\quad\text{
form an orthonormal basis for $V^{r,s}$.}
\end{split}
\end{equation*}
These $\{\sigma_{i}\}_{i=1}^{\dim(E)-1}$ and $\{\tau_{j}\}_{j=1}^{2^{\ell(r,s)}}$ form the set $\Sigma$.
\end{proof}

\begin{theorem}\label{prop:vector_v}
Fix the group $\mathcal{S}(PI_{r,s})$ and the representatives 
$$
\{\sigma_{i}\}_{i=1}^{\dim(E)-1}\in {\mathbf C}_{G(\mathfrak B_{r,s})}\big(\mathcal{S}(PI_{r,s})\big)\big/\widehat{\mathcal{S}(PI_{r,s})},
$$
$$
\{\tau_j\}_{j=1}^{2^{\ell(r,s)}}\in G(\mathfrak B_{r,s})\big/{\mathbf C}_{G(\mathfrak B_{r,s})}\big(\mathcal{S}(PI_{r,s})\big). 
$$
Assume that $v_1,v_2\in E$ generate 
two sets of invariant bases
\[
\mathfrak B_{v_k}(V^{r,s})=\{v_k,\  J_{\sigma_i}v_k,\ J_{\tau_{j}}v_k,\ J_{\tau_j}J_{\sigma_{i}}v_k\}_{i=1}^{\dim(E)-1}\,_{j=1}^{2^{\ell(r,s)}},\quad k=1,2,
\]
as in Theorem~\ref{basis expression}. Then the invariant integral structures 
\begin{equation}\label{eq:inv structures v1v2}
\begin{array}{lll}
&\spn_{\mathbb Z}\{\mathfrak B_{v_1}(V^{r,s})\}\oplus \spn_{\mathbb Z}\{\mathfrak B_{r,s}\}
\\
\\
&\spn_{\mathbb Z}\{\mathfrak B_{v_2}(V^{r,s})\}\oplus \spn_{\mathbb Z}\{\mathfrak B_{r,s}\}
\end{array}
\end{equation}
are isomorphic. 
\end{theorem}

\begin{proof}
We define the correspondence $A\colon \mathfrak B_{v_1}(V^{r,s})\to \mathfrak B_{v_2}(V^{r,s})$ by 
\begin{equation}\label{eq:correspondence}
\begin{array}{lllll}
&v_1\mapsto v_2,\quad 
&J_{\sigma_i}v_1\mapsto J_{\sigma_i}v_2,
\\
&J_{\tau_j}v_1\mapsto J_{\tau_j}v_2,\quad
&J_{\tau_{j}}J_{\sigma_i}v_1\mapsto J_{\tau_{j}}J_{\sigma_i}v_2,
\end{array}
\end{equation}
and extend it by linearity over $\mathbb Z$. We claim that the map $A\oplus \Id$ is an isomorphism of invariant integral structures~\eqref{eq:inv structures v1v2}. To show this, we denote the basis vectors from $\mathfrak B_{v_1}(V^{r,s})$ by $\{u_{\alpha}\}_{\alpha=1}^{\dim(V^{r,s})}$ and the basis vectors from $\mathfrak B_{v_2}(V^{r,s})$ by $\{w_{\alpha}\}_{\alpha=1}^{\dim(V^{r,s})}$, where $w_{\alpha}=Au_{\alpha}$. The invariance property of the bases $\mathfrak B_{v_1}(V^{r,s})$ and $\mathfrak B_{v_2}(V^{r,s})$ implies that for any $u_{\alpha}\in \mathfrak B_{v_1}(V^{r,s})$ and any  $z_k\in \mathfrak B_{r,s}$ there is $u_{\beta}\in \mathfrak B_{v_1}(V^{r,s})$ such that 
\begin{equation}\label{eq:Jz}
J_{z_k}u_{\alpha}=\pm u_{\beta}=\pm J_{\varkappa}v_1,\quad\text{for some}\quad
\varkappa\in\Sigma=\{\sigma_i,\tau_j,\tau_j\sigma_i\}.
\end{equation} 
The correspondence~\eqref{eq:correspondence} and~\eqref{eq:Jz} imply that for chosen $u_{\alpha}\in \mathfrak B_{v_1}(V^{r,s})$ and $z_k\in \mathfrak B_{r,s}$ we have 
$$
J_{z_k}Au_{\alpha}=J_{z_k}w_{\alpha}=\pm w_{\beta}=
\pm J_{\varkappa}v_2=
\pm AJ_{\varkappa}v_1=AJ_{z_k}u_{\alpha}.
$$
Note also that $A^{\tau}A=\Id_{V^{r,s}}$ since it maps an othonormal basis to an orthonormal basis. Then we have 
\begin{eqnarray}\label{eq:prop316}
\langle [Au_{\alpha},Au_{\beta}],z_k\rangle_{r,s} &=&
\langle J_{z_k}Au_{\alpha},Au_{\beta}\rangle_{V^{r,s}}
=
\langle AJ_{z_k}u_{\alpha},Au_{\beta}\rangle_{V^{r,s}}\nonumber
\\
&=&
\langle A^{\tau}AJ_{z_k}u_{\alpha},u_{\beta}\rangle_{V^{r,s}}=
\langle J_{z_k}u_{\alpha},u_{\beta}\rangle_{V^{r,s}}
\\
&=&
\langle [u_{\alpha},u_{\beta}],z_k\rangle_{r,s}.\nonumber
\end{eqnarray}
\end{proof}


\subsection{Equivalence of groups $\mathcal S$}\label{sec:equivalence_small}


We define an equivalence relation between groups $\mathcal S\subset \mathbb S_{r,s}$ that will descend to the equivalence of their generating sets $PI_{r,s}$. We also introduce signatures to distinguish sets $PI_{r,s}$ for a fixed value $(r,s)$. Different signatures will lead to non-equivalent generating sets $PI_{r,s}$ and groups $\mathcal S=\mathcal S(PI_{r,s})$. Our aim is to show that equivalent groups $\mathcal S\subset \mathbb S_{r,s}$ lead to the isomorphic invariant integral structures on the Lie algebras $\mathfrak n_{r,s}$.

We recall Notation~\ref{set of components} and extend it to the sets $PI$.
\begin{notation}\label{b+ and b- of PI}
Let $PI\in \mathbb{PI}_{r,s}$. We denote
\begin{align*}
&\mathfrak{b}^{+}(PI)=\{z_{i}\mid\ z_{i}\ \text{is a positive vector in some $p\in PI$}\},
\\
&\mathfrak{b}^{-}(PI)=\{z_{i}\mid\ z_{i}\ \text{is a negative vector in some $p\in PI$}\}.
\end{align*}
We set also $|\mathfrak{b}^{+}(PI)|$, $|\mathfrak{b}^{-}(PI)|$ for the cardinality of the respective set, and $|\mathfrak{b}(PI)|=|\mathfrak{b}^{+}(PI)|+|\mathfrak{b}^{-}(PI)|$.
\end{notation}

\begin{definition}\label{not:(T12)}
A set $PI$ consisting only of the involutions of type $T_1$ will be called
$(T1)$-type set. A set $PI$ consisting of the involutions of type $T_1$ and having at least one involution of type $T_2$ will be called
$(T2)$-type set. 
\end{definition}

\begin{proposition}
Any $(T2)$-type set of positive involutions, which contains several type $T_2$ involutions can be changed to a $(T2)$-type set containing at most one involution of type $T_2$ and the rest of involutions will be of type $T_1$. 
\end{proposition}
\begin{proof}
The proof follows directly from Proposition~\ref{prop:product}.
\end{proof}

\begin{notation}
If $C\in \Orth(r,s)$, then
we denote by the same letter $C$ its natural extension $C\colon \Cl_{\,r,s}^{\,*}
\to \Cl_{\,r,s}^{\,*}$ to the action on the group of invertible elements 
$\Cl_{\,r,s}^{\,*}\subset \Cl_{\,r,s}$.
\end{notation}

Let $\mathfrak B_{r,s}$ be a basis as in~\eqref{eq:brs}. Let $C\in \Orth(r,s)$. 
If $\sigma=z_{i_1}\cdots z_{i_k}\in G(\mathfrak B_{r,s})$, then $C$ 
defines an element $C(\sigma):=C(z_{i_1})\cdots C(z_{i_k})\in G(\mathfrak B_{r,s})$.
Since $C$ maps positive basis vectors to positive vectors 
and negative basis vectors to negative basis vectors, 
it induces a map $C\colon\mathbb{PI}_{r,s}\to \mathbb{PI}_{r,s}$. For the particular case $(r,r)$ the map $C$ can be chosen to map positive basis vectors to negative vectors and vice versa. The changes for $(r,r)$ will be discussed separately
in a forthcoming paper.

\begin{definition}\label{def:group equiv} 
We say that the groups $\mathcal S_1$ and $\mathcal S_2$ are equivalent,
writing $\mathcal S_1\sim\mathcal S_2$, 
if  
there is a map $C\in \Orth(r,s)$ such that its natural extention to
$\Cl_{\,r,s}^{*}\subset \Cl_{\,r,s}$
gives the isomorphism between the extended groups
$\widehat{\mathcal{S}_1}$ and 
$\widehat{\mathcal{S}_2}$; that is $C(\widehat{\mathcal{S}_1})=\widehat{\mathcal{S}_2}$.
\end{definition}

\begin{definition}\label{equiv relation} 
Let $PI_1$ and $PI_2$ be two sets of involutions. 
Then we say that $PI_1$ and $PI_2$ are equivalent,
writing $PI_1\sim PI_2$, 
if  
$\mathcal S(PI_1)$ is equivalent to $\mathcal S(PI_2)$ in the sense of Definition~\ref{def:group equiv}.
\end{definition}

\begin{example}
Recall Example~\ref{example 1} and consider $G(\mathfrak B_{4,0})$. Then $PI_1=\{z_1z_2z_3\}$ and $PI_2=\{z_1z_2z_4\}$ are equivalent, nevertheless $PI_1$ is not equivalent to $PI_5=\{z_1z_2z_3z_4\}$.
\end{example}


\subsection{Standard subgroups}\label{ex:standard group}

In this section we present a construction of a sequence of subgroups that will be important in Section~\ref{sec:nonisomorphic subgroups}. We call these subgroups {\it standard}. Let $\mathfrak B_{r,s}$ be an orthonormal basis of $\mathbb R^{r,s}$ as in~\eqref{eq:brs}. We form a set of mutually different pairs
\begin{equation}\label{eq:pi}
\pi_{i,j}=z_{i}z_{j},\ \ i<j,\ \ i,j\in 
\begin{cases}
\{1,\ldots,r\}&\quad\text{if\ $r$\ is even}
\\
\{1,\ldots,r-1\}&\quad\text{if\ $r$\ is odd}
\end{cases},
\end{equation}
\begin{equation}\label{eq:nu}
\nu_{k,l}=z_{k}z_{l},\ \ k<l,\ \ k,l\in 
\begin{cases}
\{r+1,\ldots,s\}&\quad\text{if\ $s$\ is even}
\\
\{r+1,\ldots,s-1\}&\quad\text{if\ $s$\ is odd}
\end{cases},
\end{equation}
and
$$
\mathfrak b(\pi_{i_1,j_1})\cap\mathfrak b(\pi_{i_2,j_2})=\emptyset,
\quad
\mathfrak b(\nu_{k_1,l_1})\cap\mathfrak b(\nu_{k_2,l_2})=\emptyset.
$$
The cardinalities of the sets of pairs are
$$
{\bf p}=\#\{\pi_{i,j}\}=\begin{cases}
\frac{r}{2}&\quad\text{if\ $r$\ is even}
\\
\frac{r-1}{2}&\quad\text{if\ $r$\ is odd}
\end{cases},
\quad
{\bf n}=\#\{\nu_{kl}\}=\begin{cases}
\frac{s}{2}&\quad\text{if\ $s$\ is even}
\\
\frac{s-1}{2}&\quad\text{if\ $s$\ is odd}
\end{cases}.
$$
Now we form a set of involutions of type $T_1$, which from now on will be denoted always by $p_i$. For any positive integers $\bar p\in \{1,\ldots,{\bf p}\}$  and $\bar n\in \{1,\ldots,{\bf n}\}$ we make a product of pairs:
\begin{equation}\label{eq:pinu}
\pi_{i_\alpha, j_{\alpha}}\pi_{i_\beta, j_{\beta}},\quad
\pi_{i_\alpha, j_{\alpha}}\nu_{k_\gamma, l_{\gamma}},\quad
\nu_{k_\gamma, l_{\gamma}}\nu_{k_\delta, l_{\delta}},\quad
\alpha,\beta\in\{1,\ldots,{\bar p}\},\ \gamma,\delta\in\{1,\ldots,{\bar n}\}.
\end{equation}
We denote by $\mathcal S^{ \bar p,\bar n}$ the group generated by involutions~\eqref{eq:pinu}.

\begin{proposition}\label{prop:standard group}
In the notation above the groups $\mathcal S^{\bar p,\bar n}$ have the following properties.
\begin{itemize} 
\item[(i)] $\mathcal S^{\bar p,\bar n}$ is a subgroup of $G(\mathfrak B_{r,s})$ for any $\bar p\in \{1,\ldots,{\bf p}\}$  and $\bar n\in \{1,\ldots,{\bf n}\}$;
\item[(ii)] $\mathcal S^{\bar p-k_1,\bar n-k_2}$ is a subgroup of $\mathcal S^{\bar p,\bar n}$ for any $k_1=0,1,\ldots, \bar p$ and $k_2=0,1,\ldots, \bar n$;
\item[(ii)] The standard groups $\mathcal S^{\bar p,\bar n}$ are equivalent for fixed $(\bar p,\bar n)$ in the sense of Definition~\ref{def:group equiv};
\item[(iv)] Any set $PI_{r,s}$ satisfying Definition~\ref{def:PI} and such that $\mathcal S^{\bf p,\bf n}=\mathcal S(PI_{r,s})$ will be equivalent in the sense of Definition~\ref{equiv relation};
\item[(v)] Pairs $\pi_{i,j}$ and $\nu_{k,l}$ commute with all elements in $\mathcal S^{\bf p,\bf n}$;
\item[(vi)] Let $\theta=z_{i_1}\cdots z_{i_{\bf p+\bf n}}$ be a product such that each $z_{i_t}$, $t=1,\ldots,\bf p+\bf n$ belongs only to one pair from~\eqref{eq:pi} or~\eqref{eq:nu}. Then $\theta$ commutes with all elements in $\mathcal S^{\bf p,\bf n}$.
\end{itemize}
\end{proposition}
\begin{proof}
Properties (i)-(ii) are obvious. Statements (iii) and (iv) follows from the fact the pairs can be chosen up to a sign permutation of the basis in $\mathbb R^{r,s}$. Properties (v) and (vi) are the consequence of the facts that pairs $\pi_{i,j}$,  $\nu_{k,l}$, and the product $\theta$ will have even number of common elements and that the number of vectors $z_i$ in any element of the group $\mathcal S^{\bf p,\bf n}\subset G(\mathfrak B_{r,s})$ is also even.
\end{proof}

\begin{example}
Consider $\mathbb R^{6,3}$ with the basis $\mathfrak  B_{6,3}=\{z_1,\ldots,z_9\}$. The first six elements of the basis are positive and the last three are negative. We can choose the pairs 
\begin{equation}\label{eq;6-3}
\pi_{1,2}=z_1z_2,\quad
\pi_{3,4}=z_3z_4,\quad
\pi_{5,6}=z_5z_6,\quad
\nu_{7,8}=z_7z_8,
\end{equation}
up to the sign permutation. They generate a group $\mathcal S^{{\bf 3},{\bf 1}}\subset G(\mathfrak B_{6,3})$ of cardinality $\#\mathcal S^{{\bf 3},{\bf 1}}=8$. A possible choice of $(T1)$-type set of involutions $PI$ generating $\mathcal S^{{\bf 3},{\bf 1}}$ is
\begin{equation}\label{eq;p6-3}
PI_{6,3}=\{p_1=\pi_{1,2}\pi_{3,4},\quad p_2=\pi_{1,2}\pi_{5,6},\quad p_3=\pi_{1,2}\pi_{7,8}\}.
\end{equation}
Any pair from~\eqref{eq;6-3} will commute with involutions in~\eqref{eq;p6-3} and therefore with all elements in the group $\mathcal S^{{\bf 3},{\bf 1}}\subset G(\mathfrak B_{6,3})$. Furthermore, 
$\theta=z_1z_3z_5z_7$, which is chosen up to a sign permutation, commutes with elements in the group $\mathcal S^{{\bf 3},{\bf 1}}\subset G(\mathfrak B_{6,3})$ as well.  The pairs
$$
\pi_{1,2},\quad \pi_{3,4},\quad \pi_{5,6}\quad \text{generates the subgroup $\mathcal S^{{\bf 3},{\bf 0}}\subset \mathcal S^{{\bf 3},{\bf 1}}$}.
$$
Likewise the pairs
$$
\pi_{1,2},\quad \pi_{3,4},\quad \pi_{7,8}\quad \text{generates the subgroup $\mathcal S^{{\bf 2},{\bf 1}}\subset \mathcal S^{{\bf 3},{\bf 1}}$}.
$$
The groups $\mathcal S^{{\bf 3},{\bf 0}}$ and $\mathcal S^{{\bf 2},{\bf 1}}$ are not equivalent. They are respective representatives in their classes of equivalence. 
\end{example}


\subsection{Connectivity of groups $\mathcal S$}


Here we introduce another invariant which describes equivalent subgroups $\mathcal S\subset \mathbb S_{r,s}$. We call it ``connectedness'' for $\mathcal S=\mathcal S(PI_{r,s})$. We use the notation $\mathfrak{b}(\mathcal S)$ to denote the set of orthonormal basis vectors from $\mathfrak B_{r,s}$ which generates the group $\mathcal S$, compare with Notations~\eqref{set of components} and~\eqref{b+ and b- of PI}. 

\begin{definition}\label{S is connected}
A group $\mathcal S\in\mathbb{S}_{r,s}$ is called connected if there is no two subgroups $\mathcal S_{(1)},\mathcal S_{(2)}\subset \mathcal S$, such that 
\begin{equation*}
\mathcal S\quad\text{is isomorphic to}\quad \mathcal S_{(1)}\times \mathcal S_{(2)}\quad\text{with}\quad \mathfrak{b}(\mathcal S_{(1)})\cap\mathfrak{b}(\mathcal S_{(2)})=\emptyset
\end{equation*} 
We write in this case $\pi_0(\mathcal S)=1$.

If a group $\mathcal S\in\mathbb{S}_{r,s}$ admits the decomposition into  subgroups
$\mathcal S=\mathcal S_{(1)}\times \ldots\times \mathcal S_{(k)}$ with $\pi_0(\mathcal S_{(i)})=1$ and $\mathfrak{b}(\mathcal S_{(i)})\cap\mathfrak{b}(\mathcal S_{(j)})=\emptyset$ for any $i\not=j$,
then we say that $\mathcal S$ is disconnected with $k$ connected components and we write $\pi_0(\mathcal S)=k$. 
\end{definition}

\begin{lemma}\label{lem:connectedness}
Let $PI=\{p_i\}_{i=1}^{\ell(r,s)}\in\mathbb{PI}^{M}_{r,s}$, and 
$|\mathfrak{b}(PI)|=r+s$. Assume that there is $z_{\alpha}\in \mathfrak B_{r,s}$ such that $z_{\alpha}\in\bigcap_{i=1}^{\ell(r,s)} \mathfrak{b}(p_i)$, and moreover, there is 
no $\sigma\in\mathcal{S}(PI)$ such that $\mathfrak{b}(\sigma)\subset \mathfrak{b}(p_i)$ for any $p_i\in PI$. Then $\pi_0(\mathcal{S}(PI))=1$.
\end{lemma}
\begin{proof}

Note that any product $\prod_{j=1}^{2k+1} p_j$ of odd number contains $z_\alpha$. Let us assume that $\mathcal S=\mathcal S_{(1)}\times \mathcal S_{(2)}$ is a non-trivial decomposition.

If both subgroups include a product of odd number of involutions $\prod_{j=1}^{2l+1} p_j$, $p_j\in PI$,
then $z_{\alpha}\in\mathfrak{b}(\mathcal S_{(1)})\bigcap \mathfrak{b}(\mathcal S_{(2)})$.  Therefore $\mathcal S$ should be connected.

Assume that the subgroup $\mathcal S_{(1)}$ consists of only even products 
$\eta=\prod_{j}^{2k} p_j$ of involutions in $PI$. We write one of these products in the form  $\eta=p_{i_0}\cdot \sigma\in \mathcal{S}_{(1)}$, where $p_{i_0}$ is one of the generators from the set $PI$ and $\sigma$ is a product of odd number of some involutions in $PI$. It implies 
that $\sigma\in \mathcal{S}_{(2)}$. 
By the assumption  
$\mathfrak{b}(\sigma)\not\subset \mathfrak{b}(p_i)$ for any $p_i\in PI$, 
there exists a basis vector $z_{\beta}\in \mathfrak b(\sigma)$ such that $z_{\beta}\notin \mathfrak b(p_{i_0})$. This implies that $z_{\beta}\in \mathfrak b(p_{i_0}\cdot \sigma)$ and therefore $z_{\beta}\in\mathfrak b(\sigma)\cap\mathfrak b(p_{i_0}\cdot \sigma)\subset\mathfrak{b}(\mathcal{S}_{(2)})\bigcap\mathfrak{b}(\mathcal{S}_{(1)})$. This shows that the group $\mathcal S$ is connected.
\end{proof}

\begin{example}\label{connected cases} 
The standard subgroups $\mathcal{S}^{\bf p,0}\in\mathbb{S}_{r,0}$ constructed in Section~\ref{ex:standard group}
are connected for any $r\geq 0$.
\end{example}

\begin{proposition}\label{equivalence and pi0}
Let $PI_1, PI_2\in \mathbb{PI}_{r,s}^M$ be two generating sets. If $PI_1~\sim PI_2$, then $\pi_0(PI_1)=\pi_0(PI_2)$.
\end{proposition}
\begin{proof}
We write $PI_1=\{p_k\}_{k=1}^{\ell(r,s)}$, $PI_2=\{q_m\}_{m=1}^{\ell(r,s)}$ and $|\mathfrak b(PI_j)|=t$, $j=1,2$. By the assumption there exists an orthogonal map $C$ such that $C\Big(\widehat{\mathcal S\big(\{p_k\}_{k=1}^{\ell(r,s)}\big)}\Big)=\widehat{\mathcal S\big(\{q_m\}_{m=1}^{\ell(r,s)}\big)}$. If 
$$
\mathcal S(PI_1)=\mathcal S_{(1)}\times\mathcal S_{(2)}=\mathcal S_{(1)}(PI_{11})\times\mathcal S_{(2)}(PI_{12}),\quad 
$$
with
$$
PI_{11}=\{p_{i_k}\}_{k=1}^{a},\quad |\mathfrak b(\{p_{i_k}\}_{k=1}^{a})|=\beta,
$$
and
$$
PI_{12}=\{p_{j_k}\}_{k=a+1}^{\ell(r,s)},\quad |\mathfrak b(\{p_{j_k}\}_{k=a+1}^{\ell(r,s)})|=t-\beta,
$$
then $\mathfrak b(\{p_{i_k}\}_{k=1}^{a})\cap \mathfrak b(\{p_{j_k}\}_{k=a+1}^{\ell(r,s)})=\emptyset$.  The map $C$ will map the disjoint sets $\mathfrak b(\{p_{i_k}\}_{k=1}^{a})$ and $\mathfrak b(\{p_{j_k}\}_{k=a+1}^{\ell(r,s)})$ onto disjoint sets $\mathcal Z_1=\{z_{i_1},\ldots,z_{i_{\beta}}\}$ and $\mathcal Z_2=\{z_{j_{\beta+1}},\ldots,z_{i_{t}}\}$. The set $\mathcal Z_1$ (with possible change of signs) will form the set $PI_{21}=\{q_{i_k}\}_{k=1}^{a}$ and the set $\mathcal Z_2$ (again with possible change of signs) will form the set $PI_{22}=\{q_{j_k}\}_{k=a+1}^{t}$.
Thus we obtain
$\mathcal S(PI_2)=\mathcal S(PI_{21})\times\mathcal S(PI_{22})$.
\end{proof}

We describe how the $\mathbb Z_2$-graded product of Clifford algebras can lead to the construction of disconnected subgroups $\mathcal S\subset \mathbb S_{r,s}$.
Consider the following decomposition of an orthonormal basis
$
\mathfrak B_{r,s}=\{z_1,\ldots, z_r,z_{r+1},\ldots,z_{r+s}\}
$:
\begin{equation*}\label{decom 1}
\underbrace{z_{1},\ldots,z_{r_1}}_{\text{positive}}, \underbrace{z_{r+1},\ldots,z_{r+s_1}}_{\text{negative}},
\quad\text{and}\quad
\underbrace{z_{r_1+1},\ldots,z_r}_{\text{positive}}, \underbrace{z_{r+s_1+1},\ldots,z_{r+s}}_{\text{negative}},
\end{equation*}
for $r_1\leq r$ and $s_1\leq s$. We put $r_2=r-r_1$ and $s_2=s-s_1$ and consider the decomposition
$\mathbb{R}^{r,s}\cong \mathbb{R}^{r_1,s_1}\oplus\mathbb{R}^{r_2,s_2}$, where
we assume $r_1+s_1\geq (r-r_1)+(s-s_1)=r_2+s_2$. This decomposition leads to the isomorphism 
$\Cl_{r_1,s_1}\widehat{\otimes}\Cl_{r_2,s_2}\cong \Cl_{r_1+r_2,s_1+s_2}
=\Cl_{r,s}$, where $\widehat{\otimes}$ denotes the $\mathbb Z_2$-graded tensor product of Clifford algebras, see~\cite[Proposition 1.5]{MR1031992}. For each of the Clifford algebras $\Cl_{r_k,s_k}$, $k=1,2$, we consider the minimal admissible modules $V^{r_k,s_k}$ and the corresponding sets $PI_{r_k,s_k}$.

For $r=r_1+r_2$ and $s=s_1+s_2$, we have 
$\ell(r_1,s_1)\leq \ell(r,s)$.
Let  $PI_{r_1,s_1}\in\mathbb{PI}_{r_1,s_1}^{M}$
and $PI_{r_2,s_2}\in\mathbb{PI}_{r_2,s_2}^{M}$ satisfy
$$
|\mathfrak{b}^{+}(PI_{r_1,s_1})|=r_1,\quad |\mathfrak{b}^{-}(PI_{r_1,s_1})|=s_1,
$$
$$
|\mathfrak{b}^{+}(PI_{r_2,s_2})|=r_2,\quad |\mathfrak{b}^{-}(PI_{r_2,s_2})|=s_2,
$$
 and 
 $PI_{r_1,s_1}\bigcap PI_{r_2,s_2}=\emptyset$. We assume also that each set contains
 at most one type $T_2$ involution $q_k\in PI_{r_k,s_k}$, $k=1,2$. Then by non-commutativity of $q_1$ and $q_2$ it is easy to see the following properties:
\begin{itemize}
\item[] If at least one of the sets $PI_{r_1,s_1}$ or $PI_{r_2,s_2}$ is $(T1)$-type set,
then
\[
PI_{r_1,s_1}\bigcup PI_{r_2,s_2}\in\mathbb{PI}_{r,s}.
\]
This implies
\begin{equation}\label{sum1}
\ell(r_1,s_1)+\ell(r_2,s_2)\leq \ell(r,s).
\end{equation}
\item[] If both $PI_{r_1,s_1}$ and $PI_{r_2,s_2}$ are $(T2)$-type sets, containing type $T_2$ involutions $q_1\in PI_{r_1,s_1}$ and $q_2\in PI_{r_2,s_2}$,
then 
\[
\big(PI_{r_1,s_1}\backslash \{q_1\}\big)\bigcup PI_{r_2,s_2}\in \mathbb{PI}_{r,s}
\quad\text{and}\quad
PI_{r_1,s_1}\bigcup \big(PI_{r_2,s_2}\backslash\{q_2\}\big)\in\mathbb{PI}_{r,s}.
\]
This implies
\begin{equation}\label{sum2}
\ell(r_1,s_1)+\ell(r_2,s_2)-1\leq \ell(r,s).
\end{equation}
\end{itemize}
One can state similar properties for any number of components in a decomposition $PI=\cup_k PI_{r_k,s_k}$.

\begin{remark}
If the equalities in~\eqref{sum1} or~\eqref{sum2} hold, then
non-connected subgroups $\mathcal S(PI_{r_1,s_1})$ and $\mathcal S(PI_{r_2,s_2})$ can be constructed from lower dimensions and 
$$
\mathcal{S}(PI_{r,s})=\mathcal S(PI_{r_1,s_1})\times\mathcal S(PI_{r_2,s_2}).
$$ 
Particularly, if $r\leq 9$ and $s\in\{0,1\}$, then all the groups are connected. It follows by showing that the inequalities~\eqref{sum1} and~\eqref{sum2} are always strict.
\end{remark}


\section{Construction of subgroups in $\mathbb{S}_{r,s}^M$, $r\in\{3,\ldots,16\}$, $s\in\{0,1\}$}\label{sec:nonisomorphic subgroups}



\subsection{General method of the construction}


In this section we apply the previous theory for the classification of groups $\mathcal S\subset \mathbb S_{r,s}$. We illustrate the theory by exact construction of non-equivalent subgroups for $0\leq r+s\leq 16$. We
restrict ourselves to $0\leq r+s\leq 16$ because we want to illustrate the main features that appear in classification without diving into technical details. The classification for arbitrary $\mathcal S\subset \mathbb S_{r,s}$ is postponed for a forthcoming paper. As we mentioned earlier for $r+s\geq 17$ the new  invariants, except of $\ell(r,s)$, $|\mathfrak b(PI)|$, $T_k$-type of $PI$ and connectedness, must be considered. For the moment we are not aware about new features that can occur for $r+s\geq 17$, and must work additionally to find them. 

We start from $s=0$. We classify groups $\mathcal S\in\mathbb S^M_{r,0}$ according to parameters: $\pi_0(\mathcal S)$, $|\mathfrak b(PI_{r,0})|$, and the type $(T1)$ or $(T2)$ of the set $PI$ generating the group $\mathcal S(PI)=\mathcal S\in\mathbb S^M_{r,0}$. We use the standard groups and notations introduced in Section~\ref{ex:standard group}. For a standard group we will add from none to two additional involutions $\pi_{1,2}$ and $\theta$, see Step 1 below for details. To distinguish the groups, where all previous parameters coincide, we assign the following information about $(TI)$-type sets, $I=1,2$: 
\begin{equation}\label{cases of adding involutions}
\left\{
\begin{array}{l}
\text{(i) We use the signature $(TI,\pi)$ if an additional involution }\\
\text{\qquad is related to product $\pi_{1,2}$;}\\
\text{(ii) We use the signature $(TI,\theta)$ if an additional involution }\\
\text{\qquad is related to product $\theta$;}\\
\text{(iii) We use the signature $(TI,\pi,\theta)$ if there are two additional }\\
\text{\qquad involutions, which are related to both products $\pi_{1,2}$ and $\theta$;}\\
\text{(iv) Finally we just write $(TI)$ if there are no involutions,}\\
\text{\qquad except of standards;}
\end{array}
\right.
\end{equation}
For each set of involutions $PI_{r,0}$ we write the signature
\begin{equation}\label{eq:signature0}
(\ell(r,0), (T1,\bullet,\bullet), |\mathfrak b(PI_{r,0})|).
\end{equation}

We summarise the results in Table~\ref{connected}. We list the set of generators $PI_{r,0}$ for each group. The group itself and the set of generators will be given up to a sign permutation. The word {\it unique} is understood in the sense of equivalence relation of Definition~\ref{def:group equiv} or Definition~\ref{equiv relation}.  

\subsubsection{Main steps of the construction of $\mathcal S\in\mathbb S_{r,0}^M$ for a fixed $r>0$.} We divide the construction into three steps.

{\it Step 1.} We start from a group $\mathcal S\in\mathbb S_{r,0}^M$  satisfying $\pi_0(\mathcal S)=1$ and $|\mathfrak b(PI_{r,0})|=r$. First we find a standard subgroup $\mathcal S^{\bf p,0}\subset \mathcal S$ and complement it (if necessary) by involutions to reach the maximal number  $\ell(r,0)$ of involutions in $PI_{r,0}$ generating $\mathcal S(PI_{r,0})\in\mathbb S_{r,0}^M$. All maximal standard subgroups $\mathcal S^{\bf p,0}\subset\mathcal S\in\mathbb{S}_{r,0}^M$, are equivalent modulo reordering by 
induction arguments
with respect to the dimension $(r,0)$, see also Proposition~\ref{prop:standard group}, item (iv). The additional involutions will be formed by checking whether the product of $\pi_{1,2}$ and/or $\theta$ by $z_r$ are involutions commuting with $\mathcal S^{\bf p,0}$. Note that once we fix a standard group $\mathcal S^{\bf p,0}$ with its generators of the form~\eqref{eq:pinu},
the complemented sets $PI^M_{r,0}$ (which is obtained by adding involutions $\pi$ and/or $\theta$) will be equivalent in the sense of Definition~\ref{equiv relation} if $PI^M_{r,0}$ have the same signature described in~\eqref{cases of adding involutions} and $\pi_0(\mathcal S(PI_{r,0}))=1$.

Then we consider a smaller standard subgroup $\mathcal S^{\bf p-1,0}\subset\mathcal S^{\bf p,0}$ and complement it by involutions to reach the maximal number $\ell(r,0)$ for $\mathcal S(PI_{r,0})$, checking whether the connectivity $\pi_0(\mathcal S(PI_{r,0}))=1$ is not violated. 

We can repeat the last step several times if the condition $\pi_0(\mathcal S(PI_{r,0}))=1$ still holds.

{\it Step 2.} We continue to look on $\pi_0(\mathcal S)=1$ and $|\mathfrak b(PI_{r,0})|=r-1$. We find a standard subgroup $\mathcal S^{\bf p,0}\subset \mathcal S$ and complement it by involutions to reach the maximal number  $\ell(r,0)$ of involutions in $PI_{r,0}$. Note that in this case the basis vector $z_r\in\mathfrak B_{r,0}$ is not appear in $PI_{r,0}$ by requirement $|\mathfrak b(PI_{r,0})|=r-1$. Then we pass to a smaller standard subgroup $\mathcal S^{\bf p-1,0}\subset\mathcal S^{\bf p,0}$ and complement it to the maximal number $\ell(r,0)$ for $\mathcal S(PI_{r,0})$. The connectivity $\pi_0(\mathcal S(PI_{r,0}))=1$ must not be violated. 

In most cases it will be a simple step back from $(r,0)$ to $(r-1,0)$ as, for example, for reduction from $PI_{4,0}$ to $PI_{3,0}$.

{\it Step 3.} Next we check $\pi_0(\mathcal S)=2$ and $\mathcal S(PI_{r,0})=\mathcal S_{(1)}(PI_{r_1,0})\times\mathcal S_{(2)}(PI_{r_2,0})$, with $r=r_1+r_2$. The pairs $(r_1,0)$ and $(r_2,0)$ satisfy~\eqref{sum1} and~\eqref{sum2}.   After finding all possible pairs satisfying~\eqref{sum1} and~\eqref{sum2} we apply steps 1 and 2 to each group $\mathcal S_{(1)}(PI_{r_1,0})$ and $\mathcal S_{(2)}(PI_{r_2,0})$. Note that by proceeding in the increasing order of $r=1,2,\ldots,16$ the groups $\mathcal S_{(1)}(PI_{r_1,0})$ and $\mathcal S_{(2)}(PI_{r_2,0})$ are already known.
If needed, we can proceed to higher number of connected components.

The equivalence of the groups constructed in the previous three steps is summarized in the following proposition.

\begin{lemma}\label{lem:3567}
If $r=3+8k, 5+8k, 6+8k,7+8k$ for $k\geq 0$, then
sets $PI_{r,0}\in\mathbb{PI}^{M}_{r,0}$ satisfying $\pi_0(\mathcal S(PI_{r,0}))=1$ and $|\mathfrak b(PI_{r,0})|=r$ are always of $(T2)$-type.
\end{lemma}

\begin{proof}
We start from $r=3+8k$. For the case $r=3$ there is only one type $T_2$ involution. Let $k\geq 1$ and assume, by contrary, that there is a $(T1)$-type set $PI_{r,0}\in \mathbb{PI}_{r,0}^{M}$. We have $\ell(r,0)=\ell(3+8k,0)=1+4k$.
The standard subgroup $\mathcal S^{\bf p,0}\subset\mathcal S(PI_{r,0})$, ${\bf p}=1+4k$, does not contain $z_r$, since $r$ is odd. Let $p_1,\ldots,p_{4k}$ be involutions generating $\mathcal S^{\bf p,0}$, then $z_r\in \mathfrak b(p_{1+4k})$, where $p_{1+4k}$ will be the last type $T_1$ involution. It implies
$$
\{p_1,\ldots,p_{4k},\  z_r\cdot p_{1+4k}\}\in \mathbb{PI}_{r-1,0}^{M}.
$$
This contradicts to $\ell(r-1,0)=\ell(2+8k,0)=\ell(3+8k,0)-1=\ell(r,0)-1$.

The arguments for the cases $r=5+8k$, and $r=7+8k$ are similar to the case $r=3+8k$.

Let $r=6+8k$. We assume that there is a $(T1)$-type set $PI_{r,0}\in \mathbb{PI}_{r,0}^{M}$. We have $\ell(r,0)=\ell(6+8k,0)=3+4k$.
The standard subgroup $\mathcal S^{\bf p,0}\subset\mathcal S(PI_{r,0})$, ${\bf p}=3+4k$, contains $z_r$. Let $p_1,\ldots,p_{2+4k}$ be involutions generating $\mathcal S^{\bf p,0}$, where we can assume that $z_r\in\mathfrak b(p_{2+4k})$ and $p_{3+4k}\in PI_{6+8k,0}$ is the last type $T_1$ involution. 
\begin{itemize}
\item[(1)] If $z_r\notin\mathfrak b(p_{3+4k})$, then 
$$
\{p_1,\ldots,p_{1+4k},\  z_r\cdot p_{2+4k},\ p_{3+4k}\}\in \mathbb{PI}_{r-1,0}^{M}.
$$
This contradicts to $\ell(r-1,0)=\ell(5+8k,0)=\ell(6+8k,0)-1=\ell(r,0)-1$.
\item[(2)] If $z_r\in\mathfrak b(p_{3+4k})$, then we replace $p_{3+4k}\in PI_{r,0}$ by another type $T_1$ involution $\tilde p_{3+4k}=p_{2+4k}p_{3+4k}$ Denote by $\widetilde{PI_{r,0}}$ the set which contains all the involutions from $PI_{r,0}$ and where $p_{3+4k}$ is replaced by the new involution $\tilde p_{3+4k}$. In this case $z_r\notin\mathfrak b(\tilde p_{3+4k})$ and the situation is reduced to the previous step (1). Note that the group $\mathcal S(PI_{r,0})$ is equivalent to $\mathcal S(\widetilde{PI_{r,0}})$.
\end{itemize}

We also note that for $r=3+8k$ and $r=7+8k$ the volume forms $\omega_r=\prod_{i=1}^rz_i$ which are  type $T_2$ involutions can be included to $PI_{r,0}$. It justifies the name $(T2)$-type set of $PI_{r,0}$ for  $r=3+8k$ and $r=7+8k$.
\end{proof}


\subsection{Construction of connected groups $\mathcal S\in\mathbb{S}^M_{r,0}$ for $r\in\{3,\ldots,16\}$}


In Table~\ref{connected} we collect the nonisomorphic connected groups $\mathcal S\in\mathbb{S}^M_{r,0}$ for $r\in\{3,\ldots, 16\}$; that is $\pi_0(\mathcal S)=1$. Note that if there is a group $\mathcal S=\mathcal S(PI_{r,0})\in\mathbb{S}^M_{r,0}$ whose generating set $PI_{r,0}$ has the signature 
$$
PI_{r,0}=(\ell(r,0), (Tk,\bullet,\bullet), |\mathfrak b(PI_{r,0})|),\ \ k=1,2,
$$
where
$$ 
\ell(r,0)=\ell(r-1,0),\ \ |\mathfrak b(PI_{r,0})|=r-1,
$$
then $\mathcal S(PI_{r,0})=\mathcal S(PI_{r-1,0})\in\mathbb{S}^M_{r-1,0}$ with
$$
PI_{r-1,0}=(\ell(r-1,0), (Tk,\bullet,\bullet), r-1),\ \ k=1,2,
$$
as for instance in Table~\ref{connected} for $r=4,8,9$ and some other values of $r$.

\begin{notation}
We write $\theta_{\overline{i,j}}$ to indicate that product in $\theta$ starts from $z_i$ and ends with $z_j$ containing all $z_k$ for odd $k$ between $i$ and $j$. We have
$
|\mathfrak b(\theta_{\overline{i,j}})|=\frac{j-i}{2}+1$.
\end{notation} 

\begin{table}[!htb]
    \caption{Connected groups in $\mathbb{S}^M_{r,0}$ for $r=3,\ldots,16$}
    \begin{minipage}{.5\linewidth}
      \centering
\scalebox{0.62}[0.62]{%
\begin{tabular}{|c|c|c|c|c|c|c|c|c|c|c|c|c|c|c|c|c|c|c|c|c||}\hline
$r$ & Signatures & $PI_{r,0}$  \\\hline

12  & $\begin{array}{ll}
&\mathcal S^{(1)}_{12}=\big(5,(T1,\pi),12\big)
\\
\\\hline
\\
&\mathcal S^{(2)}_{12}=\big(5,(T2,\theta),12\big)
\\
\\\hline
\\
&\mathcal S^{(3)}_{12}=\big(5,(T2,\pi),11\big)=\mathcal S_{11}
\end{array}$ & 
$\begin{array}{ll}
\mathcal S^{(1)}_{12}=\begin{cases}
&p_1=\pi_{1,2}\pi_{3,4}
\\
&p_2=\pi_{1,2}\pi_{5,6}
\\
&p_3=\pi_{1,2}\pi_{7,8}
\\
&p_4=\pi_{1,2}\pi_{9,10}
\\
&p_5=\pi_{1,2}\pi_{11,12}
\end{cases}
\\
\mathcal S^{(2)}_{12}=\begin{cases}
&p_1=\pi_{1,2}\pi_{3,4}
\\
&p_2=\pi_{1,2}\pi_{5,6}
\\
&p_3=\pi_{1,2}\pi_{7,8}
\\
&p_4=\pi_{1,2}\pi_{9,10}
\\
&q=\theta_{\overline{1,9}}\pi_{11,12}
\end{cases}
\end{array}$
\\\hline\hline

11  & 
$\mathcal S_{11}=\big(5,(T2,\pi),11\big)$ & 
$\begin{array}{ll}
&p_1=\pi_{1,2}\pi_{3,4}
\\
&p_2=\pi_{1,2}\pi_{5,6}
\\
&p_3=\pi_{1,2}\pi_{7,8}
\\
&p_4=\pi_{1,2}\pi_{9,10}
\\
&q=\pi_{1,2}z_{11}
\end{array}$  
\\\hline\hline

10 & $\begin{array}{ll}
&\mathcal S^{(1)}_{10}=\big(4,(T1,\pi),10\big)
\\\hline
&\mathcal S^{(2)}_{10}=\mathcal S^{(1)}_{9}
\\\hline
&\mathcal S^{(3)}_{10}=\mathcal S^{(2)}_{9}=\mathcal S^{(1)}_{8}
\\\hline
&\mathcal S^{(4)}_{10}=\mathcal S^{(3)}_{9}=\mathcal S^{(2)}_{8}=\mathcal S_{7}
\end{array}$ & 
$\mathcal S^{(1)}_{10}=\begin{cases}
&p_1=\pi_{1,2}\pi_{3,4}
\\
&p_2=\pi_{1,2}\pi_{5,6}
\\
&p_3=\pi_{1,2}\pi_{7,8}
\\
&p_4=\pi_{1,2}\pi_{9,10}
\end{cases}$
\\\hline\hline
9 & $\begin{array}{ll}
&\mathcal S^{(1)}_9=\big(4,(T2,\pi),9\big)
\\\hline
&\mathcal S^{(2)}_9=\mathcal S^{(1)}_8
\\\hline
&\mathcal S^{(3)}_9=\mathcal S^{(2)}_8=\mathcal S_7
\end{array}$ & 
$\mathcal S^{(1)}_9=\begin{cases}
&p_1=\pi_{1,2}z_3z_4
\\
&p_2=\pi_{1,2}z_5z_6
\\
&p_3=\pi_{1,2}z_7z_8
\\
&q=\pi_{1,2}z_9
\end{cases}$  \\\hline\hline

8 & $\begin{array}{ll}
&\mathcal S^{(1)}_8=\big(4,(T1,\theta),8\big)
\\
\\\hline
\\
&\mathcal S^{(2)}_8=\mathcal S_7
\end{array}$ & 
$\mathcal S^{(1)}_8=\begin{cases}
&p_1=\pi_{1,2}z_3z_4
\\
&p_2=\pi_{1,2}z_5z_6
\\
&p_3=\pi_{1,2}z_7z_8
\\
&p_4=\theta_{\overline{1,7}}
\end{cases}$  \\\hline\hline

7 &  $\mathcal S_7=\big(4,(T2,\pi,\theta),7\big)$ & 
$\begin{array}{ll}
&p_1=\pi_{1,2}z_3z_4
\\
&p_2=\pi_{1,2}z_5z_6
\\
&p_3=\theta_{\overline{1,5}}z_7
\\
&q=\pi_{1,2}z_7
\end{array}$ \\\hline\hline

6  & 
$\mathcal S_6=\big(3,(T2,\theta),6\big)$ & 
$\begin{array}{ll}
&p_1=\pi_{1,2}z_3z_4
\\
&p_2=\pi_{1,2}z_5z_6
\\
&q=\theta_{\overline{1,5}}
\end{array}$  \\\hline\hline

5  & $\mathcal S_5=\big(2,(T2,\theta),5\big)$ & 
$\begin{array}{ll}
&p=\pi_{1,2}z_3z_4
\\
&q=\theta_{\overline{1,5}}=z_1z_3z_5
\end{array}$
\\\hline\hline

4  & 
$\begin{array}{ll}
&\mathcal S^{(1)}_4=\big(1,(T1),4\big)
\\\hline
&\mathcal S^{(2)}_4=\big(1,(T2,\pi),3\big)=\mathcal S_3
\end{array}$ & 
$\begin{array}{ll}
&p=\pi_{1,2}z_3z_4
\\\hline
&q=\pi_{1,2}z_3
\end{array}$  \\\hline
\hline
3 & $\mathcal S_3=\big(1,(T2,\pi),3\big)$ & $q=\pi_{1,2}z_3$  \\\hline
\end{tabular}
}
    \end{minipage}%
    \begin{minipage}{.5\linewidth}
      \centering
\scalebox{0.56}[0.56]{%
 \begin{tabular}{|c|c|c|c|c|c|c|c|c|c|c|c|c|c|c|c|c|c|c|c|c||}\hline
$r$ & Signatures & $PI_{r,0}$  \\\hline

16 & $\begin{array}{ll}
&\mathcal S^{(1)}_{16}=\big(8,(T1,\theta),16\big)
\\
\\\hline
\\
&\mathcal S^{(2)}_{16}=\mathcal S_{15}
\end{array}$ & $
\mathcal S^{(1)}_{16}=
\begin{cases}
&p_1=\pi_{1,2}\pi_{3,4}
\\
&p_2=\pi_{1,2}\pi_{5,6}
\\
&p_3=\pi_{1,2}\pi_{7,8}
\\
&p_4=\pi_{1,2}\pi_{9,10}
\\
&p_5=\pi_{1,2}\pi_{11,12}
\\
&p_6=\pi_{1,2}\pi_{13,14}
\\
&p_7=\pi_{1,2}\pi_{15,16}
\\
&p_8=\theta_{\overline{1,13}}z_{15}
\end{cases}
$  \\\hline\hline

15 & $
\mathcal S_{15}=\big(8,(T2,\pi),15\big)
$ & 
$\begin{array}{ll}
&p_1=\pi_{1,2}\pi_{3,4}
\\
&p_2=\pi_{1,2}\pi_{5,6}
\\
&p_3=\pi_{1,2}\pi_{7,8}
\\
&p_4=\pi_{1,2}\pi_{9,10}
\\
&p_5=\pi_{1,2}\pi_{11,12}
\\
&p_6=\pi_{1,2}\pi_{13,14}
\\
&p_7=\theta_{\overline{1,13}}z_{15}
\\
&q=\pi_{1,2}z_{15}
\end{array}$  \\\hline\hline

14 &  $\begin{array}{ll}
&\mathcal S^{(1)}_{14}=\big(7,(T2,\theta),14\big)
\\
\\\hline
\\
&\mathcal S^{(2)}_{14}=\big(7,(T2,\pi,\theta),14\big)
\end{array}$ & 
$\begin{array}{ll}
&p_1=\pi_{1,2}\pi_{3,4}
\\
&p_2=\pi_{1,2}\pi_{5,6}
\\
&p_3=\pi_{1,2}\pi_{7,8}
\\
&p_4=\pi_{1,2}\pi_{9,10}
\\
&p_5=\pi_{1,2}\pi_{11,12}
\\
&p_6=\pi_{1,2}\pi_{13,14}
\\
&q=\theta_{\overline{1,13}}
\\\hline
&p_1=\pi_{1,2}\pi_{3,4}
\\
&p_2=\pi_{1,2}\pi_{5,6}
\\
&p_3=\pi_{1,2}\pi_{7,8}
\\
&p_4=\pi_{1,2}\pi_{9,10}
\\
&p_5=\pi_{1,2}\pi_{11,12}
\\
&p_6=\theta_{\overline{1,11}}\pi_{13,14}
\\
&q=\pi_{1,2}z_{14}
\end{array}$  
\\\hline\hline

13  & 
$\begin{array}{ll}
&\mathcal S^{(1)}_{13}=\big(6,(T2,\pi),13\big)
\\
\\
\\\hline
\\
\\
&\mathcal S^{(2)}_{13}=\big(6,(T2,\theta),13\big)
\\
\\
\\
\\\hline
\\
\\
&\mathcal S^{(3)}_{13}=\big(6,(T2,\pi,\theta),13\big)
\end{array}$ & 
$\begin{array}{ll}
&p_1=\pi_{1,2}\pi_{3,4}
\\
&p_2=\pi_{1,2}\pi_{5,6}
\\
&p_3=\pi_{1,2}\pi_{7,8}
\\
&p_4=\pi_{1,2}\pi_{9,10}
\\
&p_5=\pi_{1,2}\pi_{11,12}
\\
&q=\pi_{1,2}z_{13}
\\\hline
&p_1=\pi_{1,2}\pi_{3,4}
\\
&p_2=\pi_{1,2}\pi_{5,6}
\\
&p_3=\pi_{1,2}\pi_{7,8}
\\
&p_4=\pi_{1,2}\pi_{9,10}
\\
&p_5=\pi_{1,2}\pi_{11,12}
\\
&q=\theta_{\overline{1,13}}
\\\hline
&p_1=\pi_{1,2}\pi_{3,4}
\\
&p_2=\pi_{1,2}\pi_{5,6}
\\
&p_3=\pi_{1,2}\pi_{7,8}
\\
&p_4=\pi_{1,2}\pi_{9,10}
\\
&p_5=\theta_{\overline{1,9}}z_{11}z_{12}z_{13}
\\
&q=\pi_{1,2}z_{12}
\end{array}$  \\ \hline

\end{tabular}
}
    \end{minipage} 
    \label{connected}
\end{table}

We explain construction only for $r=7$, since it is the most illustrative.

The standard subgroup $\mathcal S^{\bf 3,\bf 0}\subset \mathcal S$ is generated by two involutions
\begin{equation}\label{eq:N6}
p_1=\pi_{1,2}\pi_{3,4},\quad p_2=\pi_{1,2}\pi_{5,6}.
\end{equation}
We need to add two involutions since $\ell(7,0)=4$, at least one of which must contain~$z_7$. We observe that the products $\pi_{1,2}z_7$ and $\theta_{\overline{1,5}} z_7=z_1z_3z_5z_7$ are both involutions commuting with generators~\eqref{eq:N6} and with each other. We append them both to reach $\ell(7,0)=4$. The reductions to $|\mathfrak b(PI_{7,0})|=6$ is not possible due to $\ell(6,0)<\ell(7,0)$.


\subsection{Constructions of disconnected groups $\mathcal S\in\mathbb{S}^M_{r,0}$ for $r\in\{10,\ldots, 16\}$}


We show in Table~\ref{disconnected} the disconnected groups with $\pi_0(\mathcal S_{r,0})=2$ for $r\in\{10,\ldots, 16\}$. 

We explain in details only the case $r=11$.
To construct the disconnected subgroup $\mathcal S^{(1)}_{11}=\mathcal S_{(1)}^{(1)}\times \mathcal S_{(2)}^{(1)}$ corresponding to the $\mathbb Z_2$-graded tensor product of the Clifford algebras $\Cl_{11,0}\cong \Cl_{8,0}\hat{\otimes}\Cl_{3,0}$ we start from the standard subgroup $\mathcal S_{(1)}^{\bf 4,0}\subset \mathcal S_{(1)}^{(1)}$ generated by
\begin{equation}\label{eq:N8}
p_1=\pi_{1,2}\pi_{3,4},\quad p_2=\pi_{1,2}\pi_{5,6},\quad p_3=\pi_{1,2}\pi_{7,8}.
\end{equation}
and add type $T_1$ involution $\theta=\theta_{\overline{1,7}}=z_1z_3z_5z_7$. The group $\mathcal S_{(1)}^{(1)}$ has the following signature $(4,(T1,\theta),8)$. Then the signature of $\mathcal S^{(1)}_{(2)}=\{{\bf 1}, \pi_{9,10}z_{11}\}$ is $(1,(T2,\pi),3)$ with $\pi=\pi_{9,10}z_{11}$.

To obtain $\mathcal S^{(2)}_{11}=\mathcal S_{(1)}^{(2)}\times \mathcal S_{(2)}^{(2)}$ corresponding to the $\mathbb Z_2$-graded tensor product of the Clifford algebras $\Cl_{11,0}\cong \Cl_{7,0}\hat{\otimes}\Cl_{4,0}$ we consider standard subgroup $\mathcal S_{(1)}^{\bf 3,0}\subset \mathcal S_{(1)}^{(2)}$ generated by~\eqref{eq:N6} and add type $T_1$ involution $\theta=\theta_{\overline{1,7}}$ and type $T_2$ involution $\pi=\pi_{1,2}z_7$. The group $\mathcal S_{(1)}^{(2)}$ obtain the signature $(4,(T2,\pi,\theta),7)$. Then $\mathcal S_{(2)}^{(2)}=\{{\bf 1}, \pi_{8,9}\pi_{10,11}\}$ has the signature $(1,(T1),4)$.

\begin{table}[!htb]
    \caption{Disconnected groups in $\mathbb{S}^M_{r,0}$ for $r=10,\ldots,16$}
    \begin{minipage}{.5\linewidth}
      \centering
\scalebox{0.44}[0.48]{%
\begin{tabular}{|c|c|c|c|c|c|c|c|c|c|c|c|c|c|c|c|c|c|c|c|c||}\hline
$r$ & Signatures & $PI$  \\\hline

12 & $\begin{array}{ll}
&\mathcal S^{(1)}_{12}=\big(4,(T1,\theta),8\big)\times \big(1,(T1),4\big)
\\\hline
\\
&\mathcal S^{(2)}_{12}=\big(3,(T1,\theta),7\big)\times \big(2,(T2,\pi),5\big)
\\\hline
\\
&\mathcal S^{(3)}_{12}=\big(3,(T2,\theta),6\big)\times \big(2,(T1),6\big)
\\\hline
\\
&\mathcal S^{(4)}_{12}=\mathcal S^{(1)}_{11}
\\\hline
\\
&\mathcal S^{(5)}_{12}=\mathcal S^{(2)}_{11}
\end{array}$ & $
\begin{array}{ll}
&\mathcal S^{(1)}_{12}=
\begin{cases}
(p_1)_1=\pi_{1,2}\pi_{3,4}\ & (p_1)_2=\pi_{9,10}\pi_{11,12}
\\
(p_2)_1=\pi_{1,2}\pi_{5,6}
\\
(p_3)_1=\pi_{1,2}\pi_{7,8}
\\
(p_4)_1=\theta_{\overline{1,7}}
\end{cases}
\\
&\mathcal S^{(2)}_{12}=
\begin{cases}
(p_1)_1=\pi_{1,2}\pi_{3,4}\ & (p_1)_2=\pi_{8,9}\pi_{10,11}
\\
(p_2)_1=\pi_{1,2}\pi_{5,6}\ & (q)_2=\pi_{8,9}z_{12}
\\
(p_3)_1=\theta_{\overline{1,7}}
\end{cases}
\\
&\mathcal S^{(3)}_{12}=
\begin{cases}
(p_1)_1=\pi_{1,2}\pi_{3,4}\ & (p_1)_2=\pi_{7,8}\pi_{9,10}
\\
(p_2)_1=\pi_{1,2}\pi_{5,6}\ & (p_1)_2=\pi_{7,8}\pi_{11,12}
\\
(q)_1=\theta_{\overline{1,5}}
\end{cases}
\end{array}
$  \\\hline\hline

11 & $
\begin{array}{ll}
&\mathcal S^{(1)}_{11}=\big(4,(T1,\theta),8\big)\times \big(1,(T2,\pi),3\big)
\\
\\\hline
\\
&\mathcal S^{(2)}_{11}=\big(4,(T2,\pi,\theta),7\big)\times \big(1,(T1),4\big)
\end{array}
$ & 
$\begin{array}{ll}
\begin{array}{llllll}
&(p_1)_{(1)}=\pi_{1,2}\pi_{3,4},\ & (q)_{(2)}=\pi_{9,10}z_{11}
\\
&(p_2)_{(1)}=\pi_{1,2}\pi_{5,6},
\\
&(p_3)_{(1)}=\pi_{1,2}\pi_{7,8},
\\
&(p_4)_{(1)}=\theta_{\overline{1,7}},
\end{array}
\\\hline
\begin{array}{llllll}
&(p_1)_{(1)}=\pi_{1,2}\pi_{3,4},\ & (p_1)_{(2)}=\pi_{8,9}\pi_{10,11}
\\
&(p_2)_{(1)}=\pi_{1,2}\pi_{5,6},
\\
&(p_3)_{(1)}=\theta_{\overline{1,7}},
\\
&(q)_{(1)}=\pi_{1,2}z_7,
\end{array}
\end{array}$  \\\hline\hline

10 &  $\begin{array}{ll}
&\mathcal S^{(1)}_{10}=\big(3,(T1,\theta),7\big)\times \big(1,(T2,\pi),3\big)
\\
\\\hline
\\
&\mathcal S^{(2)}_{10}=\big(3,(T2,\theta),6\big)\times \big(1,(T1),4\big)
\end{array}$ & 
$\begin{array}{ll}
\begin{array}{llllll}
&(p_1)_{(1)}=\pi_{1,2}\pi_{3,4},\ & (q)_{(2)}=\pi_{8,9}z_{10}
\\
&(p_2)_{(1)}=\pi_{1,2}\pi_{5,6},
\\
&(p_3)_{(1)}=\theta_{\overline{1,7}},
\end{array}

\\\hline
\begin{array}{llllll}
&(p_1)_{(1)}=\pi_{1,2}\pi_{3,4},\ & (p_1)_{(2)}=\pi_{7,8}\pi_{9,10}
\\
&(p_2)_{(1)}=\pi_{1,2}\pi_{5,6},
\\
&(q)_{(1)}=\theta_{\overline{1,5}},
\end{array}\end{array}$  
\\\hline
\end{tabular}

}
    \end{minipage}%
    \begin{minipage}{.5\linewidth}
      \centering
\scalebox{0.45}[0.45]{%
\begin{tabular}{|c|c|c|c|c|c|c|c|c|c|c|c|c|c|c|c|c|c|c|c|c||}\hline
$r$ & Signatures & $PI$  \\\hline

16 & $\begin{array}{ll}
\mathcal S^{(1)}_{16}=\big(4,(T1,\theta),8\big)\times \big(4,(T1,\theta),8\big)
\\
\\\hline
\\
\mathcal S^{(2)}_{16}=\mathcal S_{15}
\end{array}$ & 
$
\mathcal S^{(1)}_{16}=
\begin{cases}
(p_1)_1=\pi_{1,2}\pi_{3,4},\ & (p_1)_2=\pi_{9,10}\pi_{11,12}
\\
(p_2)_1=\pi_{1,2}\pi_{5,6},\ & (p_2)_2=\pi_{9,10}\pi_{13,14}
\\
(p_3)_1=\pi_{1,2}\pi_{7,8},\ & (p_3)_2=\pi_{9,10}\pi_{15,16}
\\
(p_4)_1=\theta_{\overline{1,7}}, & (p_4)_2=\theta_{\overline {9,15}}
\end{cases}
$  \\\hline\hline

15 & $
\mathcal S_{15}=\big(4,(T1,\theta),8\big)\times \big(4,(T2,\pi,\theta),7\big)
$ & $
\begin{array}{llllll}
&(p_1)_{(1)}=\pi_{1,2}\pi_{3,4},\ & (p_1)_{(2)}=\pi_{9,10}\pi_{11,12}
\\
&(p_2)_{(1)}=\pi_{1,2}\pi_{5,6},\ & (p_2)_{(2)}=\pi_{9,10}\pi_{13,14}
\\
&(p_3)_{(1)}=\pi_{1,2}\pi_{7,8},\ & (p_3)_{(2)}=\theta_{\overline {9,15}}
\\
&(p_4)_{(1)}=\theta_{\overline{1,7}},\ & (q)_{(2)}=\pi_{9,10}z_{15}
\end{array}
$  \\\hline\hline

14 & $
\begin{array}{ll}
&\mathcal S^{(1)}_{14}=\big(4,(T1,\theta),8\big)\times \big(3,(T2,\theta),6\big)
\\
\\\hline
\\
&\mathcal S^{(2)}_{14}=\big(4,(T2,\pi,\theta),7\big)\times \big(3,(T1,\theta),7\big)
\end{array}
$ & 
$\begin{array}{ll}
\begin{array}{llllll}
&(p_1)_{(1)}=\pi_{1,2}\pi_{3,4},\ & (p_1)_{(2)}=\pi_{9,10}\pi_{11,12}
\\
&(p_2)_{(1)}=\pi_{1,2}\pi_{5,6},\ &(p_2)_{(2)}=\pi_{9,10}\pi_{13,14}
\\
&(p_3)_{(1)}=\pi_{1,2}\pi_{7,8},\ & (q)_{(2)}=\theta_{\overline {9,13}}
\\
&(p_4)_{(1)}=\theta_{\overline{1,7}},
\end{array}
\\\hline
\begin{array}{llllll}
&(p_1)_{(1)}=\pi_{1,2}\pi_{3,4},\ & (p_1)_{(2)}=\pi_{8,9}\pi_{10,11}
\\
&(p_2)_{(1)}=\pi_{1,2}\pi_{5,6},\ &(p_2)_{(2)}=\pi_{8,9}\pi_{12,13}
\\
&(p_3)_{(1)}=\theta_{\overline{1,7}}, & (p_3)_{(2)}=\theta_{\overline {9,13}}z_{14}
\\
&(q_4)_{(1)}=\pi_{1,2}z_7,
\end{array}
\end{array}$  \\\hline\hline

13 &  $\begin{array}{ll}
&\mathcal S^{(1)}_{13}=\big(4,(T1,\theta),8\big)\times \big(2,(T2,\pi),5\big)
\\
\\\hline
\\
&\mathcal S^{(2)}_{13}=\big(4,(T2,\pi,\theta),7\big)\times \big(2,(T1),6\big)
\\
\\\hline
\\
&\mathcal S^{(3)}_{13}=\big(3,(T1,\theta),7\big)\times \big(3,(T2,\theta),6\big)
\end{array}$ & 
$\begin{array}{ll}
\begin{array}{llllll}
&(p_1)_{(1)}=\pi_{1,2}\pi_{3,4},\ &(p_1)_{(2)}=\pi_{9,10}\pi_{11,12},
\\
&(p_2)_{(1)}=\pi_{1,2}\pi_{5,6},& (q)_{(2)}=\pi_{9,10}z_{13}
\\
&(p_3)_{(1)}=\pi_{1,2}\pi_{7,8},
\\
&(p_4)_{(1)}=\theta_{\overline{1,7}},
\end{array}
\\\hline
\begin{array}{llllll}
&(p_1)_{(1)}=\pi_{1,2}\pi_{3,4},\ & (p_1)_{(2)}=\pi_{8,9}\pi_{10,11}
\\
&(p_2)_{(1)}=\pi_{1,2}\pi_{5,6},\ & (p_2)_{(2)}=\pi_{8,9}\pi_{12,13}
\\
&(p_3)_{(1)}=\theta_{\overline{1,7}},
\\
&(q)_{(1)}=\pi_{1,2}z_7,
\end{array}
\\\hline
\begin{array}{llllll}
&(p_1)_{(1)}=\pi_{1,2}\pi_{3,4},\ & (p_1)_{(2)}=\pi_{8,9}\pi_{10,11}
\\
&(p_2)_{(1)}=\pi_{1,2}\pi_{5,6},\ & (p_2)_{(2)}=\pi_{8,9}\pi_{12,13}
\\
&(p_3)_{(1)}=\theta_{\overline{1,7}},\ &(q)_{(2)}=z_{8}z_{10}z_{12}
\end{array}
\end{array}$  
\\\hline
\end{tabular}
}
    \end{minipage} 
    \label{disconnected}
\end{table}

\begin{proposition}\label{th:s=1}
Table~\ref{connected} and Table~\ref{disconnected} are the same for $H$-type Lie algebras $\mathfrak n_{r,1}$, $r\in\{3,\ldots,16\}$.
\end{proposition}
\begin{proof}
For $s=1$, the negative basis vector plays no role in forming involutions, see Definition~\ref{Def:S}. 
\end{proof}

\begin{table}[h]
\caption{Number of non-equivalent groups}
\scalebox{0.7}[0.7]{%
\begin{tabular}{|c|c|c|c|c|c|c|c|c|}
\hline
$r$                  &1&2&3&4&5&6&7&8  \\\hline
$\pi_0(\mathcal S)=1$ &0&0&1&2&1&1&1&2 \\\hline
$\pi_0(\mathcal S)=2$&0&0&0&0&0&0&0&0\\\hline\hline
 $r$  &9&10&11&12&13&14&15&16  \\\hline
$\pi_0(\mathcal S)=1$    &3&4&1&3&3&2&1&2 \\\hline
$\pi_0(\mathcal S)=2$    &0&2&2&5&3&2&1&2
\\\hline
\end{tabular}\label{tab:pNgroups}
}
\end{table}


\subsection{Construction of connected groups $\mathcal S\in\mathbb {S}^M_{r,s}$ for $0<r+s<16$}\label{s>1}


We show in Table~\ref{tab:0<r<9, s=0 to 8}
the possible maximal sets $PI_{r,s}\in\mathbb{PI}^{M}_{r,s}$
with {\bf{$\pi_{0}(PI_{r,s})=1$}} for some $r+s\leq 16$.
The different sets of involutions are determined based on the data of $\mathbb{PI}^{M}_{r+s,0}$ by a recurrent procedure.
Note that $|\mathfrak{b}(PI_{r,s})|=|\mathfrak{b}^{+}(PI_{r,s})|+|\mathfrak{b}^{-}(PI_{r,s})|\leq r+s$ because some of the basis vectors $z_j$ are not used when $s\geq 0$, as for instance the $(T2)$-type set in $\mathbb{PI}^M_{4,0}$ consists only of product of three vectors $z_j$.
We use the signature
\[
\ell(r,s), ~\text{Type~ ($T1$ or $T2$)}, ~(|\mathfrak{b}^{+}(PI)|,|\mathfrak{b}^{-}(PI)|)
\]
to indicate the non-equivalent sets of involutions $\mathbb{PI}^M_{r,s}$ in Table~\ref{tab:0<r<9, s=0 to 8}. 

The values $\ell(r,s)$ of the maximal number of involutions for $r+s\leq 16$ are collected in Table~\ref{16}. 

\begin{table}[th] 
\caption{{\small The value $\ell(r,s)$ for $r+s\leq 16$}}
\scalebox{0.7}[0.7]{%

\begin{tabular}{|c||c|c|c|c|c|c|c|c|c|c|c|c|c|c|c|c|c|c|c|c||}\hline
16& $8$ &  &&& $$ & $$  &$$ & $$& $$& $$ & $$ & $$ & $$ & $$ & $$ &$$ & $$\\\hline

15& $7$ & $7$ &&& $$ & $$  &$$ & $$& $$& $$ & $$ & $$ & $$ & $$ & $$ &$$ & $$\\\hline

14& $6$ & $7$ & $7$ && $$ & $$  &$$ & $$& $$& $$ & $$ & $$ & $$ & $$ & $$ &$$ & $$\\\hline

13& $5$ & $6$ & $7$ & $8$ & $$ & $$  &$$ & $$& $$& $$ & $$ & $$ & $$ & $$ & $$ &$$ & $$\\\hline

12& $5$ & $6$ & $7$ & $8$ & $8$ & $$  &$$ & $$& $$& $$ & $$ & $$ & $$ & $$ & $$ &$$ & $$\\\hline

11& $4$ & $5$ & $6$ & $7$ & $7$ & $7$  &$$ & $$& $$& $$ & $$ & $$ & $$ & $$ & $$ &$$ & $$\\\hline

10& $4$ & $5$ & $5$ & $6$ & $6$ & $7$  &$7$ & $$& $$& $$ & $$ & $$ & $$ & $$ & $$ &$$ & $$\\\hline

9  & $4$ & $4$ & $4$ & $5$ & $5$ & $6$  &$7$ & $8$& $$& $$ & $$ & $$ & $$ & $$ & $$ &$$ & $$\\\hline

8 & $4$ & $4$ & $4$ & $5$ & $5$ & $6$  &$7$ & $8$& $8$& $$ & $$ & $$ & $$ & $$ & $$ &$$ & $$\\\hline

7 & $3$ & $3$ & $3$ & $4$ & $4$ & $5$  &$6$ & $7$& $7$& $7$ & $$ & $$ & $$ & $$ & $$ &$$ & $$\\\hline

6  & $2$ & $3$ & $3$ & $4$ & $4$ & $5$  &$5$ & $6$& $6$& $7$ & $7$ & $$ & $$ & $$ & $$ &$$ & $$\\\hline

5  & $1$ & $2$ & $3$ & $4$ & $4$ & $4$  &$4$ & $5$& $5$& $6$ & $7$ & $8$ & $$ & $$ & $$ & $$  & $$\\\hline

4  & $1$ & $2$ & $3$ & $4$ & $4$ & $4$  &$4$ & $5$& $5$& $6$ & $7$ & $8$ & $8$ & $$ &$$& $$ & $$\\\hline

3 & $0$ & $1$ & $2$ & $3$ & $3$ & $3$  &$3$ & $4$& $4$& $5$ & $6$ & $7$ & $7$ & $7$ & $$ &$$&  \\\hline

2 & $0$ & $1$ & $1$ & $2$ & $2$ & $3$  &$3$ & $4$& $4$& $5$ & $5$ & $6$ & $6$ & $7$ & $7$ &$$&$$\\\hline

1 & $0$ & $0$ & $0$ & $1$ & $1$ & $2$  &$3$ & $4$& $4$& $4$ & $4$ & $5$ & $5$ & $6$ & $7$ &$8$ & $$\\\hline

0 & $0$ & $0$ & $0$ & $1$ & $1$ & $2$  &$3$ & $4$& $4$& $4$ & $4$ & $5$ & $5$ & $6$ & $7$ &$8$ & $8$\\\hline\hline

s/r  & $0$ & $1$ & $2$ & $3$ & $4$ & $5$  &$6$ & $7$& $8$& $9$ & $10$ & $11$ & $12$ & $13$ & $14$ &$15$ & $16$\\\hline
\end{tabular}
}
\label{16}
\end{table}

We mentioned here the facts that allow us to complete Tables~\ref{tab:0<r<9, s=0 to 8} and~\ref{tab:0<r<9, s=10 to 16}.
\begin{itemize}
\item[(F1)] $\ell(r+s,0)\geq \ell(r,s)\geq \ell(r+s,0)-1$, see Table~\ref{16}.
\item[(F2)] Assume that $r\geq 2$, $s\geq 0$, and $\ell(r,s)=\ell(r-1,s)$. There is a $(T1)$-type set $PI\in\mathbb{PI}^{M}_{r,s}$, 
if and only if there is a $(T2)$-type set $PI'\in\mathbb{PI}^{M}_{r-1,s}$. Through the natural inclusion $\Cl_{r-1,s}\subset \,\Cl_{r,s}$,
we can regard $PI'\in\mathbb{PI}^{M}_{r,s}$. For instance the $(T1)$-type set of involutions $PI_{4,0}=\{z_1z_2z_3z_4\}\in\mathbb{PI}^{M}_{4,0}$ exists if and only if exists the $(T2)$-type set $PI'=\{z_1z_2z_3\}\in\mathbb{PI}^{M}_{3,0}$.
\item[(F3)] If $p=z_{i_{1}}\cdots z_{i_{a}}\in PI_{r,0}$
is an involution, where $z_{i_{k}}$, $k=1,\ldots,a$, are all positive basis vectors, then $p'= z_{i_{1}}\cdots z_{i_{l}}z^*_{i_{l+1}}\cdots z^*_{i_{k}}\in PI_{r',s'}$, $r'+s'=r$ is an involution where we replaced the even number of positive basis vectors $z_{i_{l+1}},\ldots, z_{i_{k}}$ by negative basis vectors $z^*_{i_{l+1}},\ldots, z^*_{i_{k}}$, see~\eqref{eqT1T2}.
\end{itemize}
To complete Tables~\ref{tab:0<r<9, s=0 to 8} and~\ref{tab:0<r<9, s=10 to 16} we perform the following
\begin{itemize}
\item[(G1)] We determine the cases $(r,0)$ for all $r>0$. This was done in Table~\ref{connected}.
\item[(G2)] For any given $(r,0)$, we determine all the equivalence classes of $PI_{r',s'}$ with $r'+s'=r$ successively for $r=3,4,\ldots$ as follows.
\item[(G3)] Let us assume that $(r',s')$ with $r'+s'=r$ are such that all equivalent sets of $PI_{r'-1,s'}$ and $PI_{r',s'-1}$ are already determined. There are possible cases which define the rest of the steps:
\begin{equation}\label{r-1 case}
\ell(r',s')\geq\ell(r'-1,s'),
\end{equation}
\begin{equation}\label{s-1 case}
\ell(r',s')\geq\ell(r',s'-1),
\end{equation}
\begin{equation}\label{equal case}
\ell(r',s')\leq\ell(r,0).
\end{equation}
\item[(G4)] If $\ell(r',s')\geq\ell(r'-1,s')$ then we include all classes of involutions $PI_{r'-1,s'}$ to $\mathbb{PI}^{M}_{r',s'}$ and check whether  $(F2)$ is applicable.
\item[(G5)] If $\ell(r',s')\geq\ell(r',s'-1)$, then we include all classes of involutions $PI_{r',s'-1}$ to $\mathbb{PI}^{M}_{r',s'}$. We remove the classes of equivalence of involutions which coincide after performing steps $(G4)$ and $(G5)$. 
\item[(G6)] If $\ell(r',s')<\ell(r,0)$, then all the involutions were included in steps $(G4)$ and $(G5)$. If $\ell(r',s')=\ell(r,0)$, then we check all $PI_{r,0}\in \mathbb{PI}^{M}_{r,0}$ with $r=r'+s'$, whether we can apply property $(F3)$ to the involutions $p\in PI_{r,0}$ to get new involutions in $PI_{r',s'}$ reaching the number $\ell(r',s')=\ell(r,0)$.
\end{itemize}

\begin{remark} 
	Note that we have only two possibilities to produce new positive involutions: when we make a step from $(r'-1,s')$ to $(r',s')$ by modifying type $T2$ involution to type $T1$ involution by $(F2)$ and when we gain some new involution by $(F3)$. The induction steps $(G4)-(G6)$ allow us to include all the classes of involutions $PI_{r',s'}\in \mathbb{PI}^{M}_{r',s'}$ with $r'+s'=r$. 
\end{remark}

\begin{table}[h] 
\caption{{\small Connected groups for $r+s\leq 16$. Part 1}}
\scalebox{0.5}[0.5]{%

\begin{tabular}{|c||c|c|c|c|c|c|c|c|c|c|c|c|c|c|c|c|c||}\hline

$16$&$8,T1,(0,16)$&&&&&&&&&
\\
\hline

$15$&$7,T1,(0,15)$&$7,T1,(0,15$)&&&&&&&&
\\
&&$7,T2,(1,14)$&&&&&&&&
\\
\hline

$14$&$6,T1,(0,14)^1$&$7,T2,(1,14)$&$7,T1,(2,14)$&&&&&&&
\\
&$6,T1,(0,14)^2$&& $7,T2,(1,14)$&&&&&&&
\\
&&&$7,T2,(2,12)^{1}$&&&&&&&
\\
&&&$7,T2,(2,12)^{2}$&&&&&&&
\\
\hline

$13$&$5,T1,(0,12)$&$6,T2,(1,12)^{1}$&$7,T2,(2,12)^{1}$&$8,T2,(3,12)$&&&&&&
\\
&&$6,T2,(1,12)^{2}$&$7,T2,(2,12)^{2}$&&&&&&&
\\
\hline

$12$&$5,T1,(0,12)$&$6,T2,(1,12)^{1}$&$7,T2,(2,12)^{1}$&$8,T2,(3,12)$&$8,T1,(4,12)$&&&&&
\\
&&$6,T2,(1,12)^{2}$&$7,T2,(2,12)^{2}$&&$8,T2,(3,12)$&&&&&
\\
\hline

$11$&$4,T1,(0,10)$&$5,T2,(1,10)$&$6,T2,(2,11)^{2}$&$7,T2,(3,11)^{2}$&$7,T1,(4,11)^{2}$&$7,T1(4,11)^{2}$&&&&
\\
&$4,T1,(0,8)$ &$5,T2,(1,11)$ &$6,T2,(2,11)^{3}$&&$7,T2,(3,11)^{2}$&$7,T2,(3,11)^{2}$&&&&
\\
&&&&&&$7,T2(5,10)$&&&&
\\
\hline

$10$&$4,T1,(0,10)$&$5,T2,(1,10)$&$5,T1,(2,10)$&$6,T2,(3,10)^{1}$
&$6,T1,(4,10)^{1}$&$7,T2,(5,10)$&$7,T1,(6,10)$&&&
\\
&$4,T1,(0,8)$ &&$5,T2,(1,10)$&$6,T2,(3,10)^{3}$&$6,T2,(3,10)^{1}$ &&$7,T2,(5,10)$&&&
\\
&&&$5,T2,(2,10)$&&$6,T1,(4,10)^{3}$&&$7,T2,(6,8)^{1}$&&&
 \\
&&&&&$6,T2,(3,10)^{3}$&&$7,T2,(6,8)^{2}$&&&
\\
\hline

$9$&&$4,T2,(1,8)$&$4,T1,(2,8)$&$5,T2,(3,8)$&$5,T1,(4,8)$&$6,T2,(5,8)^{1}$&$7,T2,(6,8)^{1}$&$8,T2,(7,8)$&&
\\
&&&$4,T2,(1,8)$ & $$&$5,T2,(3,8)$&$6,T2,(5,8)^{2}$&$7,T2,(6,8)^{2}$&&&
\\
&$4,T1,(0,8)$&$4,T1,(0,8)$&$4,T1,(0,8)$&&$5,T2,(4,8)$&$6,T2,(5,8)^{3}$&&&&
\\
\hline

$8$ &$4,T1,(0,8)$&$4,T2,(1,8)$&$4,T1,(2,8)$&$5,T2,(3,8)$&$5,T1,(4,8)
 $&$6,T2,(5,8)^{1}$&$7,T2,(6,8)^{1}$&$8,T2,(7,8)$&$8,T1,(8,8)$&
\\
&&&$4,T2,(1,8)$&$ $&$5,T2,(3,8)$ &$6,T2,(5,8)^{2}$&$7,T2,(6,8)^{2}$&&$8,T2,(7,8)$&
\\
&&$4,T1,(0,8)$ & $4,T1,(0,8)$&&$5,T2,(4,8)$ & $6,T2,(5,8)^{3}$&&&&
\\
\hline

$7$ &$3,T1,(0,7)$&$3,T1,(0,7)$&$3,T1,(0,7)$&$4,T2,(3,6)$&$4,T1,(4,6)$&$5,T2,(5,6)$&
$6,T2,(6,7)^{1}$&$7,T2,(7,7)^{1}$&$7,T1,(8,7)^{1}$&$7,T2,(9,6)$
\\
&&$3,T2,(1,6)$&$3,T1,(2,6)$&$4,T2,(3,4)$&$4,T2,(3,6)$&$5,T2,(5,7)$&$6,T2,(6,7)^{2}$&&$7,T2,(7,7)^{1}$&$7,T1,(8,7)$
\\
&&$$&$3,T2,(1,6)$&&$4,T1,(4,4)$&&$6,T2,(6,7)^{3}$&&&$7,T2,(7,7)$
\\
&&$$&$3,T2,(2,4)$&&$4,T2,(3,4)$&&&&&
\\
\hline

$6$ &$2,T1,(0,6)$&$3, T2,(1,6)$&$3,T1,(2,6)$&$4,T2,(3,6)$&$4,T1,(4,6)$&$5,T2,(5,6)$&$5,T1,(6,6)$& $6,T2,(7,6)^{1}$
&$6,T1,(8,6)^{1}$&$7,T2,(9,6)$
\\
&&&$3,T2,(1,6)$&$4,T2,(3,4)$&$4,T2,(3,6)$&&$5,T2,(5,6)$&$6,T2,(7,6)^{2}$&$6,T2,(7,6)^{1}$& 
\\
&&&$3,T2,(2,4) $&&$4,T1,(4,4)$&&$5,T2,(6,6)$&$6,T2,(7,6)^{3}$&$6,T1,(8,6)^{2}$&
\\
&&&&&$4,T2,(3,4)$&&&$$&$6,T2,(7,6)^{2}$&
\\
&&&&&$4,T2,(4,5)$&&&&$6,T1,(8,6)^{3}$&
\\
&&&&&&&&&$6,T2,(7,6)^{3}$&
\\
\hline
       
$5$&
$1,T1,(0,4)$&$2,T2,(1,4)$&$3,T2,(2,4)$&$4,T2,(3,4)$&$4,T2,(4,5)$&$4,T2,(5, 4)$&${4,T1,(6,4)}$&$5,T2,(7,5)$&$5,T1,(8,5)$&$6,T2,(9,4)^1$
\\
&&&&&$4,T1,(4,4)$&$4,T1,(5, 5)$ &$4,T2,(5,4)$&$5,T2,(7,4)$&$5,T2,(7,5)$&$6,T2,(9,4)^2$
\\
&&&&&$4,T2,(3,4)$&$4,T2,(4,5)$ &$4,T1,(5,5)$&&$5,T1,(8,4)$&$6,T2,(9,4)^3$ 
\\
&&&&&&$4,T1,(4,4)$ &$4,T2,(4,5)$ &$$&$5,T2,(7,4)$&
\\
&&&&&&$4,T2,(3,4)$&$4,T1,(4,4)$&&&
\\
&&&&&&&$4,T2,(3,4)$&&&
\\
\hline

$4$&$1,T1,(0,4)$&$2,T2,(1,4)$&$3,T2,(2,4)$&$4,T2,(3,4)$
&$4,T1,(4,4)$&$4,T2,(5,4)$&$4,T1,(6,4)$ & $5,T2,(7,4)$&$5,T1,(8,4)$&$6,T2,(9,4)^1$
\\
&&&&&$4,T2,(3,4)$&$4,T1,(4,4)$&$4,T2,(5,4)$&&$5,T2,(7,4)$&$6,T2,(9,4)^2$
\\
&&&&&&$4,T2,(3,4)$&$4,T1,(4,4)$&&$5,T2,(8,4)$&$6,T2,(9,4)^3$
\\
&&&&&&&$4,T2,(3,4)$&&&
\\
\hline

$3$&$\ell=0$ & $1,T2,(1,2)$&$2,T2,(2,3)$ & $3,T2,(3,3)$ & $3,T1,(4,3)$ & $3,T2,(5,2)$ & 
$3,T1,(6,2)$ & $4,T2,(7,2)$ & $4,T1,(8,2)$&$5,T2,(9,3)$
\\
&&&&&$3,T2,(3,3)$&$3,T1,(4,3)$&$3,T2,(5,2)$&$4,T2,(7,0)$& $4,T2,(7,2)$&$5,T2,(9,2)$
\\
&&&&&&$3,T2,(3,3)$&$3,T2,(6,0)$&&$4,T1,(8,0)$&
\\
&&&&&&&$3,T1,(4,3)$&&$4,T2,(7,0)$&
\\
&&&&&&&$3,T2,(3,3)$&&&
\\
\hline

$2$ &$\ell=0$    &$1, T2,(1,2)$ & $1,T1, (2,2)$ & $2, T2,(3,2)$ & $2,T1, (4,2)$ & $3,T2,(5,2)$& 
$3,T1,(6,2)$ & $4,T2,(7,2)$ & $4,T1,(8,2)$&$5,T2,(9,2)$
\\
&&&$1,T2,(1,2)$&&$2,T2,(3,2)$&&$3,T2,(5,2)$&$4,T2,(7,0)$&$4,T2,(7,2)$&
\\
&&&&&&&$3,T2,(6,0)$&&$4,T1,(8,0)$&
\\
&&&&&&&&&$4,T2,(7,0)$&
\\
\hline

$1$ &$\ell=0$&$\ell=0$&$\ell=0$&$1, T2,(3,0)$&$1, T1, (4,0)$& $2,T2,(5,0)$
&$3,T2,(6,0)$&$4,T2,(7,0)$&$4,T1,(8,0)$&$4,T2,(9,0)$
\\ 
&&&&&$1,T2,(3,0)$&&&&$4,T2,(7,0)$&$4,T1,(8,0)$
\\ 
&&&&&&&&&&$4,T2,(7,0)$
\\ 
&&&&&&&&&&
\\
\hline

$0$&$\ell=0$&$\ell=0$&$\ell=0$&$1,T2, (3,0)$&$1, T1, (4,0)$
&$2,T2,(5,0)$&$3,T2,(6,0)$&$4,T2,(7,0)$&$4,T1,(8,0)$&$4,T2,(9,0)$
\\
&&&&&$1, T2, (3,0)$&&&&$4,T2,(7,0)$&$4,T1,(8,0)$
\\
&&&&&&&&&&$4,T2,(7,0)$ 
\\
&&&&&&&&&&
\\
\hline\hline

s/r  & $0$ & $1$ & $2$ & $3$ & $4$ & $5$  &$6$ & $7$& $8$ & $9$ 
\\
\hline
\end{tabular}
}
\label{values of ell(r,s),  r+s leq 16}
\label{tab:0<r<9, s=0 to 8}
\end{table}

\begin{table}[h] 
\caption{{\small Connected groups for $r+s\leq 16$. Part 2}}
\scalebox{0.7}[0.7]{%

\begin{tabular}{|c||c|c|c|c|c|c|c|c|c|c|c|c|c|c|c|c|c||}\hline

$6$ &$7,T1,(10,6)$&&&&&&
\\
&$7,T2,(9,6)$&&&&&& 
\\
&$7,T2,(10,4)^1$&&&&&&
\\
&$7,T2,(10,4)^2$&&&&&&
\\
\hline
       
$5$&$7,T2,(10,4)^1$&$8,T2,(11,4)$&&&&&
\\
&$7,T2,(10,4)^2$&&&&&&
\\
\hline

$4$&$7,T2,(10,4)^1$&$8,T2,(11,4)$&$8,T1,(12,4)$&&&&
\\
&$7,T2,(10,4)^2$&&$8,T2,(11,4)$&&&&
\\
\hline

$3$
&$6,T2,(10,3)^2$&$7,T1,(11,3)^{2}$&$7,T1,(12,3)^2$&$7,T2,(13,3)^2$&&&
\\
&$6,T2,(10,3)^3$&&$7,T2,(11,3)^2$&$7,T2,(11,3)^2$&&&
\\
&&&&$7,T2,(13,2)$&&&
\\
\hline

$2$ &$5,T1,(10,2)$&$6,T2,(11,2)^{1}$&$6,T2,(11,2)^1$&$7,T2,(13,2)$&$7,T1,(14,2)$&&
\\
&$5,T2,(9,2)$&$6,T2,(11,2)^{2}$&$6,T1,(12,2)^1$&&$7,T2,(13,2)$&&
\\
&$5,T2,(10,2)$&&$6,T2,(11,2)^3$&&$7,T2,(14,0)^1$&&
\\
&&&$6,T1,(12,2)^3$&&$7,T2,(14,0)^2$&&
\\
\hline

$1$ &$4,T1,(10,0)$&$5,T2,(11,0)$&$5,T1,(12,0)$&$6,T2,(13,0)^1$&$7,T2,(14,0)^1$&$8,T2,(15,0)$&$8,T1,(16,0)$
\\ 
&$4,T2,(9,0)$&&$5,T2,(11,0)$&$6,T2,(13,0)^2$&$7,T2,(14,0)^2$&&$8,T2,(15,0)$
\\ 
&$4,T1,(8,0)$&&$5,T2,(12,0)$&$6,T2,(13,0)^3$&&&
\\ 
&$4,T2,(7,0)$&&&&&&
\\
\hline

$0$ & $4,T1,(10,0)$&$5,T2,(11,0)$&$5,T1,(12,0)$&$6,T2,(13,0)^1$&$7,T2,(14,0)^1$&$8,T2,(15,0)$&$8,T1,(16,0)$
\\
&$4,T2,(9,0)$&&$5,T2,(11,0)$&$6,T2,(13,0)^2$&$7,T2,(14,0)^2$&&$8,T2,(15,0)$
\\
&$4,T1,(8,0)$&&$5,T2,(12,0)$&$6,T2,(13,0)^3$&&& 
\\
&$4,T2,(7,0)$&&&&&&
\\
\hline\hline

s/r & $10$ & $11$ & $12$& $13$ & $14$ & $15$ & $16$
\\
\hline
\end{tabular}
}
\label{tab:0<r<9, s=10 to 16}
\end{table}


\subsection{Constructions of disconnected groups $\mathcal S\in\mathbb {S}^M_{r,s}$ for $0<r+s<16$}\label{disconnected s>1}

Let $\pi_{0}(PI_{r,s})=2$ for $10\leq r+s\leq 16$.

If $\ell(r,s)=\ell(r+s,0)$, then the different sets of involutions are determined based on the data of $\mathbb{PI}^{M}_{r+s,0}$ with $\pi_{0}(PI)=2$. We apply the rules for connected sets to each collection of involutions in the disconnected set listed in Table~\ref{disconnected} for $r=10,\ldots 16$. We summarize the possible maximal sets $PI\in\mathbb{PI}^{M}_{r,s}$ in Tables~\ref{r=10, S5} and~\ref{r=13,S4}.

If $\ell(r,s)<\ell(r+s,0)$ then we proceed as in steps (G4)-(G6) of Section~\ref{s>1} for each connected subgroup $\mathcal S_{(1)}$ in the direct product decomposition $\mathcal S=\mathcal S_{(1)}\times \mathcal S_{(2)}$. We will not write this cases into Tables~\ref{r=10, S5} and~\ref{r=13,S4}, since they can be easily obtained by applying steps (G4)-(G6).

We explain in details the case $r+s=10$ to illustrate the procedure for the situation $\ell(r,s)=\ell(r+s,0)$. In Table~\ref{disconnected} there are two disconnected subgroups 
$$
\mathcal S^{(1)}_{10}\quad\text{with the signature}\quad (3,(T1,\theta),(7,0))\times(1,(T2,\pi),(3,0)),
$$
$$
\mathcal S^{(2)}_{10}\quad\text{with the signature}\quad (3,(T2,\theta),(6,0))\times(1,T1,(4,0)).
$$

Consider the case $\mathcal S^{(1)}_{10}$. We analyse Tables~\ref{tab:0<r<9, s=0 to 8} and~\ref{tab:0<r<9, s=10 to 16} and find all possible $(T1)$-type sets of involutions having the signature 
$$
(3,T1,(r,s))\quad\text{with}\quad r+s=7.
$$
We obtain
\begin{equation}\label{eq:70}
\begin{array}{lll}
&(3,T1,(7,0))\quad\text{which comes from }\ (3,T2,(6,0)) \quad\text{by using property}\quad (F2) 
\\
&(3,T1,(3,4))\quad\text{which comes from }\ (3,T2,(2,4)) \quad\text{by using property}\quad (F2)
\\
&(3,T1,(4,3))\quad\text{and}\quad (3,T1,(0,7)) \quad\text{listed in Tables~\ref{tab:0<r<9, s=0 to 8} and~\ref{tab:0<r<9, s=10 to 16}}.
\end{array}
\end{equation}
Next, we find all possible $(T2)$-type sets of involutions having the signature 
$$
(1,T2,(r,s))\quad\text{with}\quad r+s=3.
$$
They are the following
\begin{equation}\label{eq:30}
(1,T2,(3,0))\quad\text{and}\quad (1,T2,(1,2)).
\end{equation}
At the end we make all possible products of two groups of involutions, where the first one belongs to~\eqref{eq:70} and the second one belongs to~\eqref{eq:30}. For instance, we obtain the disconnected subgroups in $\mathbb{PI}^M_{r,s}$ with $r+s=10$:
$$
\mathcal S^{(1)}_{10}=
(3,T1,(7,0))
\times (1,T2,(3,0))
\in\mathbb{PI}^M_{10,0},\quad
$$
and 
$$
\mathcal S^{(1)}_{6,4}=
(3,T1,(3,4))
\times(1,T2,(3,0))
\in\mathbb{PI}^M_{6,4},
$$
and so on. These disconnected groups and all other are listed in Tables~\ref{r=10, S5} and~\ref{r=13,S4}. 

We do analogous calculations for the disconnected group $\mathcal S^{(2)}_{10}\in\mathbb{PI}^M_{10,0}$ and write the results in Tables~\ref{r=10, S5} and~\ref{r=13,S4}.

\begin{center}
\begin{table}[h]  
\caption{{\small $\pi_{0}(PI_{r,s})=2$, $r+s\leq 16$,\ \text{Part 1}}}
\scalebox{0.35}[0.35]{
\begin{tabular}{|c||c|c|c|c|c|c|c|c|c|c|c|c||}\hline

16&$\mathcal S^{(1)}_{0,16}=\begin{array}{llll}(4,T1,(0,8))\\\times (4,T1,(0,8))\end{array}$&&&&&&&& 

\\
\hline
15&&&&&&&&& 
\\
\hline
14&&&&&&&&&
\\
\hline

13&&&&&&&&& 
\\
\hline

12&$\mathcal S^{(1)}_{0,12}=\begin{array}{llll}(4,T1,(0,8))\\\times (1,T1,(0,4))\end{array}$&
$\mathcal S^{(1)}_{1,12}=\begin{array}{llll}(4,T1,(0,8))\\\times (2,T2,(2,3))\end{array}$
&
$\mathcal S^{(1)}_{2,12}=\begin{array}{llll}(4,T1,(0,8))\\\times (3,T2,(2,4))\end{array}$
&$\mathcal S_{3,12}=\begin{array}{llll}(4,T1,(0,8))\\\times (4,T2,(3,4))\end{array}$&
$\begin{array}{lll}
\mathcal S^{(1)}_{4,12}=\begin{array}{llll}(4,T1,(4,4))\\\times (2,T1,(0,8))\end{array}\\\hline\hline
\mathcal S^{(1)}_{4,12}=\mathcal S_{3,12}\\
\end{array}$&&&&
\\
\hline

11&&$\mathcal S^{(1)}_{1,11}=\begin{array}{llll}(3,T1,(0,7))\\\times (2,T2,(1,4))\end{array}$&
$\begin{array}{lll}
\mathcal S^{(1)}_{2,11}=\begin{array}{llll}(4,T1,(0,8))\\\times (2,T2,(2,3))\end{array}\\\hline
\mathcal S^{(3)}_{2,11}=\begin{array}{llll}(3,T1,(0,7))\\\times (3,T2,(2,4))\end{array}\\
\end{array}$
&
$\begin{array}{lll}
\mathcal S^{(1)}_{3,11}=\begin{array}{llll}(4,T1,(0,8))\\\times (3,T2,(3,3))\end{array}\\\hline
\mathcal S^{(2)}_{3,11}=\begin{array}{llll}(2,T2,(3,4))\\\times (3,T1,(0,7))\end{array}\\
\end{array}$&
&&&&
\\
\hline

10&&$\mathcal S^{(1)}_{1,10}=\begin{array}{llll}(4,T1,(0,8))\\\times (1,T2,(1,2))\end{array}$&
$\begin{array}{lll}
\mathcal S^{(1)}_{2,10}=\begin{array}{llll}(4,T1,(0,8))\\\times (1,T1,(2,2))\end{array}\\\hline
\mathcal S^{(2)}_{2,10}=\begin{array}{llll}(3,T1,(0,7))\\\times (2,T2,(2,3))\end{array}\\\hline
\mathcal S^{(3)}_{2,10}=\begin{array}{llll}(3,T2,(2,4))\\\times (2,T1,(0,6))\end{array}\\\hline\hline
\mathcal S^{(4)}_{2,10}=\mathcal S^{(1)}_{1,10}
\end{array}$
&$\begin{array}{lll}
\mathcal S^{(1)}_{3,10}=\begin{array}{llll}(4,T1,(0,8))\\\times (2,T2,(3,2))\end{array}\\\hline
\mathcal S^{(1)}_{3,10}=\begin{array}{llll}(4,T2,(3,4))\\\times (2,T1,(0,6))\end{array}\\\hline
\mathcal S^{(1)}_{3,10}=\begin{array}{llll}(3,T1,(0,7))\\\times (3,T2,(3,3))\end{array}\\
\end{array}$
&&&&&
\\
\hline

9&&$\mathcal S^{(1)}_{1,9}=\begin{array}{llll}(3,T1,(0,7))\\\times (1,T2,(1,2))\end{array}$&&
$\begin{array}{lll}
\mathcal S^{(2)}_{3,9}=\begin{array}{llll}(3,T1,(0,7))\\\times (2,T2,(3,2))\end{array}\\\hline
\mathcal S^{(3)}_{3,9}=\begin{array}{llll}(3,T2,(3,3))\\\times (2,T1,(0,6))\end{array}
\end{array}$
&&&&&
\\
\hline

8&&&
$\mathcal S^{(2)}_{2,8}=\begin{array}{lll}(3,T2,(2,4))\\\times (1,T1,(0,4))\end{array}$&
$\begin{array}{lll}
\mathcal S^{(1)}_{3,8}=\begin{array}{llll}(4,T1,(0,8))\\\times (1,T2,(3,0))\end{array}\\\hline
\mathcal S^{(2)}_{3,8}=\begin{array}{llll}(4,T2,(3,4))\\\times (1,T1,(0,4))\end{array}
\end{array}$&
$\begin{array}{lll}
\mathcal S^{(1)}_{4,8}=\begin{array}{llll}(4,T1,(4,4))\\\times (1,T1,(0,4))\end{array}\\\hline
\mathcal S^{(1)}_{4,8}=\begin{array}{llll}(4,T1,(0,8))\\\times (1,T1,(4,0))\end{array}\\\hline
\mathcal S^{(2)}_{4,8}=\begin{array}{llll}(3,T1,(3,4))\\\times (2,T2,(1,4))\end{array}\\\hline
\mathcal S^{(3)}_{4,8}=\begin{array}{llll}(3,T2,(2,4))\\\times (2,T1,(2,4))\end{array}\\\hline\hline
\mathcal S^{(4)}_{4,8}=\mathcal S^{(1)}_{3,8}\\\hline
\mathcal S^{(5)}_{4,8}=\mathcal S^{(2)}_{3,8}
\end{array}$
&$\begin{array}{lll}
\mathcal S^{(1)}_{5,8}=\begin{array}{llll}(4,T1,(4,4))\\\times (2,T2,(1,4))\end{array}\\\hline
\mathcal S^{(1)}_{5,8}=\begin{array}{llll}(4,T1,(0,8))\\\times (2,T2,(5,0))\end{array}\\\hline
\mathcal S^{(2)}_{5,8}=\begin{array}{llll}(4,T2,(3,4))\\\times (2,T1,(2,4))\end{array}\\\hline
\mathcal S^{(3)}_{5,8}=\begin{array}{llll}(3,T1,(3,4))\\\times (3,T2,(2,4))\end{array}\\
\end{array}$
&
$\begin{array}{lll}
\mathcal S^{(1)}_{6,8}=\begin{array}{llll}(4,T1,(4,4))\\\times (3,T2,(2,4))\end{array}\\\hline
\mathcal S^{(1)}_{6,8}=\begin{array}{llll}(4,T1,(0,8))\\\times (3,T2,(6,0))\end{array}\\\hline
\mathcal S^{(2)}_{6,8}=\begin{array}{llll}(2,T2,(3,4))\\\times (3,T1,(3,4))\end{array}\\
\end{array}$
&$\begin{array}{lll}
\mathcal S^{(1)}_{7,8}=\begin{array}{llll}(4,T1,(4,4))\\\times (4,T2,(3,4))\end{array}\\\hline
\mathcal S^{(2)}_{7,8}=\begin{array}{llll}(4,T1,(0,8))\\\times (4,T2,(7,0))\end{array}\\
\end{array}$&
$\begin{array}{lll}
\mathcal S^{(1)}_{8,8}=\begin{array}{llll}(4,T1,(8,0))\\\times (2,T1,(0,8))\end{array}\\\hline\hline
\\
\mathcal S^{(2)}_{8,8}=\mathcal S_{7,8}^{(1)}\\\hline
\mathcal S^{(2)}_{8,8}=\mathcal S_{7,8}^{(2)}\\
\end{array}$
\\
\hline

7&&&&
$\begin{array}{lll}
\mathcal S^{(1)}_{3,7}=\begin{array}{llll}(3,T1,(0,7))\\\times (1,T2,(3,0))\end{array}\\\hline
\mathcal S^{(2)}_{3,7}=\begin{array}{lll}(3,T2,(3,3))\\\times (1,T1,(0,4))\end{array}
\end{array}$
&&
$\begin{array}{lll}
\mathcal S^{(2)}_{5,7}=\begin{array}{llll}(3,T1,(4,3))\\\times (2,T2,(1,4))\end{array}\\\hline
\mathcal S^{(2)}_{5,7}=\begin{array}{llll}(3,T1,(3,4))\\\times (2,T2,(2,3))\end{array}\\\hline
\mathcal S^{(2)}_{5,7}=\begin{array}{llll}(3,T1,(0,7))\\\times (2,T2,(5,0))\end{array}\\\hline
\mathcal S^{(3)}_{5,7}=\begin{array}{llll}(3,T2,(3,3))\\\times (2,T1,(2,4))\end{array}\\\hline
\mathcal S^{(3)}_{5,7}=\begin{array}{llll}(3,T2,(2,4))\\\times (2,T1,(3,3))\end{array}
\end{array}$&
$\begin{array}{lll}
\mathcal S^{(1)}_{6,7}=\begin{array}{llll}(4,T1,(4,4))\\\times (2,T2,(2,3))\end{array}\\\hline
\mathcal S^{(2)}_{6,7}=\begin{array}{llll}(4,T2,(3,4))\\\times (2,T1,(3,3))\end{array}\\\hline
\mathcal S^{(3)}_{6,7}=\begin{array}{llll}(3,T1,(4,3))\\\times (3,T2,(2,4))\end{array}\\\hline
\mathcal S^{(3)}_{6,7}=\begin{array}{llll}(3,T1,(3,4))\\\times (3,T2,(3,3))\end{array}\\\hline
\mathcal S^{(3)}_{6,7}=\begin{array}{llll}(3,T1,(0,7))\\\times (3,T2,(6,0))\end{array}\\
\end{array}$&
$\begin{array}{lll}
\mathcal S^{(1)}_{7,7}=\begin{array}{llll}(4,T1,(4,4))\\\times (3,T2,(3,3))\end{array}\\\hline
\mathcal S^{(2)}_{7,7}=\begin{array}{llll}(2,T2,(7,0))\\\times (3,T1,(0,7))\end{array}\\\hline
\mathcal S^{(2)}_{7,7}=\begin{array}{llll}(2,T2,(3,4))\\\times (3,T1,(4,2))\end{array}\\
\end{array}$&
\\
\hline

6&&&&&
$\begin{array}{lll}
\mathcal S^{(1)}_{4,6}=\begin{array}{llll}(3,T1,(3,4))\\\times (1,T2,(1,2))\end{array}\\\hline
\mathcal S^{(2)}_{4,6}=\begin{array}{lll}(3,T2,(2,4))\\\times (1,T1,(2,2))\end{array}
\end{array}$&
$\begin{array}{lll}
\mathcal S^{(1)}_{5,6}=\begin{array}{llll}(4,T1,(4,4))\\\times (1,T2,(1,2))\end{array}\\\hline
\mathcal S^{(2)}_{5,6}=\begin{array}{llll}(4,T2,(3,4))\\\times (1,T1,(2,2))\end{array}
\end{array}$&
$\begin{array}{lll}
\mathcal S^{(1)}_{6,6}=\begin{array}{llll}(4,T1,(4,4))\\\times (1,T1,(2,2))\end{array}\\\hline
\mathcal S^{(2)}_{6,6}=\begin{array}{llll}(3,T1,(4,3))\\\times (2,T2,(2,3))\end{array}\\\hline
\mathcal S^{(2)}_{6,6}=\begin{array}{llll}(3,T1,(3,4))\\\times (2,T2,(3,2))\end{array}\\\hline
\mathcal S^{(3)}_{6,6}=\begin{array}{llll}(3,T2,(6,0))\\\times (2,T1,(0,6))\end{array}\\\hline
\mathcal S^{(3)}_{6,6}=\begin{array}{llll}(3,T2,(3,3))\\\times (2,T1,(3,3))\end{array}\\\hline
\mathcal S^{(3)}_{6,6}=\begin{array}{llll}(3,T2,(2,4))\\\times (2,T1,(4,2))\end{array}\\\hline\hline
\mathcal S^{(4)}_{6,6}=\mathcal S^{(1)}_{5,6}\\\hline
\mathcal S^{(5)}_{6,6}=\mathcal S^{(2)}_{5,6}
\end{array}$
&$\begin{array}{lll}
\mathcal S^{(1)}_{7,6}=\begin{array}{llll}(4,T1,(4,4))\\\times (2,T2,(3,2))\end{array}\\\hline
\mathcal S^{(2)}_{7,6}=\begin{array}{llll}(4,T2,(7,0))\\\times (2,T1,(0,6))\end{array}\\\hline
\mathcal S^{(2)}_{7,6}=\begin{array}{llll}(4,T2,(3,4))\\\times (2,T1,(4,2))\end{array}\\\hline
\mathcal S^{(3)}_{7,6}=\begin{array}{llll}(3,T1,(4,3))\\\times (3,T2,(3,3))\end{array}\\
\end{array}$&
\\
\hline 

5&&&&&&
$\begin{array}{lll}
\mathcal S^{(1)}_{5,5}=\begin{array}{lll}(3,T1,(4,3))\\\times (1,T2,(1,2))\end{array}\\\hline
\mathcal S^{(2)}_{5,5}=\begin{array}{lll}(3,T2,(3,3))\\\times (1,T1,(2,2))\end{array}
\end{array}$
&&
$\begin{array}{lll}
\mathcal S^{(1)}_{7,5}=\begin{array}{llll}(3,T1,(4,3))\\\times (2,T2,(3,2))\end{array}\\\hline
\mathcal S^{(2)}_{7,5}=\begin{array}{llll}(3,T2,(3,3))\\\times (2,T1,(4,2))\end{array}\\
\end{array}$
&
\\
\hline

4&&&&&&&
$\begin{array}{lll}
\mathcal S^{(1)}_{6,4}=\begin{array}{lll}(3,T1,(3,4))\\\times (1,T2,(3,0))\end{array}\\\hline
\mathcal S^{(2)}_{6,4}=\begin{array}{lll}(3,T2,(6,0))\\\times (1,T1,(0,4))\end{array}\\\hline
\mathcal S^{(2)}_{6,4}=\begin{array}{lll}(3,T2,(2,4))\\\times (1,T1,(4,0))\end{array}
\end{array}$&
$\begin{array}{lll}
\mathcal S^{(1)}_{7,4}=\begin{array}{llll}(4,T1,(4,4))\\\times (1,T2,(3,0))\end{array}\\\hline
\mathcal S^{(2)}_{7,4}=\begin{array}{llll}(4,T2,(7,0))\\\times (1,T1,(0,4))\end{array}\\\hline
\mathcal S^{(2)}_{7,4}=\begin{array}{llll}(4,T2,(3,4))\\\times (1,T1,(4,0))\end{array}
\end{array}$&
$\begin{array}{lll}
\mathcal S^{(1)}_{8,4}=\begin{array}{llll}(4,T1,(8,0))\\\times (1,T1,(0,4))\end{array}\\\hline
\mathcal S^{(1)}_{8,4}=\begin{array}{llll}(4,T1,(4,4))\\\times (1,T1,(4,0))\end{array}\\\hline
\mathcal S^{(2)}_{8,4}=\begin{array}{llll}(3,T1,(7,0))\\\times (2,T2,(1,4))\end{array}\\\hline
\mathcal S^{(2)}_{8,4}=\begin{array}{llll}(3,T1,(3,4))\\\times (2,T2,(5,0))\end{array}\\\hline
\mathcal S^{(3)}_{8,4}=\begin{array}{llll}(3,T2,(6,0))\\\times (2,T1,(2,4))\end{array}\\\hline
\mathcal S^{(3)}_{8,4}=\begin{array}{llll}(3,T2,(2,4))\\\times (2,T1,(6,0))\end{array}\\\hline\hline
\mathcal S^{(2)}_{8,4}=\mathcal S^{(1)}_{7,4}\\\hline
\mathcal S^{(3)}_{8,4}=\mathcal S^{(2)}_{7,4}\\\hline
\mathcal S^{(3)}_{8,4}=\mathcal S^{(2)}_{7,4}
\end{array}$
\\
\hline

3&&&&&&&&
$\begin{array}{lll}
\mathcal S^{(1)}_{7,3}=\begin{array}{llll}(3,T1,(4,3))\\\times (1,T2,(3,0))\end{array}\\\hline
\mathcal S^{(2)}_{7,3}=\begin{array}{llll}(3,T2,(3,3))\\\times (1,T1,(4,0))\end{array}
\end{array}$&
\\
\hline

2&&&&&&&&&
$\begin{array}{lll}
\mathcal S^{(1)}_{8,2}=\begin{array}{llll}(3,T1,(7,0))\\\times (1,T2,(1,2))\end{array}\\\hline
\mathcal S^{(2)}_{8,2}=\begin{array}{llll}(3,T2,(6,0))\\\times (1,T1,(2,2))\end{array}
\end{array}$
\\
\hline

1&&&&&&&&&
\\
\hline
         
0&&&&&&&&&
\\
\hline
\hline         

s/r&0&1&2&3&4&5&6&7&8
\\
\hline
\end{tabular}
}
\label{r=10, S5}
\end{table}
\end{center}
\newpage

\begin{center}
\begin{table}[h]  
\caption{{\small $\pi_{0}(PI_{r,s})=2$, $r+s\leq 16$,\ \text{Part 2}}}
\scalebox{0.38}[0.38]{
\begin{tabular}{|c||c|c|c|c|c|c|c|c|c|c|c|}\hline

4&
$\begin{array}{lll}
\mathcal S^{(1)}_{9,4}=\begin{array}{llll}(4,T1,(8,0))\\\times (2,T2,(1,4))\end{array}\\\hline
\mathcal S^{(1)}_{9,4}=\begin{array}{llll}(4,T1,(4,4))\\\times (2,T2,(5,0))\end{array}\\\hline
\mathcal S^{(2)}_{9,4}=\begin{array}{llll}(4,T2,(7,0))\\\times (2,T1,(2,4))\end{array}\\\hline
\mathcal S^{(2)}_{9,4}=\begin{array}{llll}(4,T2,(3,4))\\\times (2,T1,(6,0))\end{array}\\\hline
\mathcal S^{(3)}_{9,4}=\begin{array}{llll}(3,T1,(7,0))\\\times (3,T2,(2,4))\end{array}\\\hline
\mathcal S^{(3)}_{9,4}=\begin{array}{llll}(3,T1,(3,4))\\\times (3,T2,(6,0))\end{array}\\
\end{array}$
&
$\begin{array}{lll}
\mathcal S^{(1)}_{10,4}=\begin{array}{llll}(4,T1,(8,0))\\\times (3,T2,(2,4))\end{array}\\\hline
\mathcal S^{(1)}_{10,4}=\begin{array}{llll}(4,T1,(4,4))\\\times (3,T2,(6,0))\end{array}\\\hline
\mathcal S^{(2)}_{10,4}=\begin{array}{llll}(2,T2,(7,0))\\\times (3,T1,(3,4))\end{array}\\\hline
\mathcal S^{(2)}_{10,4}=\begin{array}{llll}(2,T2,(3,4))\\\times (3,T1,(7,0))\end{array}\\
\end{array}$
&
$\begin{array}{lll}
\mathcal S^{(1)}_{11,4}=\begin{array}{llll}(4,T1,(8,0))\\\times (4,T2,(3,4))\end{array}\\\hline
\mathcal S^{(2)}_{11,4}=\begin{array}{llll}(4,T1,(4,4))\\\times (4,T2,(7,0))\end{array}\\
\end{array}$

&$\begin{array}{lll}
\mathcal S^{(1)}_{12,4}=\begin{array}{llll}(4,T1,(8,0))\\\times (4,T1,(4,4))\end{array}\\\hline\hline
\\
\mathcal S^{(2)}_{12,4}=\mathcal S^{(1)}_{11,4}\\\hline
\mathcal S^{(2)}_{12,4}=\mathcal S^{(2)}_{11,4}\\
\end{array}$&&&&
\\
\hline

3&
$\begin{array}{lll}
\mathcal S^{(2)}_{9,3}=\begin{array}{llll}(3,T1,(7,0))\\\times (2,T2,(2,3))\end{array}\\\hline
\mathcal S^{(2)}_{9,3}=\begin{array}{llll}(3,T1,(4,3))\\\times (2,T2,(5,0))\end{array}\\\hline
\mathcal S^{(3)}_{9,3}=\begin{array}{llll}(3,T2,(6,0))\\\times (2,T1,(3,3))\end{array}\\\hline
\mathcal S^{(3)}_{9,3}=\begin{array}{llll}(3,T2,(3,3))\\\times (2,T1,(6,0))\end{array}
\end{array}$
&
$\begin{array}{lll}
\mathcal S^{(1)}_{10,3}=\begin{array}{llll}(4,T1,(8,0))\\\times (2,T2,(2,3))\end{array}\\\hline
\mathcal S^{(2)}_{10,3}=\begin{array}{llll}(4,T2,(7,0))\\\times (2,T1,(3,3))\end{array}\\\hline
\mathcal S^{(3)}_{10,3}=\begin{array}{llll}(3,T1,(7,0))\\\times (3,T2,(3,3))\end{array}\\\hline
\mathcal S^{(3)}_{10,3}=\begin{array}{llll}(3,T1,(4,3))\\\times (3,T2,(6,0))\end{array}\\
\end{array}$
&$\begin{array}{lll}
\mathcal S^{(1)}_{11,3}=\begin{array}{llll}(4,T1,(8,0))\\\times (3,T2,(3,3))\end{array}\\\hline
\mathcal S^{(2)}_{11,3}=\begin{array}{llll}(2,T2,(7,0))\\\times (3,T1,(4,3))\end{array}\\
\end{array}$&&&&&
\\
\hline

2&$\begin{array}{lll}
\mathcal S^{(1)}_{9,2}=\begin{array}{llll}(4,T1,(8,0))\\\times (1,T2,(1,2))\end{array}\\\hline
\mathcal S^{(2)}_{9,2}=\begin{array}{llll}(4,T2,(7,0))\\\times (1,T1,(2,2))\end{array}
\end{array}$&
$\begin{array}{lll}
\mathcal S^{(1)}_{10,2}=\begin{array}{llll}(4,T1,(8,0))\\\times (1,T1,(2,2))\end{array}\\\hline
\mathcal S^{(2)}_{10,2}=\begin{array}{llll}(3,T1,(7,0))\\\times (2,T2,(3,2))\end{array}\\\hline
\mathcal S^{(3)}_{10,2}=\begin{array}{llll}(3,T2,(6,0))\\\times (2,T1,(4,2))\end{array}\\\hline\hline
\mathcal S^{(4)}_{10,2}=\mathcal S^{(1)}_{9,2}\\\hline
\mathcal S^{(5)}_{10,2}=\mathcal S^{(2)}_{9,2}
\end{array}$
&$\begin{array}{lll}
\mathcal S^{(1)}_{11,2}=\begin{array}{llll}(4,T1,(8,0))\\\times (2,T2,(3,2))\end{array}\\\hline
\mathcal S^{(2)}_{11,1}=\begin{array}{llll}(2,T2,(7,0))\\\times (2,T1,(4,2))\end{array}\\
\end{array}$&&&&&
\\
\hline

1&&&&&&&&
\\
\hline
0 &&
$\begin{array}{lll}
\mathcal S^{(1)}_{10}=\begin{array}{llll}(3,T1,(7,0))\\\times (1,T2,(3,0))\end{array}\\\hline
\mathcal S^{(2)}_{10}=\begin{array}{llll}(3,T2,(6,0))\\\times (1,T1,(4,0))\end{array}
\end{array}$
&
$\begin{array}{lll}
\mathcal S^{(1)}_{11}=\begin{array}{llll}(4,T1,(8,0))\\\times (1,T2,(3,0))\end{array}\\\hline
\mathcal S^{(2)}_{11}=\begin{array}{llll}(4,T2,(7,0))\\\times (1,T1,(4,0))\end{array}
\end{array}$
&$\begin{array}{lll}
\mathcal S^{(1)}_{12}=\begin{array}{llll}(4,T1,(8,0))\\\times (1,T1,(4,0))\end{array}\\\hline
\mathcal S^{(2)}_{12}=\begin{array}{llll}(3,T1,(7,0))\\\times (2,T2,(5,0))\end{array}\\\hline
\mathcal S^{(3)}_{12}=\begin{array}{llll}(3,T2,(6,0))\\\times (2,T1,(6,0))\end{array}\\\hline\hline
\begin{array}{llll}\mathcal S^{(4)}_{12}=\mathcal S^{(1)}_{11}\\\hline
\mathcal S^{(5)}_{12}=\mathcal S^{(2)}_{11}\end{array}
\end{array}$ & 
$\begin{array}{lll}
\mathcal S^{(1)}_{13}=\begin{array}{llll}(4,T1,(8,0))\\\times (2,T2,(5,0))\end{array}\\\hline
\mathcal S^{(2)}_{13}=\begin{array}{llll}(4,T2,(7,0))\\\times (2,T1,(6,0))\end{array}\\\hline
\mathcal S^{(3)}_{13}=\begin{array}{llll}(3,T1,(7,0))\\\times (3,T2,(6,0))\end{array}\\
\end{array}$&
$\begin{array}{lll}
\mathcal S^{(1)}_{14}=\begin{array}{llll}(4,T1,(8,0))\\\times (3,T2,(6,0))\end{array}\\\hline
\mathcal S^{(2)}_{14}=\begin{array}{llll}(4,T2,(7,0))\\\times (3,T1,(7,0))\end{array}\\
\end{array}$&
$\mathcal S_{15}=\begin{array}{llll}(4,T1,(8,0))\\\times (4,T2,(7,0))\end{array}$&
$\begin{array}{lll}
\mathcal S^{(1)}_{16}=\begin{array}{llll}(4,T1,(8,0))\\\times (4,T1,(8,0))\end{array}\\\hline\hline
\mathcal S^{(2)}_{16}=\mathcal S_{15}\\
\end{array}$
\\\hline\hline
s/r&9&10&11&12&13&14&15 &16\\\hline
\end{tabular}
}
\label{r=13,S4}
\end{table}
\end{center}


\section{Isomorphism of invariant integral structures}\label{sec:uniform-subgroups}


\begin{theorem}\label{th:12}
If 
\begin{equation}\label{low-rs} 
(r,s)\in\{(0,0),(1,0),(2,0),(0,1),(1,1),(2,1),(0,2)\},
\end{equation}
then for any orthonormal basis $\mathfrak B_{r,s}=\{z_j\}$ and $v\in V^{r,s}$, with $\langle v,v\rangle_{V^{r,s}}=\pm 1$ the invariant orthonormal structures spanned by bases as in Table~\ref{tab:trival-groups} are isomorphic.

\begin{table}[h]
\caption{Invariant integral structures for $(r,s)$ in Theorem~\ref{th:12}}
\scalebox{0.75}{
\begin{tabular}{|c||c|c|c|c|}
\hline
$2$  &$\{v,J_{z_1}v,J_{z_2}v,J_{z_1}J_{z_2}v,z_1,z_2\}$&&
\\
\hline
$1$    &$\{v,J_{z_1}v,z_1\}$&$\{v,J_{z_1}v,J_{z_2}v,J_{z_1}J_{z_2}v,z_1,z_2\}$&$\{v,J_{z_1}v,J_{z_2}v,J_{z_1}J_{z_2}v,z_1,z_2\}$ 
\\
\hline
$0$    &$v$&$\{v,J_{z_1}v,z_1\}$&$\{v,J_{z_1}v,J_{z_2}v,J_{z_1}J_{z_2}v,z_1,z_2\}$
\\
\hline\hline
$s/r$  &$0$&$1$&$2$  
\\\hline
\end{tabular}\label{tab:trival-groups}
}
\end{table}
\end{theorem}
\begin{proof} There are only trivial groups $\mathcal S\subset\mathbb S^M_{r,s}$ for $(r,s)$ as in~\eqref{th:12} since there are no involutions. 
The proof of uniqueness literally repeats the proof of Theorem~\ref{prop:vector_v}.
\end{proof}
%
%
%
%
We fix an orthonormal basis $\mathfrak B_{r,s}=\{z_1,\ldots,z_{r+s}\}$ and a group $\mathcal S=\mathcal S(PI_{r,s})$. Recall the construction of an invariant basis $\mathcal B_v(V^{r,s})$ on the minimal admissible module $V^{r,s}$ from Theorem~\ref{basis expression}, which used the centraliser of the isotropy group $\mathcal S=\mathcal S(PI_{r,s})=\mathcal S_v$ of a unit vector $v\in V^{r,s}$. The invariant integral structure on the Lie algebra $\mathfrak n_{r,s}(V^{r,s})$ given by $\mathcal S$ will be denoted by
$$
\mathcal L(\mathcal S)=\spn_{\mathbb Z}\{\mathcal B_v(V^{r,s})\}\oplus\spn_{\mathbb Z}\{\mathfrak B_{r,s}\}.
$$

\begin{theorem}\label{th:From group}
If two groups $\mathcal S_1$ and $\mathcal S_2$ are equivalent;
that is there exists a map $C\in O(r,s)$ such that $C(\widehat{\mathcal S_1})=\widehat{\mathcal S_2}$, 
then the invariant integral structures $\mathcal L(\mathcal S_1)$ and $\mathcal L(\mathcal S_2)$ are isomorphic under a map $A\oplus C$, 
where
$A\colon V^{r,s}\to V^{r,s}$ is an orthogonal map with respect to $\langle.\,,.\rangle_{V^{r,s}}$; that is $A^{\tau}A=\Id_{V^{r,s}}$.
\end{theorem}

\begin{proof} The proof is a slight generalisation of Theorem~\ref{prop:vector_v}.
Let $\mathcal S_1=\mathcal S(PI_1)$ and $\mathcal S_2=\mathcal S(PI_2)$ be equivalent groups. It implies that there is $C\in\Orth (r,s)$ such that $C(\widehat{\mathcal S_1})=\widehat{\mathcal S_2}$ where we denoted by the same letter $C$ the extension of the orthogonal map to the group $\Cl^*_{r,s}\subset \Cl_{r,s}$ of invertible elements of the Clifford algebra $\Cl_{r,s}$. Let 
\begin{equation}\label{eq:Isom basis 1}
\mathcal B_v(V^{r,s})=\Big\{v, J_{\sigma_i}(v),
J_{\tau_j}(v),
J_{\tau_j}J_{\sigma_i}(v)\mid\  \sigma_i,\tau_j, \sigma_i\tau_j\in\Sigma(\mathcal S_1)\Big\}
\end{equation}
be the invariant basis, constructed in Theorem~\ref{basis expression} by making use the eigenspaces of involutions from $PI_1$. The set $PI_1$ is equivalent to $PI_2$ under $C$. We use the method of Theorem~\ref{basis expression} and obtain a basis
\begin{eqnarray}\label{eq:Isom basis 2}
\mathcal B_w(V^{r,s})&=&\Big\{w, J_{C(\sigma_i)}(w),
 J_{C(\tau_j)}(w),
J_{C(\tau_j)}J_{C(\sigma_i)}(w)\mid\nonumber
\\
&&C(\sigma_i),C(\tau_j),C(\sigma_i)C(\tau_j)\in\Sigma(\mathcal S_2)\Big\},
\end{eqnarray}
where $\mathcal S_2\cong\mathcal S(PI_2)\cong\mathcal S(C(PI_1))$ and the set $PI_2$ was replaced by $C(PI_1)$. Note that since $C(\mathfrak B_{r,s})=\mathfrak B_{r,s}$ we also have $G(\mathfrak B_{r,s})=G\big(C(\mathfrak B_{r,s})\big)$.

We construct a correspondence $A\colon \mathcal B_v(V^{r,s})\to \mathcal B_w(V^{r,s})$ by 
\begin{align*}
&v\longmapsto w,\quad 
J_{\sigma_i}(v)\longmapsto J_{C(\sigma_i)}(w),\quad
J_{\tau_j}(v)\longmapsto J_{C(\tau_j)}(w),
\\
&J_{\tau_j}(v)J_{\sigma_i}(v)\longmapsto
J_{C(\tau_j)}(w)J_{C(\sigma_i)}(w),
\end{align*}
and $C\colon z_k\longmapsto C(z_k)$. The correspondence $A\oplus C$ extended to a linear map over $\mathbb R$ or $\mathbb Z$ is an orthogonal map on $V^{r,s}$ since it maps orthonormal basis ~\eqref{eq:Isom basis 1} to orthonormal basis~\eqref{eq:Isom basis 2}. To show that the linear map $A\oplus C$ is an isomorphism of invariant integral structures, we argue as in Theorem~\ref{prop:vector_v}. By the invariance of the bases $\mathcal B_v(V^{r,s})$ and $\mathcal B_w(V^{r,s})$ we have 
$$
J_{C(z_k)}Au_{\alpha}=
\pm J_{C(\varkappa)}w=
\pm AJ_{\varkappa}v=AJ_{z_k}u_{\alpha}
$$
for any $u_{\alpha}\in \mathcal B_v(V^{r,s})$, $z_k\in \mathfrak B_{r,s}$, and for some $\varkappa\in\Sigma=\{\sigma_i,\tau_j,\tau_j\sigma_i\}$.
It implies 
\begin{eqnarray*}
\langle [Au_{\alpha},Au_{\beta}],C(z_k)\rangle_{r,s} &=&
\langle J_{C(z_k)}Au_{\alpha},Au_{\beta}\rangle_{V^{r,s}}
=
\langle AJ_{z_k}u_{\alpha},Au_{\beta}\rangle_{V^{r,s}}\nonumber
\\
&=&
\langle A^{\tau}AJ_{z_k}u_{\alpha},u_{\beta}\rangle_{V^{r,s}}=
\langle J_{z_k}u_{\alpha},u_{\beta}\rangle_{V^{r,s}}
\\
&=&
\langle [u_{\alpha},u_{\beta}],z_k\rangle_{r,s}\nonumber
\end{eqnarray*}
for any $u_{\alpha},u_{\beta}\in \mathcal B_v(V^{r,s})$ and $z_k\in \mathfrak B_{r,s}$.
\end{proof}

\begin{theorem}\label{th:To groups}
Let $\mathcal S_1,\mathcal S_2\in\mathbb{S}^{M}_{r,s}$ and $\mathcal{L}(\mathcal S_1)$, $\mathcal{L}(\mathcal S_2)$ be the corresponding invariant integral structures. If there is an isomorphism
\begin{equation}\label{eq:lattice isom}
A\oplus C\colon\mathcal{L}(\mathcal S_1) \to \mathcal{L}(\mathcal S_2)
\end{equation}
with
$A\colon V^{r,s}\to V^{r,s}$ such that $A^\tau A=\Id_{V^{r,s}}$, then 
$\mathcal S_1$ and $\mathcal S_2$ are equivalent in the sense of Definition~\ref{def:group equiv}.
\end{theorem}

\begin{proof}
Let $ \mathcal S_1=\mathcal S_v$ be the isotropy subgroup of a unit vector $v\in V^{r,s}$ and $\mathcal S_2=\mathcal S_u$ the isotropy subgroup of a unit vector $u\in V^{r,s}$. Let
$$
\mathcal L(\mathcal S_v)=\spn_{\mathbb Z}\{\mathcal B_v(V^{r,s})\}\oplus\spn_{\mathbb Z}\{\mathfrak B_{r,s}\}=L_1\oplus\spn_{\mathbb Z}\{\mathfrak B_{r,s}\}
$$
$$
\mathcal L(\mathcal S_u)=\spn_{\mathbb Z}\{\mathcal B_u(V^{r,s})\}\oplus\spn_{\mathbb Z}\{\mathfrak B_{r,s}\}=L_2\oplus\spn_{\mathbb Z}\{\mathfrak B_{r,s}\}
$$
be the invariant integral structures generated by the groups $\mathcal S_v$ and $\mathcal S_u$. 
Since $A\oplus C$ is an isomorphism, we obtain $A(L_1)=L_2$. By noting 
that $A^{-1}(L_2)=A^{\tau}(L_2)=L_1$, we deduce that $A^{\tau}A(L_1)=L_1$.

We denote by the same letter  
$A\oplus C\in\Aut(\mathfrak{n}_{r,s})$ the automorphism of $\mathfrak n_{r,s}(V^{r,s})$ which restriction to $\mathcal L(\mathcal S_v)$ gives map~\eqref{eq:lattice isom}. 
The properties $A^{\tau}A=\Id_{V^{r,s}}$ and $A^{\tau}J_{C(z)}A=J_{z}$ imply $AJ_zx=J_{C(z)}Ax$ for $x\in L_1$ and $C\in \Orth(r,s)$, the latter one being an orthogonal transformation over $\mathbb Z$ as well. For  $v\in\mathfrak B_v(V^{r,s})$ we find a basis vector $u_j\in\mathfrak B_u(V^{r,s})$ such that 
$Av=u_j$. If there holds $Av=-u_j$, then the proof is similar. By renumbering the basis vectors $\{u_j\}$ we can assume that $A v=u$. We have for the stationary group of $Av$
\begin{eqnarray}\label{eq:0}
\mathcal S_{Av}&=&\{\tilde\sigma\in G\big(C(\mathfrak B_{r,s})\big)\mid\ J_{\tilde\sigma}Av=Av\}\nonumber
\\
&=&
\{\tilde\sigma\in G\big(C(\mathfrak B_{r,s})\big)\mid\ J_{\tilde\sigma}u=u\}=\mathcal S_{u}.
\end{eqnarray}
Since $\tilde\sigma=C(z_{i_1})\ldots C(z_{i_k})$, and $AJ_zx=J_{C(z)}Ax$, $x\in L_1$ we have
\begin{equation*}\label{eq1}
Av=J_{\tilde\sigma}Av=J_{C(z_{i_1})}\ldots J_{C(z_{i_k})}Av=AJ_{z_{i_1}}\ldots J_{z_{i_k}}v=AJ_{\sigma}v.
\end{equation*}
This implies
$v=J_{\sigma}v$ for any $\sigma\in G(\mathfrak B_{r,s})$.
Thus we conclude that if $\tilde\sigma\in \mathcal S_{Av}$, for $\tilde\sigma=C(z_{i_1})\ldots C(z_{i_k})\in G(C(\mathfrak B_{r,s}))$ then $\sigma=z_{i_1}\ldots z_{i_k}\in \mathcal S_{v}$. Thus the groups $\mathcal S_{Av}$ and $\mathcal S_{v}$ are equivalent. The equalities~\eqref{eq:0} show that $\mathcal S_2=\mathcal S_{u}=\mathcal S_{Av}$ and $\mathcal S_1=\mathcal S_{v}$ are equivalent.
\end{proof}

Table~\ref{tab:F-one} shows the classical groups $\mathbb A$ such that the map $A\oplus \Id$ with $A\in \mathbb A$ is the automorphism of $H$-type Lie algebras $\mathfrak n_{r,s}(V^{r,s})$, see also~\cite[Table 3]{FurMar21} for non-minimal admissible modules. The groups $\Sp(n), \Orth(n,\mathbb C),\U(n),\Orth^*(n)$ are subgroups of orthogonal transformations.

\begin{table}[h]
\centering
\caption{Groups $\mathbb A$}
\scalebox{0.57}{
\begin{tabular}{| c || c| c| c| c| c| c| c|	c| c|}
\hline
8 & $\GL(1,\mathbb{R})$ & & & & & & & & \\
\hline

7 &$\Orth(1,1,\mathbb R)$& $\U(1,1)$ & $\Sp(1,1)$ & $\Sp(1)\times \Sp(1)$& & & & & \\
\hline

6 & $ \Orth(2,\mathbb C)$ & $\Orth^*(2)$ & $\GL(1, \mathbb H)$  & $\Sp(1)$ & & & & & \\
\hline

5 &$\Orth^*(4)$  &$\Orth^*(2)\times \Orth^*(2)$  &$\Orth^*(2)$&  $\U(1)$& & & & & \\
\hline

4 & $\GL(1, \mathbb{H})$ &  $\Orth^*(2)$ & $\Orth(1, \mathbb{C})$ & $\Orth(1, \mathbb{R})$ & $\GL(1, \mathbb{R})$ & & & & \\
\hline

3 &$\Sp(1,1)$ & $\U(1,1)$ &$\Orth(1,1,\mathbb R)$ & $\Orth(1,\mathbb R)\times \Orth(1,\mathbb R)$  & $\Orth(1,1,\mathbb R)$& $\U(1,1)$& $\Sp(1,1)$ & $\Sp(1)\times \Sp(1)$& \\
\hline

2 &$\Sp(2,\mathbb C)$ & $\Sp(2,\mathbb R)$ & $\GL(2,\mathbb R)$ & $\Orth(2,\mathbb R)$&$\Orth(2,\mathbb C)$ &$\Orth^*(2)$ & $\GL(1, \mathbb H)$ &  $\Sp(1)$&  \\
\hline

1 &$\Sp(2,\mathbb R)$  &$\Sp(2,\mathbb{R})\times \Sp(2,\mathbb{R})$ & $\Sp(4,\mathbb R)$&$\U(2)$&$\Orth^*(4)$  &$\Orth^*(2)\times\Orth^*(2)$ & $\Orth^*(2)$ &$\U(1)$ & \\
\hline

0 &  &$\Sp(2,\mathbb R)$ &$\Sp(2,\mathbb C)$ & $\Sp(1)$ &  $\GL(1,\mathbb{H})$ &  $\Orth^*(2)$ & $\Orth(1,\mathbb{C})$& $\Orth(1,\mathbb{R})$ & $\GL(1,\mathbb{R})$ \\ 
\hline
\hline

 & 0 & 1 &2 & 3 & 4 & 5 & 6 & 7 & 8 \\
\hline
\end{tabular}
}
\label{tab:F-one}
\end{table}

\begin{theorem}\label{th:orthogonal cases}
Let $(r,s)$ be such that the group $\mathbb A$ in Table~\ref{tab:F-one} is a subgroup of orthogonal transformations. The groups $\mathcal{S}_1,\mathcal{S}_2\in\mathbb S_{r,s}^M$
are equivalent in sense of Definition~\eqref{equiv relation}, if and only if the corresponding invariant integral structures $\mathcal L(\mathcal{S}_1)$ and $\mathcal L(\mathcal{S}_2)$ are isomorphic.
\end{theorem}
\begin{proof}
If $(r,s)$ as in the statement of Theorem~\ref{th:orthogonal cases}, then for an automorphism $\tilde A\oplus\Id$ of $\mathfrak n_{r,s}(V^{r,s})$ we have $\tilde A^{\tau}\tilde A=\Id_{V^{r,s}}$. It implies that the general automorphisms  $A\oplus C$ of $\mathfrak n_{r,s}(V^{r,s})$ also satisfies $A^{\tau}A=\Id_{V^{r,s}}$, see~\cite[Section 3.2]{FurMar21}. 

Thus if the invariant integral structures $\mathcal L(\mathcal{S}_1)$ and $\mathcal L(\mathcal{S}_2)$ are isomorphic, then they will be isomorphic under a map $A\oplus C$ with $A^{\tau}A=\Id_{V^{r,s}}$. It implies that the groups $\mathcal{S}_1$ and $\mathcal{S}_2$ are equivalent by Theorem~\ref{th:To groups}. 

Conversely, if we assume now that the groups $\mathcal{S}_1$ and $\mathcal{S}_2$ are equivalent, then by Theorem~\ref{th:From group} the corresponding invariant integral structures will be isomorphic.
\end{proof}

%
%

Note that in the proof of Theorem~\ref{th:To groups} the crucial assumption was $A^{\tau} A=\Id_{V^{r,s}}$. The following theorem shows that it is enough to find a subset $E\subset V^{r,s}$, which is invariant under the action $A^{\tau} A$. It allows to prove the general theorem.

\begin{theorem}\label{one general criterion-1}
The groups $\mathcal{S}_1,\mathcal{S}_2\in\mathbb S_{r,s}^M$
are equivalent in sense of Definition~\ref{equiv relation}, if and only if the corresponding invariant integral structures $\mathcal L(\mathcal{S}_1)$ and $\mathcal L(\mathcal{S}_2)$ are isomorphic.
\end{theorem}

\begin{proof} 
If $\mathcal S_1$ is equivalent to $\mathcal S_2$, then  the corresponding invariant integral structures 
$\mathcal{L}(\mathcal S_1)$ and $\mathcal{L}(\mathcal S_2)$ 
are isomorphic by Theorem~\ref{th:From group}.

Suppose that invariant integral structures 
$\mathcal{L}(\mathcal S_1)$ and $\mathcal{L}(\mathcal S_2)$ 
are isomorphic. By contrary we assume that the groups $\mathcal S_1=\mathcal{S}(PI_1)\in \mathbb{S}^M_{r,s}$ and $\mathcal S_2=\mathcal{S}(PI_2)\in \mathbb{S}^M_{r,s}$ are not equivalent. Then there are $q_1\in PI_1$ 
and $q_2\in PI_2$
such that
$q_1\cdot q_2 = - q_2\cdot q_1$. For if $q_1\cdot p = p\cdot q_1$ for all $p\in PI_2$, then $q_1\in PI_2$, which would contradict to the maximality of $\mathcal{S}(PI_2)$.

Without loss of generality we can assume that the groups $\mathcal{S}(PI_1)$ and $\mathcal{S}(PI_2)$ are written in the standard form as in Example~\ref{ex:standard group}. Let $k$ be a maximal number of type $T_1$ involutions $p_j$ satisfying $p_j\in PI_1\cap PI_2$, $j=1,\ldots k$. Note that $k<\ell(r,s)$ since 
$\mathcal{S}(PI_1)$ and $\mathcal{S}(PI_1)$ are not equivalent.

Let $\mathfrak n_{r,s}(V^{r,s})$ be a pseudo $H$-type Lie algebra and 
$$
E=\{x\in V^{r,s}\mid\ J_{p_j}x=x,\ p_j\in PI_1\cap PI_2,\ j=1,\ldots, k\}.
$$
Since $q_1p_{j}=p_{j}q_1$ and $q_2p_{j}=p_{j}q_2$ the subspace $E\subset V^{r,s}$ is invariant under the action of both $J_{q_1}$ and $J_{q_2}$.

For an isomorphism
$A\oplus C\colon \mathcal{L}(\mathcal S_1)\to \mathcal{L}(\mathcal S_2)$
we set 
$$
F=A(E)=\{Ax\in V^{r,s}\mid\ AJ_{p_j}x=J_{C(p_j)}Ax=Ax,\ p_j\in PI_1\cap PI_2,\ j=1,\ldots, k\}.
$$
The map $C$, extended to the Clifford algebra $\Cl_{r,s}$, satisfies $C(p_j)C(q_1)=C(q_1)C(p_j)$ and $C(p_j)C(q_2)=C(q_2)C(p_j)$, $j=1,\ldots k$.  
These imply that $A(E)$ is invariant under the action of $J_{C(q_1)}$ and $J_{C(q_2)}$ by the same arguments as for $E$.
Thus the direct sum
\begin{equation}\label{eq:orthog sum}
F=F_{+}\oplus F_{-},\qquad F_{+}=\{y\in F\mid\ J_{C(q_2)}y=y\},\quad F_{-}=\{y\in F\mid\ J_{C(q_2)}y=-y\}
\end{equation}
is the orthogonal sum of non-trivial vector spaces.

Let $x\in E$ and put $Ax=y_+(x)+y_-(x)$, where $y_+(x)\in F_{+}$ and $y_-(x)\in F_{-}$. We have for the type $T_1$ involution $q_1\in PI_1$ that
\[
J_{C(q_1)}Ax=J_{C(q_1)}(y_+(x)+y_-(x))=J_{C(q_1)}y_+(x)+J_{C(q_1)}y_-(x).
\]
Since $C(q_1)C(q_2)=-C(q_2)C(q_1)$ we obtain 
$$
J_{C(q_1)}\colon F_+\to F_-,\quad\text{and}\quad J_{C(q_1)}y_+(x)\in F_-,\quad J_{C(q_1)}y_-(x)\in F_+,
$$
and therefore
$
y_+(x)=J_{C(q_1)}y_-(x)$ and $y_-(x)=J_{C(q_1)}y_+(x)
$
by the uniqueness of the decomposition into a direct sum of vector spaces. 
We conclude
\begin{equation}\label{eq:JCp}
J_{C(q_1)}Ax=y_+(x) +J_{C(q_1)}y_+(x).
\end{equation}

Since $p_j$ are $T_1$-type involutions, we obtain 
$$
AJ_{p_j}=J_{C(p_{j})}A,\ \ J^{\tau}_{p_j}=J_{p_j},\ \ J^{\tau}_{C(p_j)}=J_{C(p_j)},\ \ A^{\tau}J_{C(p_j)}=J_{p_j}A^{\tau},\ \ j=1,\ldots, k.
$$ 
It implies $A^{\tau}A(E)=E$. Let $\{v_i\}$ be an orthonormal basis of the space $E$, which is a
part of the invariant basis for $V^{r,s}$ defined by the $\mathcal S_1=\mathcal{S}(PI_1)$. 
The matrix components $a_{ij}$ of the operator $A^{\tau}A\colon E\to E$ with respect to the basis $\{v_i\}$ have the form 
\begin{align*}
a_{ij}&=\langle A^{\tau} Av_i,v_j\rangle_{V^{r,s}}=
\langle Av_i, A v_j\rangle_{V^{r,s}}
=
\langle C(q_1), C(q_1)\rangle_{r,s}\langle Av_i, A v_j\rangle_{V^{r,s}}
\\
&=
\langle J_{C(q_1)}Av_i, J_{C(q_1)}A v_j\rangle_{V^{r,s}}
\\
&=
\langle y_+(v_i)+J_{C(q_1)}(y_+(v_i)), y_+(v_j)+J_{C(q_1)}y_+(v_j)\rangle_{V^{r,s}}
\\
&=
2\langle y_+(v_i), y_+(v_j)\rangle_{V^{r,s}},
\end{align*}
where we used~\eqref{eq:orthog sum} and~\eqref{eq:JCp}. Hence the non-vanishing components of the matrix $A^{\tau}A$ restricted to $E$ are always even numbers, so that $\det\,A^{\tau}A=2^{\dim E}\cdot k$, $k\in\mathbb{Z}$. 

Let us look on the structure of the map $A^{\tau}A$ acting on the entire minimal admissible module $V^{r,s}$. The space $V^{r,s}$ is an orthogonal sum of subspaces $W_{i}$
of the form $J_{\varkappa_{i}}(E)=W_{i}$ with $\varkappa_i\in\Sigma$ from Theorem~\ref{basis expression}. Let $x\in V^{r,s}$. We write
$$
x=x_E+x_{1}+\ldots+x_{m},\quad x_i\in W_i,
$$
and $A^{\tau}A\vert_{W_i}$ for the restriction of the map $A^{\tau}A$ on the set $W_i$. For any $x_i\in W_i$ there is $y_E\in E$ such that $J_{\varkappa_i}(y_E)=x_i$. Choose $i\in \{1,\ldots,m\}$ and assume that $\varkappa_i$ is a product of an even number of the basis vectors. Then we obtain
\begin{eqnarray*}
A^{\tau}A\vert_{W_i}(x_i)=A^{\tau}A\vert_{W_i}J_{\varkappa_i}(y_E)=J_{\varkappa_i}A^{\tau}A\vert_{E}(y_E)
=J_{\varkappa_i}A^{\tau}A\vert_{E}J^{-1}_{\varkappa_i}(x_i).
\end{eqnarray*}
Thus in this case $A^{\tau}A\vert_{W_i}=J_{\varkappa_i}A^{\tau}A\vert_{E}J^{-1}_{\varkappa_i}$.
If $\varkappa_i$ is a product of an odd number of the basis vectors, then the isomorphism condition 
\eqref{automorphism condition:eq} for $A\oplus C$ implies $A^{\tau}AJ_{\varkappa_i}A^{\tau}A=J_{\varkappa_i}$. This leads to
\begin{eqnarray*}
A^{\tau}A\vert_{W_i}(x_i)=A^{\tau}A\vert_{W_i}J_{\varkappa_i}(y_E)=J_{\varkappa_i}(A^{\tau}A\vert_{E})^{-1}(y_E)
=J_{\varkappa_i}(A^{\tau}A\vert_{E})^{-1}J^{-1}_{\varkappa_i}(x_i)
\end{eqnarray*}
and therefore $A^{\tau}A\vert_{W_i}=J_{\varkappa_i}(A^{\tau}A\vert_{E})^{-1}J^{-1}_{\varkappa_i}$.

The elements of the matrix $A$ of the isomorphism of the lattices 
$\mathcal{L}(\mathcal S_{1})$ and $\mathcal{L}(\mathcal S_{2})$ are all integers.
If all $W_i$ are images $J_{\varkappa_i}(E)$ with $\varkappa_i$ being a product of an even number of the basis vectors, then it is clear that $A^{\tau}A\notin \SL(\dim V^{r,s},\mathbb{Z})$. If there is $W_i=J_{\varkappa_i}(E)$ with $\varkappa_i$ being a product of an odd number of the basis vectors, then $\det A^{\tau}A\vert_{W_i}=\frac{1}{2^{\dim E}\cdot k}$ which contradicts to the fact that the terms of the matrix $A^{\tau}A$ are all integers.  
\end{proof}

\section{Final conclusions}
We summarize the results obtained in the article. We have provided a classification of invariant orthonormal integral structures on the pseudo $H$‑type Lie algebras $\mathfrak n_{r,s}$ whose complements to the centres are isometric to a minimal admissible module of the Clifford algebras $\Cl(\mathbb R^{r,s})$. The classification is complete for $1 \leq r+s \leq 16$. This classification can be extended to any values of the parameters $r,s \in \mathbb N$ by using the Atiyah–Bott type periodicity, although for $r+s > 16$ the classification may no longer be complete.

The core of the classification is the collection of all non-equivalent sets of positive involutions $PI_{r,0}$ for $r \in {1,\ldots,16}$. The invariants of the classification are the maximal number of involutions $\ell(r,s)$ and the type ($T1$ or $T2$) of the set of involutions $PI_{r,0}$ containing all $r$ basis vectors. The maximal sets of involutions $PI_{r,0}$ containing fewer than $r$ basis vectors belong to the sets $PI_{t,0}$ with $t < r$. For $PI_{r,0}$ with $r \geq 10$, a new parameter enters the classification, namely the number of connected components in the set of involutions; see~\eqref{sum1} and \eqref{sum2}. It is not known to the author at the present moment what kind of new phenomenon may influence the structure of $PI_{r,0}$ for $r > 16$.
Once the description of $PI_{r,0}$ for $r > 16$ is obtained, the classification can be extended to $PI_{r,s}$ with $r+s > 16$. The paper contains several tables in which all non-equivalent sets of involutions are listed. In forthcoming papers, we aim to obtain a complete classification of orthonormal invariant integral structures for the entire range of parameters $r,s \in \mathbb N$.

\bibliographystyle{alpha}
\bibliography{Bibliography.bib}
\end{document}